\documentclass[a4paper,12pt]{article}

\usepackage[usenames]{xcolor} 
\usepackage{amssymb} 
\usepackage{mathtools} 
\usepackage{amsfonts}
\usepackage{amsthm}
\usepackage[utf8]{inputenc} 
\usepackage[british]{babel}
\usepackage{array} 
\usepackage{multirow} 
\usepackage{float} 
\usepackage[margin=1.3in]{geometry}
\usepackage{multirow}
\usepackage{enumerate}
\usepackage{parskip}
\usepackage[all,cmtip]{xy}
\usepackage{tikz}
\usetikzlibrary{matrix,arrows.meta,positioning}
\usepackage{arydshln}
\usepackage{comment}
\usepackage{setspace}

\DeclareMathOperator{\Sel}{Sel}
\DeclareMathOperator{\Gal}{Gal}
\DeclareMathOperator{\cor}{cor}
\DeclareMathOperator{\res}{res}

\renewcommand{\Im}{\textrm{Im}}
\newcommand{\Frob}{\textrm{Frob}}
\newcommand{\Fitt}{\textrm{Fitt}}
\newcommand{\ord}{\textrm{ord}}
\renewcommand{\char}{\textrm{char}}
\DeclareMathOperator{\Hom}{Hom}

\DeclareMathOperator{\Aut}{Aut}
\DeclareMathOperator{\KS}{KS}
\DeclareMathOperator{\ES}{ES}
\newcommand{\loc}{\textrm{loc}}
\newcommand{\m}{\mathfrak{m}}
\newcommand{\p}{\mathfrak{p}}
\newcommand{\du}{^\vee}
\newcommand{\coker}{\textrm{coker}}
\newcommand{\length}{\textrm{length}}

\newcommand{\rank}{\textrm{rank}}

\newcommand{\Ind}{\textrm{Ind}}
\DeclareMathOperator{\cond}{cond}
\newcommand{\Q}{\mathbb Q}
\newcommand{\R}{\mathbb R}
\newcommand{\C}{\mathbb C}
\newcommand{\Z}{\mathbb Z}
\newcommand{\F}{\mathbb F}
\newcommand{\N}{\mathbb N}
\newcommand{\FF}{\mathcal F}

\newcommand{\Fcl}{{\mathcal{F}_{\textrm{cl}}}}
\newcommand{\Fcan}{{\mathcal{F}^{\textrm{can}}}}
\newcommand{\FLambda}{{\mathcal{F}_\Lambda}}
\newcommand{\FBK}{{\mathcal{F}_{\textrm{BK}}}}

\newcommand{\PP}{\mathcal P}
\newcommand{\NN}{\mathcal N}
\newcommand{\DD}{\mathcal D}
\newcommand{\GG}{\mathcal G}

\newcommand{\OO}{\mathcal O}
\newcommand{\LL}{\mathcal L}
\newcommand{\II}{\mathcal I}

\newcommand{\mupi}{\mu_{p^\infty}}

\newcommand{\tr}{\textrm{tr}}

\newcommand{\f}{\textrm{f}}
\newcommand{\s}{\textrm{s}}
\newcommand{\Tr}{\textrm{Tr}}
\newcommand{\fs}{\textrm{fs}}

\renewcommand{\f}{\textrm{f}}

\newcommand{\GL}{\textrm{GL}}

\newcommand{\tors}{\textrm{tors}}
\newcommand{\per}{\textrm{per}}
\newcommand{\crys}{\textrm{crys}}
\newcommand{\dR}{\textrm{dR}}

\newcommand{\chibar}{{\overline{\chi}}}
\newcommand{\psibar}{{\overline{\psi}}}

\renewcommand{\Im}{\textrm{Im}}

    \DeclareFontFamily{U}{wncy}{}
    \DeclareFontShape{U}{wncy}{m}{n}{<->wncyr10}{}
    \DeclareSymbolFont{mcy}{U}{wncy}{m}{n}
\DeclareMathSymbol{\Sha}{\mathord}{mcy}{"58}

\theoremstyle{definition}
\newtheorem{theorem}{Theorem}[section]
\newtheorem{theorem*}{Theorem}[section]
\newtheorem{definition}[theorem]{Definition}
\newtheorem{example}[theorem]{Example}
\newtheorem{remark}[theorem]{Remark}
\newtheorem{remark*}[theorem*]{Remark}
\newtheorem{proposition}[theorem]{Proposition}
\newtheorem{corollary}[theorem]{Corollary}
\newtheorem{corollary*}[theorem*]{Corollary}
\newtheorem{lemma}[theorem]{Lemma}
\newtheorem{conjecture}[theorem]{Conjecture}

\newtheorem{assumption}[theorem]{Assumption}

\makeatletter
\def\namedlabel#1#2{\begingroup
    #2%
    \def\@currentlabel{#2}%
    \phantomsection\label{#1}\endgroup
}
\makeatother

\usepackage[backend=biber, style=alphabetic]{biblatex}

\usepackage{csquotes}

\usepackage[pdfencoding=auto,hidelinks]{hyperref} 

\begin{document}

\title{Kolyvagin systems of rank 0 and the structure of the Selmer group of elliptic curves over abelian extensions}
\author{Alberto Angurel}
\date{}
\maketitle

\begin{abstract}
With the motivation to study the Selmer group of an elliptic curve, we improve the theory of Kolyvagin systems to describe the Fitting ideals of a Selmer group in the core rank zero situation. By relaxing a Selmer structure of rank zero at certain prime, we can construct an auxiliary Kolyvagin system whose localisation determines most of the Fitting ideals of the Selmer group, and all of them when the Galois representation is not self-dual. With this new theory, one can describe, in terms of the modular symbols, the Galois structure of the Selmer group of an elliptic curve $E/\mathbb{Q}$ over a finite abelian extension whose degree is coprime to $p$.
\end{abstract}

\tableofcontents

\section{Introduction}
\label{sec:intro}

The theory of Euler systems was introduced by V.A. Kolyvagin in \cite{Kolyvagin} in order to bound the rank and the order of the Tate-Shafarevich group of certain elliptic curves. In \cite{Rubin}, K.~Rubin generalised the theory to bound certain Selmer groups for arbitrary Galois representations. Those bounds are given by a collection of cohomology classes, called Kolyvagin's derivative classes, and can be obtained from every Euler system. In \cite{MazurRubin}, B.~Mazur and K.~Rubin axiomatized the theory of these derivative classes to introduce the notion of Kolyvagin systems. The theory can be applied to study the Selmer groups associated with all geometric Galois representations satisfying certain mild assumptions.

Important applications of this theory are the computations of class groups of number fields and Selmer groups of elliptic curves, which encode their rank and the structure of the Tate-Shafarevich group. However, in the elliptic curve case, the results in \cite{MazurRubin} only described the Selmer group 'restricted at $p$', which differs from the classical Selmer group in the local condition at the prime $p$. The main innovation of this article is the description of the structure of the classical Selmer group of an elliptic curve defined over $\Q$ in terms of its modular symbols.

From the modular symbols of an elliptic curve $E/\Q$, one can define, for many square-free integers $n\in \Z$, quantities $\delta_n\in \Z_p$, called Kurihara numbers after \cite{Kur2012}. They are only defined modulo some power of $p$ depending on $n$. Following \cite{Kim23}, if we group the quantities $\delta_n$ by the number of prime divisors of $n$, we can find, under some technical assumptions (see Theorem \ref{th:Kim_Selmer_str}), the rank $r$ of $E(\Q)$ and certain quantities $a_i$ such that 
\[\Sha(E/\Q)[p^\infty]=\biggl(\Z/(p)^{a_r}\biggr)^2\times \biggl(\Z/(p)^{a_{r+2}}\biggr)^2\times \cdots \times \biggl(\Z/(p)^{a_{s-2}}\biggr)^2\]
for some natural number $s$. We will generalise this result to study the Selmer groups $\Sel(K,E[p^\infty])$ for abelian extensions $K/\Q$. In order to do that, we define the twisted $\delta_{n,\chi}$ for all characters $\chi$ of $\Gal(K/\Q)$ and use these new quantities to characterise the Galois structure of the Mordell-Weil group $E(K)$ and the Tate-Shafarevich group $\Sha(E/K)$.

In order to introduce the main results in detail, we have to introduce some notation. Let $R$ be a complete, noetherian, principal, local ring of residue characteristic $p$ with maximal ideal $\m$ and let $T$ be a free $R$-module of finite rank endowed with a continuous action of the absolute Galois group $G_\Q$ of the rationals. A Selmer structure $\FF$ is a collection of subgroups $H^1_\FF(\Q_\ell,T)$ of the local Galois cohomology groups $H^1(\Q_\ell,T)$, where $\ell$ runs through all the primes. Those subgroups are called local conditions. Its associated Selmer group, denoted by $H^1_\FF(\Q,T)$, will be the elements in the global cohomology group $H^1(\Q,T)$ satisfying that their localisation for every prime belongs to the specified local condition.

The above Selmer groups are a generalisation of the classical Selmer groups that are defined, for instance, for the Tate module of an elliptic curve. In this case, the local conditions are the images of the local Kummer maps and the rank of the Selmer group is an upper bound of the algebraic rank of the elliptic curve, this upper bound being exactly the rank when the $p$-primary part of the Tate-Shafarevich group is finite.

Associated with a Selmer structure, there  is a quantity called core rank and denoted by $\chi(T,\FF)$ that, roughly speaking, measures the difference of the ranks between the Selmer groups of $T$ and its Cartier dual $T^*$. The theory in \cite{MazurRubin} can only be applied when $\chi(T,\FF)=1$. In this case, the module of Kolyvagin systems $\KS(T,\FF)$ is free of rank one over $R$, so a primitive Kolyvagin system, i.e., a generator of $\KS(T,\FF)$, describes the structure of the Selmer group.

The theory of Kolyvagin systems is based on the existence of a set of Kolyvagin primes $\PP$ satisfying certain properties. A Kolyvagin system is a collection of elements $\kappa_n\in H^1_{\FF(n)}(\Q,T)$ with $n$ running through the square-free products of primes in $\PP$, satisfying certain relations between $\kappa_n$ and $\kappa_{n\ell}$. Here $\FF(n)$ is the Selmer structure obtained by modifying conveniently the local conditions at the prime divisors of $n$. From the elements $\kappa_n$ in which $n$ is the product of exactly $i$ primes in $\PP$, one can define some ideals $\Theta_i$ in $R$, which coincide with the $i^{th}$ Fitting ideals of $H^1_\FF(\Q,T)$.

When $\chi(T,\FF)>1$, the module of Kolyvagin systems is not finitely generated, so the notion of primitivity cannot be defined. The theory was generalised in \cite{MazurRubin16} (see also \cite{BurnsSakamotoSano2}) by redefining the Kolyvagin systems as collections of classes in the exterior powers of certain cohomology groups. With this new definition, $\KS(T,\FF)$ is free of rank one over $R$ and a primitive Kolyvagin system describes the structure of the Selmer group similarly.

The case $\chi(T,\FF)=0$ remained unsolved by these methods. The main reason was the lack of nonzero Kolyvagin systems, as shown in \cite{MazurRubin}. The idea of R.~Sakamoto in \cite{Sakamoto_rk0} was to consider, for every Kolyvagin prime $\ell$, the Selmer structure $\FF^\ell$ relaxed at $\ell$, i.e., the structure obtained by modifying the local condition at $\ell$ to $H^1(\Q_\ell, T)$. These Selmer structures are of core rank one, so they admit non-trivial Kolyvagin systems.

R. Sakamoto defined the Kolyvagin systems of rank $0$ for $\FF$ as collections of (rank one) Kolyvagin systems of $\FF^\ell$ with $\ell$ running through the Kolyvagin primes, satisfying some relations among them. He proved that the module of rank zero Kolyvagin systems is free of rank one and that the primitive Kolyvagin systems carry some information about the Selmer group. He managed to construct the ideal $\Theta_i$ similarly and proved it was contained in the $i^{th}$ Fitting ideal of $H^1_\FF(\Q,T)$. For $i=0$, the ideal $\Theta_0$ is exactly the $0^{\textrm{th}}$ Fitting ideal of $H^1_\FF(\Q,T)$. This result is not powerful enough to determine the full structure of the Selmer group but it establishes its rank and cardinality.


In this article, we will find that it is only necessary to consider the relaxed Selmer structure $\FF^\ell$ for a unique prime $\ell$ in order to study the Selmer group of $\FF$. We also do not impose $\ell$ to be a Kolyvagin prime. The only condition we need $\ell$ to satisfy is that 
\[\frac{H^1(\Q_\ell,T/\m^kT)}{H^1_{\FF}(\Q_\ell,T/\m^kT)}\cong R/\m^k\ \ \forall k\geq 0.\]

Since $T,\FF^\ell)=1$, we can find a primitive Kolyvagin system for this Selmer structure and define
\[\delta_n:=\loc_\ell^s(\kappa_n),\]
where $\loc_\ell^s:\ H^1(\Q,T)\to \frac{H^1(\Q_\ell,T/\m^kT)}{H^1_{\FF}(\Q_\ell,T/\m^kT)}\cong R/\m^k$ is the localisation map and $k$ is some index depending on $n$.

Define $\Theta_i$ as the ideal of $R$ generated by all the $\delta_n$ in which $n$ is the product of exactly $i$ different primes. In case $\ell\in \PP$, the ideals $\Theta_i$ coincide with those defined in \cite{Sakamoto_rk0}. These ideals are contained in the respective Fitting ideals of the Selmer group and we can prove the equality holds in enough cases to determine the full structure of the Selmer group as an $R$-module.

\begin{theorem*} (See Theorem \ref{th:kur_par})
Under the assumptions \ref{Hirred}-\ref{Hprimes} in \textsection\ref{sec:global}, for every $i\in \Z_{\geq 0}$ we have that 
\[\Theta_i\subset \Fitt_i^R\biggl(H^1_{\FF^*}(\Q,T^*)\du\biggr).\]
where $\Fitt^i_R$ denotes the $i^{\textrm{th}}$ Fitting ideal. The equality holds when $i$ satisfies one of the following conditions:
\begin{itemize}
\item $i=\rank_R\biggl(H^1_{\FF^*}(\Q,T^*)\du\biggr)$,
\item $\Theta_{i-1}\subsetneq \Fitt_{i-1}^R\biggl(H^1_{\FF^*}(\Q,T^*)\du\biggr)$.
\end{itemize}
\label{th:kur_par_intro}
\end{theorem*}

In particular, Theorem \ref{th:kur_par_intro} guarantees that the equality between the ideal $\Theta_i$ and the $i^{\textrm{th}}$ Fitting ideal of the Selmer group cannot fail for two consecutive indices $i$. Using the structure theorem for finitely generated modules over principal ideal domains, there is enough information to compute all Fitting ideals of $H^1_{\FF^*}(\Q,T^*)^\vee$ and, therefore, it is determined up to isomorphism by the ideals $\Theta_i$.

The quantities $\delta_n$ have an explicit and effectively computable formula in the case of the classical Selmer group of an elliptic curve. Assume $E$ is an elliptic curve defined over $\Q$ for which the absolute Galois group $G_\Q$ acts surjectively on the Tate module $T_pE$. In this case, M. Kurihara gave in \cite{Kur2012} a description of the structure of the Selmer group in terms of the following quantities, later known as Kurihara numbers:

\[
\widetilde \delta_n=\sum_{a\in (\Z/n)^*} \left[\frac{a}{n}\right]^+ \prod_{\ell\mid n}\log_{\eta_\ell}(a) \in \Z/p^{k_n}.
\]
Here $\left[\frac{a}{n}\right]^+$ is the real part of the modular symbol, normalised by the Néron period, of the modular form associated with $E$ and $k_n$ is some natural number depending on $n$. Similarly, one can define the ideals $\widetilde \Theta_i$ as the ones generated by the $\widetilde\delta_n$'s when $n$ runs through the square-free products of exactly $i$ primes in $\PP$.

Under some assumptions, like the vanishing of the Iwasawa $\mu$-invariant or the non-degeneracy of the $p$-adic height pairing, M. Kurihara proved that $\widetilde \Theta_i$ coincides with the $i^{th}$ Fitting ideal of $\Sel(\Q,E[p^\infty])\du$ when $i$ has the same parity as the algebraic rank of the elliptic curve and that $\widetilde \Theta_i=0$ otherwise. C.~H.~Kim improved the result by only assuming weaker hypotheses.
\begin{theorem*}(\cite[Theorems 1.9 and 1.11]{Kim23})
    Let $E/\Q$ be an elliptic curve defined over $\Q$ and let $p\geq 5$ be a prime number such that:
    \begin{itemize}
    \item $G_\Q$ acts surjectively on $T_pE$,
    \item The Tamagawa numbers of $E$ are prime to $p$,
    \item $E(\Q_p)$ contains no $p$-torsion,
    \item The Manin constant $c_1(E)$ is prime to $p$,
    \item The Iwasawa main conjecture holds.
    \end{itemize}
Let $n_i$ be the exponent $\widetilde \Theta_i=(p)^{n_i}$ and denote $a_i=\frac{n_i-n_{i+2}}{2}$.  Then there exist (minimal) indices $r$ and $s$ such that $\widetilde \Theta_r\neq 0$ and $\widetilde \Theta_s=\Z_p$ and the Selmer group can be described up to isomorphism as
\[\Sel(\Q,E[p^\infty])\cong \Bigl({\Q_p}/{\Z_p}\Bigr)^r \times \Bigl(\Z/(p)^{a_r}\Bigr)^2\times \Bigl(\Z/(p)^{a_{r+2}}\Bigr)^2\times \cdots\times \Bigl(Z/(p)^{a_{s-2}}\Bigr)^2. \]
\label{th:Kim_Selmer_str}
\end{theorem*}

The proof of Theorem \ref{th:Kim_Selmer_str} in \cite{Kim23} is based on the existence of Cassels-Tate pairing, which is a bilinear non-degenerate pairing defined on the torsion part of the Selmer group. 

Theorem \ref{th:kur_par_intro} provides a different proof for this result since $\widetilde \delta_n=\loc_p^s(\kappa_n)$ when $\kappa$ represents the Kolyvagin system obtained by deriving Kato's Euler system. When $i$ has a different parity than the rank of $E$, the functional equation applied to the Kurihara numbers implies that $\widetilde \Theta_i=0$, so the equality in Theorem \ref{th:kur_par_intro} holds for the other $i$'s and hence Theorem \ref{th:Kim_Selmer_str} follows as a consequence of the structure theorem of finitely generated $\Z_p$-modules. This argument will be presented in more generality to study the Galois structure of $\Sel(K,E[p^\infty])$ when $K$ is a finite abelian extension of $\Q$.

Therefore, one cannot expect the equality in Theorem \ref{th:kur_par_intro} to hold for every $i$, as we have seen a counterexample in $\Sel(\Q,E[p^\infty])$. The impediment appears when $T$ is self dual. In this case, there exists generalised Cassels-Tate pairings that force the Selmer groups to have specific forms, so there are obstructions to this method in specific values of $i$. However, if we impose a non-self-duality assumption on $T$, the equality holds for all $i$.
\begin{theorem*} (See Theorem \ref{th:kur})
Assume, in addition to \ref{Hirred}-\ref{Hprimes} in \textsection\ref{sec:global}, that there are no isomorphic subquotients of $T$ and $T^*$ and that the map $G_\Q\to \Aut(T)$ (given by the Galois action) has a sufficiently large image. Then 
\[\Theta_i=\Fitt^R_i\biggl(H^1_{\FF}(\Q,T^*)\du\biggr)\ \ \forall i\in \Z_{\geq 0}.\]
\label{th:kur_intro}
\end{theorem*}

\begin{remark*}
While working on this project, the author was informed by R. Sakamoto that he had proven in his unpublished work \cite{Sakamoto24} the above result under the rather strong assumption $H^1_{\FF}(\Q,T)=0$. In this sense, Theorem \ref{th:kur_intro} is a generalization of R.~Sakamoto's result. 
\end{remark*}

With this result, we can generalise Theorem \ref{th:Kim_Selmer_str} to describe $\Sel(K,E[p^\infty])$ when $K/\Q$ is an abelian extension of degree prime to $p$, unramified at $p$ and at every bad prime of $E$. Recall that we need $E$ to be defined over $\Q$. We aim to define the structure of the Selmer group not only as a $\Z_p$-module but as a $\Z_p[\Gal(K/\Q)]$-module. In order to do that, we can use Shapiro's lemma to split the Selmer group into character parts:
$$\Sel(\Q,E[p^\infty])\otimes \OO_d=\bigoplus_{\chi} \Sel\biggl(\Q,E[p^\infty]\otimes \OO_d(\chi)\biggr),$$
where $\chi$ runs through the characters of $\Gal(K/\Q)$ and $\OO_d(\chi)$ is $\Z_p[\mu_d]$ endowed with an action of $G_\Q$ given by $\sigma\cdot x:=\chi(\sigma)x$.

In order to compute the structure of $\Sel\Bigl(\Q,E[p^\infty]\otimes \OO_d(\chi)\Bigr)$, we define the $\chi$-twisted Kurihara numbers as
\[\delta_{n,\chi}=\sum_{a\in (\Z/cn\Z)^*} \chibar(a) \left[\frac{a}{cn}\right]^{\chi(-1)} \prod_{\ell\mid n} \log_{\eta_\ell}(a)\in \OO_d/p^{k_n}.\]
where $c$ is the conductor of $\chi$. Similarly, the ideal $\widetilde \Theta_{i,\chi}$ is the ideal generated by all $\delta_{n,\chi}$ where $n$ is the square-free product of $i$ primes in $\PP$.

\begin{theorem*}(Theorem \ref{th:EK_str})
    Let $E/\Q$ be an elliptic curve defined over $\Q$ and let $\chi$ be a Dirichlet character of order prime to $p$ and conductor $c$. Denote by $K$ the fixed field of $\chi$. Assume the following:
    \begin{itemize}
    \item $G_\Q$ acts surjectively on $T_pE$,
    \item The Manin constant $c_1(E)$ is prime to $p$,
    \item $\textrm{gcd}(c,Np)=1$,
    \item The Tamagawa numbers of $E$ over $K$ are prime to $p$,
    \item $E(K_\p)$ contains no $p$-torsion for every prime $\p$ above $p$,
    \item The Iwasawa main conjecture for $f_\chi$ holds, where $f_\chi$ is the twist by $\chi$ of the associated modular form to $E$.
    \end{itemize}
Let $n_i$ be the exponent $\widetilde \Theta_{i,\chi}=(p)^{n_i}$.  Then there exist (minimal) indices $r$ and $s$ such that $\widetilde \Theta_r\neq 0$ and $\widetilde \Theta_s=\Z_p$: Also, denote $a_i=\frac{n_i-n_{i+2}}{2}$.
\begin{itemize}
\item If $\chi^2=1$:
\[\Sel\biggl(\Q,E[p^\infty]\otimes \OO_d(\chi)\biggr)\cong \left(\OO_d\otimes \frac{\Q_p}{\Z_p}\right)^r \times \left(\frac{\OO_d}{(p)^{a_r}}\right)^2\times\left(\frac{\OO_d}{(p)^{a_{r+2}}}\right)^2\times \cdots\times \left(\frac{\OO_d}{(p)^{a_{s-2}}}\right)^2.\]
\item If $\chi^2\neq 1$:
 \[\Sel\biggl(\Q,E[p^\infty]\otimes \OO_d(\chi)\biggr)\cong \left(\OO_d\otimes \frac{\Q_p}{\Z_p}\right)^r \times \frac{\OO_d}{(p)^{{n_r-n_{r+1}}}}\times \cdots\times \frac{\OO_d}{(p)^{{n_{s-1}-n_{s}}}}.\]
\label{th:EK_str_intro}
\end{itemize}
\end{theorem*}

\begin{remark*}
Theorem \ref{th:EK_str_intro} describes the structure of $\Sel(K,E[p^\infty])$ as a $\Z_p[\Gal(K/\Q)]$-module (see Proposition \ref{prop:fitting_integral}). An example of this process in a particular elliptic curve is shown in \textsection\ref{sec:examples}.
\end{remark*}

Although Theorem \ref{th:EK_str_intro} assumes the Iwasawa main conjecture to describe the structure of the Selmer group, the theory of Kurihara numbers also provides a numerical check of the Iwasawa main conjecture.

\begin{theorem*} (Theorem \ref{th:EK_IMC})
Assume all the assumptions of Theorem \ref{th:EK_str_intro} holds except for the Iwasawa main conjecture. Then there exists some square-free product of primes in $\PP$, say $n$, such that $\delta_{n,\chi}\in \OO_d^\times$ if and only if the Iwasawa main conjecture for $f_\chi$ holds.
\label{th:EK_IMC_intro}
\end{theorem*}

The article is structured in three main sections. In \textsection\ref{sec:gc}, the basic definitions and results about the Selmer structures are stated and the assumptions \ref{Hirred}-\ref{Hprimes} necessary for the main results are explained. \textsection\ref{sec:ks} is dedicated to introducing the general theory of Kolyvagin systems and to the proof of the results describing the structure of the Selmer group in the core rank zero situation. In particular, Theorem \ref{th:kur_par_intro} is proven in \textsection\ref{sec:fitd} and the proof of Theorem \ref{th:kur} appears in \textsection\ref{sec:fitn}.

The particular case when the Selmer group is the one associated with an elliptic curve appears in \textsection\ref{sec:EC}. This section is structured with a detailed description of this specific Selmer structure and the main results in \textsection\ref{sec:EC_intro}, with the rest of the section dedicated to the proofs. \textsection\ref{sec:exp}-\ref{sec:Kato_kol} develop the link between the twisted Kato's Kolyvagin system and the twisted Kurihara numbers. The proof of Theorem \ref{th:EK_IMC_intro} is concluded in \textsection\ref{sec:proof_IMC} and Theorem \ref{th:EK_str_intro} is proven in \textsection\ref{sec:proof_str}.

\paragraph{Acknowledgments}
 The author is indebted to Christian Wuthrich for proposing the question that led to this article and for many helpful discussions and comments throughout the development of this paper. The author also thanks the anonymous referee for their careful reading of an earlier version of this manuscript and their attentive and insightful comments that contributed to the improvement of this work.

\section{Galois cohomology}
\label{sec:gc}

This section outlines some preliminary results on Galois cohomology and Selmer structures. The reader is referred to \cite{MazurRubin}, where a more complete treatment is given. Most results in this section are proven in loc.~cit. and, when so, the concrete result is precisely stated.

\subsection{Local Galois cohomology}

Throughout this and next section, we will assume the following.

\begin{assumption}
Let $R$ be a discrete valuation ring of characteristic $0$ with finite residue field of characteristic $p>2$ and let $\m$ denote its maximal ideal.\footnote{Some results are valid only when $R$ is artinian. In those cases, we will state the result for $R/\m^k$ for some positive integer $k$.} Let $T$ be a free, finitely generated $R$-module endowed with an $R$-linear continuous action of the absolute Galois group $G_\Q$ of the rational numbers that is ramified only at finitely many primes. We will denote the Cartier dual of $T$ by $T^*=\Hom(T,\mu_{p^\infty})$, where $\mu^{p^\infty}$ is the set of roots of unity whose orders are powers of $p$.
\label{ass:R}
\end{assumption}

For every rational prime $\ell$, we fix an inclusion of the algebraic closures $\overline \Q\hookrightarrow \overline \Q_\ell$. This choice determines an inclusion of the absolute Galois groups $G_{\Q_\ell}\subset G_\Q$. We will consider the Galois cohomology group:
$$H^1(\Q_\ell, T):=H^1(G_{\Q_\ell}, T).$$

Let $\II_\ell$ be the inertia subgroup of $G_{\Q_\ell}$, i.e., the absolute Galois group of the maximal unramified extension $\Q_{\ell, \textrm{ur}}$ of $\Q_\ell$.
\begin{definition}
Suppose that $T$ satisfies Assumption \ref{ass:R} and let $\ell$ be a prime different form $p$. The \emph{finite} cohomology subgroup $H^1_\f(\Q_\ell, T)\subset H^1(\Q_\ell, T)$ is the kernel of the restriction map
$$H^1(\Q_\ell, T)\to H^1(\II_\ell, T\otimes \Q_p).$$
We will also denote
\[H^1_\s(\Q_\ell, T):= \frac{H^1(\Q_\ell, T)}{H^1_\f(\Q_\ell,T)}.\]
\end{definition}

\begin{proposition}(\cite[lemma 1.3.5]{Rubin})
If $T$ is unramified at $\ell\neq p$, then 
$$H^1_\f(\Q_\ell, T)=H^1(\Q_{\ell, ur}/\Q_\ell, T^{G_{\Q_{\ell, \textrm{ur}}}})=\ker\left(H^1(\Q_\ell, T)\to H^1(\II_\ell, T)\right).$$
\end{proposition}

For certain primes, the finite cohomology group helps to describe the whole cohomology group. 

\begin{definition} 
Fix a natural number $k\in \N$. Let $\PP_k$ be the set of prime numbers $\ell$ satisfying the following assumptions:
\begin{itemize}
\item\namedlabel{P1modp}{(P1)} $\ell\equiv 1\mod p^k$.
\item\namedlabel{P1dim}{(P2)} $T/(\m^k, \Frob_\ell-1) T$ is a free $R/\m^k$-module of rank $1$, where $\Frob_\ell$ is the arithmetic Frobenius at $\ell$.
\end{itemize}
The set of square-free products of primes in $\PP_k$ will be denoted by $\NN(\PP_k)$. For every $n\in \NN(\PP_k)$, we write $\nu(n)$ for the number of primes dividing $n$.
\end{definition}

In order to make the computations of the primes $\PP_k$ in the particular example described in \textsection \ref{sec:EC}, it is interesting to define $\PP_k$ relative to an extension $F/\Q$.

\begin{definition}
Let $k\in \N$ be a natural number and let $F/\Q$ be a Galois extension. Define $\PP_{k,F}$ as the subset of primes $\ell$ in $\PP_k$ such that 
\begin{itemize}
\item \namedlabel{Psplits}{(P3)} $\ell$ splits completely in $F/\Q$.
\end{itemize}
\end{definition}

\begin{remark}
Let $L$ be the field associated with the kernel of the map $G_\Q\to \textrm{Aut}(T^*[p^k])$. By the Chebotarev density theorem, if $\PP_k$ is non-empty and $L\cap F=\Q$, then $\PP_{k,F}$ is an infinite set.
\end{remark}

\begin{lemma} (\cite[lemma 1.2.1]{MazurRubin})
Let $T$ be as in Assumption \ref{ass:R}, assume $\ell$ satisfies \ref{P1modp} and denote by $\GG_\ell$ the $p$-primary part of the Galois group of $\Q(\mu_\ell)/\Q$, where $\mu_\ell$ is the set of $\ell^{\textrm{th}}$ roots of unity. Then there are canonical functorial isomorphisms:
\begin{itemize}
\item $H^1_\f(\Q_\ell, T/\m^k T)\cong T/(\m^k,\Frob_\ell-1) T$.
\item $H^1_\s(\mathbb Q_\ell, T/\m^k T)\otimes \GG_\ell \cong (T/\m^kT)^{\Frob_\ell=1}$.
\end{itemize}
\label{lem:mr121}
\end{lemma}

For every prime $\ell$ in $\PP_k$, it is possible to define the finite-singular map (see \cite[Definition 1.2.2]{MazurRubin}), which is a canonical isomorphism
\[\phi_\ell^\fs:\ H^1_\f\left(\Q_\ell, {T}/{\m^kT}\right)\to H^1_\s\left(\Q_\ell, {T}/{\m^kT}\right)\otimes \GG_\ell.\]

\begin{lemma}(\cite[lemma 1.2.3]{MazurRubin})
If $\ell\in \PP_k$ and $T$ satisfies Assumption \ref{ass:R}, the map $\phi_\ell^{\fs}$ is an isomorphism.
\label{lem:mr123}
\end{lemma}

Another important subgroup of the local Galois cohomology group is the transverse subgroup, which is defined via the cohomological inflation map:
$$H^1_\tr(\Q_\ell, T/\m^k T):=\textrm{Im}\left(\textrm{Inf}: H^1(\Q_\ell(\mu_\ell)/\Q_\ell, T/\m^k T)\hookrightarrow H^1(\Q_\ell, T/\m^kT)\right).$$

\begin{lemma}(\cite[lemma $1.2.4$]{MazurRubin})
Let $T$ be as in Assumption \ref{ass:R} and assume $\ell$ satisfies \ref{P1modp}. Then there is a functorial splitting
$$H^1(\Q_\ell, T/\m^k T)=H^1_\f(\Q_\ell, T/\m^k T)\oplus H^1_{\tr}(\Q_\ell, T/\m^kT).$$
\label{lem:local_split}
\end{lemma}

The most important feature about the finite local cohomology is that it behaves well under Tate duality.
\begin{proposition}(\cite[Proposition 1.3.2]{MazurRubin})
Let $\ell$ be a finite prime. The cup product induces a perfect pairing
$$H^1(\Q_\ell,T)\times H^1(\Q_\ell, T^*)\to H^2(\Q_\ell, \mu_{p^\infty})\cong \Q_p/\Z_p.$$
If $T$ is unramified at $\ell\neq p$, then
\begin{enumerate}
\item $H^1_\f(\Q_\ell,T)$ and $H^1_\f(\Q_\ell, T^*)$ are orthogonal complements under this pairing.
\item If $\ell$ satisfies \ref{P1modp} for some $k\in \N$, then $H^1_{\tr}(\Q_\ell,  T)$ and $H^1_{\tr}(\Q_\ell, T^*[\m^k])$ are orthogonal complements.
\end{enumerate}
\label{prop:local_duality}
\end{proposition}

\begin{corollary}
Suppose $T$ satisfies Assumption \ref{ass:R}. Given a prime $\ell\in \PP_k$, the finite and singular cohomology groups $H^1_\f(\Q_\ell,T/\m^k)$, $H^1_\s(\Q_\ell, T/m^k)$, $H^1_\f(\Q_\ell,T^*[\m^k])$ and $H^1_\s(\Q_\ell,T^*[\m^k])$ are free $R/\m^k$-modules of rank one, where $T^*[m^k]$ denotes the $m^{k}$-torsion elements of $T^*$.
\label{cor:local_coh_cyclic}
\end{corollary}

\begin{proof}
By \ref{P1dim} and Lemma \ref{lem:mr121}, $H^1_\f(\Q_\ell, T/\m^k T)$ is free of rank one over $R/\m^k$. By Lemma \ref{lem:mr123}, the finite-singular map $\phi_\ell^{\fs}$ is an isomorphism, so $H^1_\s(\Q_\ell,T)$ is a rank one free $R/\m^k$-module too.

If, for every $R$-module $A$, $A^\vee$ denotes its Pontryagin dual 
\[A^\vee:=\Hom(A,\Q_p/\Z_p),\] 
then Proposition \ref{prop:local_duality} implies that
\[\begin{array}{cc}
H^1_\f(\Q_\ell,T^*[\m^k])\cong H^1_\s(\Q_\ell, T/\m^k T)^\vee;\ \ \ &
H^1_\s(\Q_\ell,T^*[\m^k])\cong H^1_\f(\Q_\ell, T/\m^k T)^\vee.\ \ \end{array}\]
Hence $H^1_\f(\Q_\ell,T^*[\m^k])$ and $H^1_\s(\Q_\ell,T^*[\m^k])$ are also free $R/\m^k$-modules of rank one.
\end{proof}

\subsection{Selmer structures}
\label{sec:global}

An important tool to study the cohomology of the global absolute Galois group is the use of Selmer structures.

\begin{definition}
A \emph{Selmer structure} $\mathcal F$ on $T$ is a collection of the following data:
\begin{itemize}
\item A finite set $\Sigma(\mathcal F)$ of places of $\Q$, including $\infty$, $p$ and all primes where $T$ is ramified.
\item For every $\ell\in \Sigma(\mathcal F)$, an $R[[G_{\Q_\ell}]]$-submodule $H^1_\mathcal F(\Q_\ell,T)\subset H^1(\Q_\ell, T)$, where $G_{\Q_\ell}$ is the absolute Galois group of $\Q_\ell$. This choice is usually referred to as a \emph{local condition} at the prime $\ell$.
\end{itemize}
\label{def:selmer_structure}
\end{definition}

\begin{definition}
The \emph{Selmer module} $H^1_\FF(\Q,T)\subset H^1(\Q,T)$ associated with a Selmer structure $\FF$ is the kernel of the sum of restriction maps
$$H^1(\Q_{\Sigma(\FF)}/\Q,T)\to \bigoplus_{\ell\in \Sigma(\FF)} H^1(\Q_\ell,T)/H^1_\FF(\Q_\ell,T),$$
where $\Q_{\Sigma(\FF)}$ is the maximal extension of $\Q$ unramified outside $\Sigma(\FF)$.
\end{definition}

\begin{remark}
Given a Selmer structure $\FF$, we will denote $H^1_\FF(\Q_\ell, T)=H^1_\f(\Q_\ell,T)$ for every prime $\ell\notin \Sigma(\FF)$.
\end{remark}

\begin{definition}
Let $\FF$ be a Selmer structure defined on $T$. If $T'\subset T$ is a submodule of $T$, $\FF$ induces a Selmer structure where the local conditions $H^1_\FF(\Q_\ell, T')$ are the inverse images of $H^1_\FF(\Q_\ell,T)$ under the natural map $H^1(\Q_\ell,T')\to H^1(\Q_\ell, T)$.\\
Similarly, if $T\twoheadrightarrow T''$ is a surjection, $\FF$ induces a Selmer structure on $T''$ where the local conditions $H^1(\Q_\ell, T'')$ are the images $H^1(\Q_\ell, T)$ under the map $H^1(\Q_\ell,T)\to H^1(\Q_\ell, T'')$.
\label{def:induced_str}
\end{definition}

\begin{definition}
Given a Selmer structure $\FF$ on $T$, there is a dual Selmer structure $\FF^*$ on the Cartier dual $T^*$ where, for every prime $\ell$, the local conditions $H^1_{\FF^*}(\Q_\ell,T^*)$ is the orthogonal complement of $H^1_\FF(\Q_\ell,T)$ under the pairing defined in Proposition \ref{prop:local_duality}. Note $\FF^*$ is well defined by Proposition \ref{prop:local_duality}.
\end{definition}

In order to compare Selmer structures, Poitou-Tate global duality is helpful.
\begin{proposition}(\cite[Theorem 2.3.4]{MazurRubin})
Let $\FF$ and $\GG$ be Selmer structures of $T$ such that $H^1_{\FF}(\Q_\ell, T)\subset H^1_{\GG}(\Q_\ell,T)$ for every prime $\ell$. Then the following sequence is exact.

\begin{center}
\begin{tikzpicture}[descr/.style={fill=white,inner sep=1.5pt}]
        \matrix (m) [
            matrix of math nodes,
            row sep=3.5em,
            column sep=3em,
            text height=1.5ex, text depth=0.25ex
        ]
        { 0 & H^1_{\FF}(\Q,T) & H^1_{\GG}(\Q,T) & \displaystyle{\bigoplus_{\ell\in \Sigma(\FF)\cup\Sigma(\GG)} \frac{H^1_{\GG}(\Q_\ell,T)}{H^1_{\FF}(\Q_\ell,T)}} \\
            & H^1_{\FF^*}(\Q,T^*)^\vee & H^1_{\GG^*}(\Q,T^*)^\vee& 0, &\\
        };

        \path[overlay,->, font=\scriptsize,>=latex]
        (m-1-1) edge (m-1-2)
        (m-1-2) edge (m-1-3)
        (m-1-3) edge (m-1-4)
        (m-1-4) edge [out=355,in=175] (m-2-2) 
        (m-2-2) edge (m-2-3)
        (m-2-3) edge (m-2-4);
        
\end{tikzpicture}
\end{center}
where the third map is induced by Proposition \ref{prop:local_duality}.
\label{prop:global_duality}
\end{proposition}

In order to compute the structure of a Selmer group, it is helpful to make use of slightly modified Selmer structures.
\begin{definition}
Given a Selmer structure $\mathcal F$ defined on a Galois representation $T$, and square-free and pairwise relatively prime integers $a$, $b$ and $c$, we define a new Selmer structure $\mathcal F_a^b(c)$ given by the following data:
\begin{itemize}
\item $\Sigma(\FF_a^b(c)):=\Sigma(\mathcal F)\cup\{\ell|abc\},$
\item $H^1_{\FF_a^b(c)}(\Q,T)=\left\{\begin{aligned}
&H^1(\Q_\ell,T)\ &\textrm{if }\ell|a,\\
&0\ &\textrm{if }\ell|b,\\
&H^1_{\tr}(\Q_\ell,T)\ &\textrm{if }\ell|c,\\
&H^1_\FF(\Q_\ell,T)\ &\textrm{otherwise.}
\end{aligned}\right.$
\end{itemize}
If any of $a$, $b$ or $c$ is not explicitly written, we will assume it to be one. For example $\FF^\ell$ is shorthand of $\FF^\ell_1(1)$.
\label{def:modified}
\end{definition}

We now define the notion of Selmer triples
\begin{definition}
A \emph{Selmer triple} is a triple $(T,\FF,\PP)$ where $\FF$ is a Selmer structure on a Galois representation $T$ and $\mathcal P$ is a set of rational primes disjoint from $\Sigma(\FF)$.
\end{definition}

For technical reasons, we need to assume the following mild hypotheses in our Selmer triples:
\begin{itemize}

\item\namedlabel{Hirred}{(H1)} $T/\m T$ is an absolutely irreducible $R/\m[G_\Q]$-representation.
\item\namedlabel{Hcoh0}{(H2)} $H^1(\Q(T,\mu_{p^\infty})/\Q,T/\m T)=H^1(\Q(T,\mupi)/\Q,T^*[\m])=0$, where $\Q(T,\mu_{p^\infty})$ is the minimal extension of $\Q$ whose absolute Galois group acts trivially on $T$, $T^*$ and $\mu_{p^\infty}$.

\item\namedlabel{Hsur}{(H3)} There is a $\tau\in G_\Q$ such that $\tau$ acts trivially on $\mu_{p^\infty}$ and $T/(\tau-1)T$ is free of rank one over $R$.
\item\namedlabel{Hchev}{(H4)} Either
\begin{itemize}
\item\namedlabel{Hchev_hom}{(H4a)} $\Hom_{\F_p[G_\Q]}(T/\m T,T^*[\m])=0$ or
\item\namedlabel{Hchev_5}{(H4b)} $p\geq 5$.
\end{itemize}
\item\namedlabel{Hloc}{(H5)} There is a prime $q$ such that $H^1_{/\FF}(\Q_q,T)\cong R$ and $H^2(\Q_q, T)=0$.

\item\namedlabel{Hcartesian}{(H6)} For every $\ell\in\Sigma(\mathcal F)$, the local condition of $\mathcal F$ at $\ell$ is cartesian in the category of quotients of $T$, in the sense of \cite[definition 1.1.4]{MazurRubin}.
\item\namedlabel{Hcore0}{(H7)} The core rank $\chi(T,\FF)$ is zero (see Theorem \ref{th:core_rank} below), i.e., for every $n\in\mathcal N(\mathcal P)$ and $j\in \Z_{\geq0}$ there is a non-canonical isomorphism
$$H^1_{\mathcal F(n)}(\Q,T/\m^j)\cong H^1_{\mathcal F^*(n)}(\Q,T^*[\m^j]).$$

\item \namedlabel{Hprimes}{(H8)} There exists some $k\in \mathbb N$ and some Galois extension $F/\Q$ such that $F\cap \Q(T/\m^k)=\Q$ and $\left(\PP_{k,F}\setminus \Sigma(\FF)\right)\subset \PP\subset \PP_1 $.
\end{itemize}

\begin{remark} (\cite[remark 3.5.1]{MazurRubin})
Suppose that $(T,\FF,\PP)$ satisfies hypotheses \ref{Hirred}-\ref{Hprimes}. Then $( T/\m^k T, \FF,\PP)$ also satisfies \ref{Hirred}-\ref{Hprimes} for every $k\in \Z_{\geq}0$.
\label{rem:Hi_quot}
\end{remark}

All of the above hypothesis except for \ref{Hloc} and \ref{Hcore0} were imposed by Mazur and Rubin in \cite[\textsection $3.5$]{MazurRubin} to develop the basic theory of Kolyvagin systems when $\chi(T,\FF)=1$. 

\ref{Hirred}, \ref{Hcoh0} and \ref{Hcartesian} are necessary to ensure that the Selmer group in the cohomology with torsion coefficient rings $H^1_\FF(\Q,T/\m^j)$ (resp. $H^1_{\FF^*}(\Q,T^*[\m^j])$) coincides with the $\m^j$-torsion of the full Selmer group $H^1_{\FF}(\Q,T)$ (resp. $H^1_{\FF^*}(\Q,T^*)$). This is explained in the following lemmas.

\begin{lemma}(\cite[lemma 3.5.3]{MazurRubin}, \cite[Corollary 3.8]{BurnsSakamotoSano2})\footnote{This lemma was first established in \cite{MazurRubin}. However, the proof in loc.~cit. was not fully correct. A complete argument is given in \cite{BurnsSakamotoSano2}.}
Assume that $R$ and $T$ satisfy Assumption \ref{ass:R} and that \ref{Hirred} and \ref{Hcoh0} hold. For every $j\in \Z_{\geq 0}$, the inclusion $T^*[\m^j]\hookrightarrow T$ induces an isomorphism:
\[H^1_{\FF^*}(\Q,T^*[\m^j])\cong H^1_{\FF^*}(\Q,T^*)[\m^j].\]
\label{lem:mr353}
\end{lemma}

\begin{lemma}(\cite[lemma 3.5.4]{MazurRubin})
    Suppose that $R$ and $T$ satisfy Assumption \ref{ass:R}, that $R$ that \ref{Hirred}, \ref{Hcoh0} and \ref{Hcartesian} hold. For every $j\leq k$, if $\pi$ is a generator of $\m$, the multiplication by $\pi^{k-j}$ induces a map $T/\m^j\hookrightarrow T/\m^k$, which in turn induces an isomorphism in the cohomology groups:
    \[H^1_{\FF}(\Q,T/\m^j)\cong H^1_{\FF}(\Q,T/\m^k)[\m^j].\]
    \label{lem:mr354}
    \end{lemma}

\ref{Hsur}, \ref{Hchev} and \ref{Hprimes} are required to ensure there exists infinitely many Kolyvagin primes satisfying \ref{P1modp} and \ref{P1dim} that bound the Selmer group appropriately. Assumption \ref{Hprimes} is slightly weaker than the corresponding hypothesis in \cite{MazurRubin}. It does not suppose any inconvenience in the proofs of this reference and permits a cleaner description of the Kolyvagin primes in the example in \textsection\ref{sec:EC}.

\ref{Hcore0} is imposed because we want to generalise the theory in \cite{MazurRubin} to the case of core rank 0. However, those Selmer structures do not admit any non-trivial Kolyvagin systems. Therefore, \ref{Hloc} is required to modify the structure in order to link it with a Selmer structure of rank one, $\FF^q$, which the theory in \cite{MazurRubin} can be applied to. Note that, by local Tate duality, the condition $H^2(\Q_q,T)$ is equivalent to $H^0(\Q_q, T^*)=0$.

Assumption \ref{Hloc} can be understood in terms of the cohomology groups with torsion coefficients. 

\begin{lemma}
    Suppose that $R$ and $T$ satisfy Assumption \ref{ass:R} and let $q$ be a prime satisfying \ref{Hloc}. Then, for every $j\in \Z_{\geq 0}$, there is an isomorphism
$$H^1_\s(\Q_q,T/\m^j)\cong R/\m^j.$$
\label{lem:Hloc_tors}
\end{lemma}

\begin{proof}
Let $\pi$ be a generator of $\m$ and consider the short exact sequence
\[\xymatrix{0\ar[r] & T\ar[r]^{\pi^j} & T\ar[r] & T/\m^j T\ar[r] &0}.\]
Since $H^2(\Q_q,T)=0$, there is an isomorphism
$$H^1(\Q_q,T)/\m^j H^1(\Q_q,T)\xrightarrow{\sim} H^1(\Q_q,T/\m^jT).$$
By Definition \ref{def:induced_str}, $H^1_\FF(\Q_q, T/\m^j T)$ is the image of $H^1_\FF(\Q_q,T)$ via the projection map $T\twoheadrightarrow T/\m^j T$. By \ref{Hloc}, we have that 
\[H^1_{/\FF}(\Q_q, T/\m^j T)\cong H^1_{/\FF}(\Q_q, T)/\m^j H^1_{/\FF}(\Q_q, T)\cong R/\m^j.\qedhere\]
\end{proof}

Some elements in $\NN(\PP)$ will play an important role in the theory of Kolyvagin systems.

\begin{definition}
An element $n\in \NN(\PP)$ is said to be a \emph{core vertex} if either $H^1_{\FF(n)}(\Q,T)=0$ or $H^1_{\FF^*(n)}(\Q,T^*)=0$.
\end{definition}

\begin{proposition}(\cite[Corollary 4.1.9]{MazurRubin})
Suppose that $R$ and $T$ satisfy Assumption \ref{ass:R}. For every $n\in \NN(\PP)$, there is some core vertex $m\in \NN(\PP)$ such that $n|m$.
\label{prop:core_vertex}
\end{proposition}

The concept of core rank plays a vital role in the study of Selmer structures.

\begin{theorem}(\cite[Theorem 4.1.5]{MazurRubin})
    Suppose that $R$ and $T$ satisfy Assumption \ref{ass:R} and let $\FF$ be a cartesian Selmer structure. Then there exist integers $x,y\in \Z_{\geq 0}$, one of which can be taken to be $0$, such that for every $n\in \NN(\PP)$ and every $j\in \Z_{\geq0}$ there is a non-canonical isomorphism
    $$H^1_{\FF(n)}(\Q,T/\m^j)\oplus (R/\m^j)^x\cong H^1_{\FF^*(n)}(\Q,T^*[\m^j])\oplus (R/\m^j)^y.$$
    The integer $\chi(T,\FF)=y-x$ is said to be the \emph{core rank} of $\FF$.
    \label{th:core_rank}
    \end{theorem}

\begin{proposition}(\cite[Corollary 3.21]{Sakamoto18})
    Let $(T,\FF,\PP)$ be a Selmer triple in which $T$ satisfies Assumption \ref{ass:R} and let $a,b,c\in \NN(\PP_k)$ be pairwise relatively prime. Then
    $$\chi( T/\m^k T,\FF_a^b(c))=\chi(T/\m^k T,\FF)+\nu(b)-\nu(a).$$
    \label{prop:core_abc}
    \end{proposition}

    \begin{remark}
    Proposition \ref{prop:core_abc} is also true when $a,b\in \mathcal N(\PP')$, where $\PP'$ denotes the set of primes $\ell$ such that $H^1_\s(\Q_\ell,T/\m^k T)$ is free of rank one over $R/\m^k$. 
    \label{rem:core_abc}
    \end{remark}

The key idea of this work is that we can bound the Selmer group by using appropriate primes, whose existence is known due to the Chebotarev density theorem. To be precise, we will now state the result needed.

\begin{proposition}(\cite[Proposition 3.6.1]{MazurRubin})
Suppose that $R$ and $T$ satisfy Assumption \ref{ass:R} and let $k\in \N$. Let $c_1,c_2\in H^1(\Q,T/\m^k)\setminus\{0\}$ and let $c_3,c_4\in H^1(\Q,T^*[\m^k])\setminus\{0\}$. Then there exists an infinite set of primes $S\subset \PP_k$ such that
$$\loc_\ell(c_i)\neq 0\ \ \forall \ell\in S,\ \forall i=1,2,3,4.$$
\label{prop:mr361}
\end{proposition}

The main consequence of this proposition is that it can be used to find primes whose finite localisation maps are surjective. The following corollary is proven within \cite[Proposition 4.5.8]{MazurRubin}.

\begin{corollary}(\cite[Proposition 4.5.8]{MazurRubin})
Let $(T,\FF,\PP)$ be a Selmer triple in which $T$ satisfies Assumption \ref{ass:R} and \ref{Hirred}-\ref{Hprimes} and let $j\in\NN$ be such that $\m^{j-1} H^1_\FF(\Q,T)\neq 0$. Then there are infinitely many primes $\ell \in \PP$ such that 
\[\loc_\ell\bigl(H^1_{\FF}(\Q,T/\m^j)\bigr)=H^1_f(\Q_\ell,T/\m^j).\]
Assume further that $\m^{j-1} H^1_\FF(\Q,T^*)\neq 0$. Then there are infinitely many primes $\ell\in \PP$ such the following two conditions hold:
\[\begin{array}{cc}
    \loc_\ell\bigl(H^1_{\FF}(\Q,T/\m^j)\bigr)=H^1_\f(\Q_\ell,T/\m^j),\ &
    \loc_\ell\bigl(H^1_{\FF^*}(\Q,T^*[\m^j])\bigr)=H^1_\f(\Q_\ell,T^*[\m^j]).
\end{array}\]

\label{cor:loc_sur}
\end{corollary}

This result alongside the following lemma plays a central role in the proof of Theorem \ref{th:mr459} below.

\begin{lemma}(\cite[lemma 4.1.7(ii)]{MazurRubin})
Let $(T,\FF,\PP)$ be a Selmer triple in which $T$ satisfies Assumption \ref{ass:R} and let $\ell\in \PP_k$ be a prime not dividing $n\in\NN(\PP_k)$. If the localisation map $$\loc_\ell: H^1_{\FF(n)}(\Q, T)\to H^1_{\f}(\Q_\ell, T/\m^k T)$$ is surjective, then
$$H^1_{\FF^*(n\ell)}(\Q,T^*[\m^k])=H^1_{(\FF^*)_\ell(n)}(\Q,T^*[\m^k]).$$
\label{lem:mr417}
\end{lemma}

The above two results imply the following
\begin{corollary}(\cite[Proposition 4.5.8]{MazurRubin})
Let $(T,\FF,\PP)$ be a Selmer triple in which $T$ satisfies Assumption \ref{ass:R}. Suppose that \ref{Hirred}-\ref{Hchev}, \ref{Hcartesian} and \ref{Hprimes} are satisfied and assume that $\chi(T/\m^k T,\FF)>0$ and that $H^1_{\FF^*}(\Q,T^*[\m^k])\cong R/\m^{e_1}\times \cdots\times R/\m^{e_s}$ for some integers $e_1\geq e_2\geq\ldots \geq e_s$. Then there are infinitely many primes $\ell\in \PP$ such that 
$$H^1_{\FF^*(\ell)}(\Q,T^*[\m^k])\cong R/\m^{e_2}\times \cdots\times R/\m^{e_s}.$$
\label{cor:mr458}
\end{corollary}

However, the above lemma only applies when $\chi(T/\m^k T,\FF)\geq 0$. In the case where $\chi(T/\m^k T,\FF)=0$, those primes might not exist but we can still get the following partial result.

\begin{lemma}
Let $(T,\FF,\PP)$ be a Selmer triple in which $T$ satisfies Assumption \ref{ass:R} and that \ref{Hirred}-\ref{Hchev}, \ref{Hcartesian} and \ref{Hprimes} hold. In addition, assume that $\chi(T/\m^k\FF)=0$ and that $H^1_{\FF^*}(\Q,T^*[\m^k])\cong R/\m^{e_1}\times \cdots\times R/\m^{e_s}$ for some integers $e_1\geq e_2\geq\ldots \geq e_s$. Then there are infinitely many primes $\ell\in \PP$ such that 
$$H^1_{\FF^*(\ell)}(K,T^*[\m^k])\cong R/{\m^j}\times R/\m^{e_3}\times \cdots\times R/\m^{e_s},$$
where $e_2\leq j\leq k$. In the case when $e_1>e_2$, then $j=e_2$ holds true for infinitely many primes $\ell\in \PP$.
\label{lem:cr0}
\end{lemma}

\begin{proof}
Choose some $k\in \N$ satisfying \ref{Hprimes}. Since $\chi(T/\m^k T,\FF)=0$, Theorem \ref{th:core_rank} implies that 
$$H^1_{\FF^*}(K,T^*[\m^k])\cong H^1_{\FF}(K,T/\m^k T)\cong R/\m^{e_1}\times\cdots \times R/\m^{e_s}.$$

We can pick classes $c\in H^1_{\FF}(K,T/\m^k T)$ and $c^*\in H^1_{\FF^*}(K,T^*[\m^k])$ such that $\pi^{e_1} c\neq 0$ and $\pi^{e_1} c^*\neq 0$. By Proposition \ref{prop:mr361}, we can choose a Kolyvagin prime $\ell\in \PP$ such that 
\[\begin{array}{cc}
\loc_\ell(\pi^{e_1-1} c) \neq 0,\ & \loc_\ell(\pi^{e_1-1}c^*)\neq 0.
\end{array}\]
Consider the diagram 
\[\xymatrix{ H^1_{\FF}(K,T/\m^{e_1})\ar[r]^{\loc_\ell} \ar[d]^{\pi^{k-e_1}}      &       H^1_\f(K_\ell,T/\m^{e_1}) \ar[d]^{\pi^{k-e_1}} \\
            H^1_{\FF}(K,T/\m^k T)\ar[r]^{\loc_\ell}     &       H^1_\f(K_\ell,T/\m^k T).
}\]
By Lemma \ref{lem:mr354}, since $H^1_{\FF}(K,T/\m^k T)=H^1_{\FF}(K,T/\m^k T)[\m^{e_1}]$, the leftmost vertical map is surjective, so there is some $c'\in H^1_{\FF}(K,T/\m^{e_1})$ such that $\loc_\ell(\pi^{e_1-1}c')\neq 0$, implying that the top horizontal map is surjective.

Note that the element $\tau\in G_K$ from \ref{Hsur} satisfies that 
\[T/(\m^{e_1},\tau-1)\cong R/\m^{e_1}\]
and, since $K(T/\m^{e_1})_{p^{e_1}}\subset K(T)_;$, then $\Frob_\ell$ is conjugate to $\tau$ in $\Gal(K(T/\m^{e_1})_{p^{e_1}}/K)$, so it is a Kolyvagin prime for $T/\m^{e_1}$. Then we can apply Lemma \ref{lem:mr417} to guarantee that
\[H^1_{(\FF^*)(\ell)}(K,T^*[\m^{e_1}])=H^1_{(\FF^*)_\ell}(K,T^*[\m^{e_1}])\cong R/\m^{e_2}\times \cdots \times R/\m^{e_s}.\]

By Lemma \ref{lem:mr353},
\[H^1_{\FF^*(\ell)}(\Q,T^*)[\m^{e_1}]= H^1_{\FF^*(\ell)}(\Q,T^*[\m^{e_1}])\cong R/\m^{e_2}\times \cdots \times R/\m^{e_s}.\]
Finally, Theorem \ref{th:core_rank} implies that
\[H^1_{\FF(\ell)}(K,T)[\m^{e_1}]\cong H^1_{\FF^*(\ell)}(\Q,T^*)[\m^{e_1}] \cong R/\m^{e_2}\times \cdots \times R/\m^{e_s}. \]

Since $\chi(T/\m^k T,\FF^\ell)=1$ by by Proposition \ref{prop:core_abc}, then Theorem \ref{th:core_rank} implies that
$$ H^1_{\FF(\ell)}(\Q,T)\subset H^1_{\FF^\ell}(\Q,T)\cong R/\m^k\oplus H^1_{(\FF_\ell)^*}(\Q,T^*).$$
Therefore, $H^1_{\FF(\ell)}(\Q,T)$ can be injected into $R/\m^k\times R/\m^{e_2}\times\cdots \times R/\m^{e_s}$ and its $\m^{e_1}$-torsion is isomorphic to $R/\m^{e_2}\times \cdots \times R/\m^{e_s}$. Under those considerations, the lemma follows by the structure theorem of $R/\m^k$-modules. 
\end{proof}

Given a Galois representation $T$ defined over a discrete valuation ring $\OO$, there are some natural Selmer structures that can be defined on $T$. The first example is called the \emph{canonical Selmer structure}. The finite cohomology group is a well behaved local condition for primes different from $p$ and $\infty$. The idea behind the canonical example is to consider the local conditions at those primes to be the whole local cohomology groups.

\begin{definition}
    Let $T$ be a Galois representation over a discrete valuation ring of residual characteristic $p>0$. Then the canonical local conditions are defined as
    $$\left\{\begin{aligned}
& H^1_{\Fcan}(\Q_\ell,T)=H^1_\f(\Q_\ell, T)\ \ \textrm{if}\ \ell\neq p,\\
&H^1_{\Fcan}(\Q_p,T)=H^1(\Q_p,T),\\
&H^1_{\Fcan}(\R,T)=H^1(\R,T).
\end{aligned}\right.$$
\label{def:Fcan_str}
\end{definition}

The canonical structure is the one that can be studied directly by using Euler systems. However, it is usually different to the Selmer groups that appear naturally in arithmetic geometry, because they differ in the local conditions at $p$ and $\infty$. The infinite prime is usually not an issue since $H^1(\R,T)=0$ when $p>2$. At $p$, S.~Bloch and K.~Kato gave in \cite{BlochKato} a definition of a local condition at $p$, based purely on the Galois representation $T$. In the most arithmetically important cases, it coincides with the local condition at $p$ defined geometrically. Their definition uses Fontaine's period rings from $p$-adic Hodge theory. Moreover, Bloch-Kato Selmer group is conjecturally directly related to the special values of the motivic $L$-function.

\begin{definition}
Let $T$ be a Galois representation over a discrete valuation ring $R$ of residual characteristic $p>0$. Then the Bloch-Kato local conditions are defined as
$$\left\{\begin{aligned}
& H^1_{\FBK}(\Q_\ell,T):=H^1_\f(\Q_\ell, T)\ \ \textrm{if}\ \ell\neq p,\\
&H^1_{\FBK}(\Q_p,T):=\ker\left(H^1(\Q_\ell, T)\to H^1(\Q_\ell, T\otimes_{\Z_p} B_\textrm{crys})\right).
\end{aligned}\right.,$$
where $B_\crys$ denotes the crystalline period ring.
\label{def:BK_str}
\end{definition}

\section{Kolyvagin systems}
\label{sec:ks}

\subsection{General theory}
\label{sec:kol}

\begin{definition}
A \emph{Kolyvagin system} for a Selmer triple $(T,\FF,\PP)$ is a collection of cohomology classes
$$\kappa=\left\{\kappa_n\in H^1_{\mathcal F(n)}(\Q,T/\m^{k_n} T)\otimes  \bigotimes_{\ell\mid n} \GG_\ell:\ n\in \NN(\PP)\right\}$$
such that 
\begin{equation}
\loc_\ell^\s(\kappa_{n\ell})=\phi_\ell^{\fs}\circ \loc_\ell(\kappa_n)
\label{eq:kol_cond}
\end{equation}
for every $n\in \NN(\PP)$ and $\ell\in \PP$ such that $(n,\ell)=1$. Here, $\loc_{\ell}^\s$ denotes the composition of the localisation at $\ell$ with the projection to the singular cohomology group. By Definition \ref{def:modified}, we have that 
\[\begin{aligned}
&\loc_\ell(\kappa_n)\in H^1_\f(\Q_\ell,T/\m^{k_n} T)\otimes \bigotimes_{q\mid n} \GG_q,\\
&\loc_\ell^\s(\kappa_{n\ell})\in H^1_\s(\Q_\ell,T/\m^{k_{n\ell}})\otimes \bigotimes_{q\mid n\ell} \GG_q.\\
\end{aligned}\]
Applying the finite-singular map at $\ell$ from Lemma \ref{lem:mr123} to $\kappa_{n}$, we obtain an element in 
\[\phi_\ell^{\fs}\circ \loc_\ell(\kappa_n)\in H^1_\s(\Q_\ell,T/\m^{k_n} T)\otimes\bigotimes_{q\mid n\ell} \GG_q.\]
Both $\loc_\ell^\s(\kappa_{n\ell})$ and $\phi_\ell^{\fs}\circ \loc_\ell(\kappa_n)$ project to an element in 
\[H^1_\s(\Q_\ell,T/\m^{\min\{k_n,k_{n\ell}\}})\otimes\bigotimes_{q\mid n\ell} \GG_q,\] 
so it is in this group where the relation in \eqref{eq:kol_cond} shows up.
\label{def:kol}
\end{definition}

The module of Kolyvagin systems will be denoted by $\textrm{KS}(T,\FF,\PP)$ and it can be determined depending on the core rank of the Selmer triple $(T,\FF,\PP)$.
\begin{theorem} (\cite[Theorems 4.2.2, 4.3.2 and 5.2.10]{MazurRubin})
Let $T$ be as in Assumption \ref{ass:R}, let $(T,\FF,\PP)$ be a Selmer triple satisfying \ref{Hirred}-\ref{Hchev}, \ref{Hcartesian} and \ref{Hprimes} and let $k\in \N$.
\begin{itemize}
\item If $\chi(T,\FF)=0$, then $\textrm{KS}(T,\FF,\PP)=0$.
\item If $\chi(T,\FF)=1$, then $\textrm{KS}(T,\FF,\PP)$ is free of rank $1$ over $R$. 
\end{itemize}
\label{th:KS}
\end{theorem}

The Kolyvagin systems giving the best bound to the Selmer group are those generating the module $\textrm{KS}(T,\FF,\PP)$.
\begin{definition}
If $(T,\FF,\PP)$ satisfies Assumption \ref{ass:R}, \ref{Hirred}-\ref{Hchev}, \ref{Hcartesian} and \ref{Hprimes} and $\chi(T,\FF)=1$, a Kolyvagin system is said to be \emph{primitive} if it generates $\textrm{KS}(T,\FF,\PP)$ as an $R$-module.
\label{def:primitive}
\end{definition}

\begin{proposition}(\cite[Corollary 4.5.2]{MazurRubin})
In the setting of Definition \ref{def:primitive}, a Kolyvagin system $\kappa\in \textrm{KS}(T,\FF,\PP)$ is primitive if and only if its image under the canonical map $\textrm{KS}(T,\FF,\PP)\to\textrm{KS}(T/\m T,\FF,\PP)$ is non-zero.
\label{prop:mr452}
\end{proposition}

\begin{definition}
Given a Kolyvagin system $\kappa \in \textrm{KS}(T/\m^k T,\FF, \PP)$, for every $n\in \NN(\PP)$, define the order of $\kappa_n$
$$\ord(\kappa_n)=\max\left\{j\in \N\cup\{0,\infty\}: \kappa_n\in \m^jH^1_{\mathcal F(n)}(\Q,T/\m^{k_n} T)\right\}.$$
We also consider the following quantities
$$\partial^{(i)}(\kappa)=\min\left\{\ord(\kappa_n):n\in\NN(\PP),\nu(n)=i\right\}.$$
\label{def:order}
\end{definition}

\begin{remark}
When $\ord(\kappa_n)$ is less than $k_n$, our definition of order coincides with the standard $p$-adic valuation $\ord_p$. The only difference appears when $\kappa_n\in \m^{k_n} H^1_{\mathcal F(n)}(\Q,T/\m^{k_n} T)$. It means that $\kappa_n$ vanishes, so Definition \ref{def:order} says that $\ord(\kappa_n)=\infty$.
\end{remark}

\begin{theorem}(\cite[Proposition 4.5.8, theorem 4.5.9 and theorem 5.2.12]{MazurRubin})
Suppose that $(T,\FF,\PP)$ satisfies Assumption \ref{ass:R}, \ref{Hirred}-\ref{Hchev}, \ref{Hcartesian} and \ref{Hprimes}, let $\chi(T,\FF)=1$ and assume $\kappa \in \textrm{KS}(T,\FF, \PP)$ is a primitive Kolyvagin system. For $n\in \NN(\PP)$, if
\[\length\left(H^1_{\FF^*(n)}(\Q,T^*[\m^{k_n}])\right)>k_n,\]
then $\ord(\kappa_n)=\infty$. Otherwise,
\[\ord(\kappa_n)=\length\left(H^1_{\FF^*(n)}(\Q,T^*[\m^{k_n}])\right).\]
Furthermore, the quantities $\partial^{(i)}(\kappa)$ describe the Fitting ideals of the Selmer group:
$$\m^{\partial^{(i)}(\kappa)}=\Fitt_i^{R}\Bigl(H^1_{\FF^*}(\Q,T^*)^\vee\Bigr),$$
where $\Fitt^i_R$ denotes the $i^{\textrm{th}}$ Fitting ideal (see \cite[\textsection 20.2]{Eisenbud} or \cite[Chapter 3]{Northcott})
\label{th:mr459}
\end{theorem}


\begin{remark}
In the case where $\chi(T,\FF)=1$, if $\kappa\in \textrm{KS}(T,\FF, \PP)$ is non-zero but not primitive, then there is some $j\in\N$ such that $\kappa$ generates $\m^j\textrm{KS}(T,\FF, \PP)$. Then $j=\min\{\ord(\kappa_n):\ n\in \NN(\PP)\}$ and 
$$\m^{\partial^{(i)}(\kappa)-j}=\Fitt_i^{R}\Bigl(H^1_{\FF^*}(\Q,T^*)^\vee\Bigr).$$
\end{remark}

\subsection{\texorpdfstring{$\Lambda$}{Lambda}-adic Kolyvagin systems}
\label{sec:lambda}

Let $\Q_\infty$ denote the cyclotomic $\Z_p$-extension of $\Q$ and let $\Q_n$ be the degree $p^n$ subextension of $\Q_\infty$. Let $\Lambda$ denote the Iwasawa algebra
\[\Lambda=\Z_p[[\Gal(\Q_\infty/\Q)]]=\varprojlim_n \Z_p[\Gal(\Q_n/\Q)]\cong \Z_p[[X]].\]

\begin{lemma} (\cite[lemma 5.3.1]{MazurRubin}) Let $\Sigma$ be a finite set of primes containing $p$ and $\infty$ and all the places where $T$ is ramified. Let $\Q_\Sigma$ be the maximal extension of $\Q$ unramified outside $\Sigma$. Then,
\begin{itemize}
    \item $H^1(\Q_\Sigma/\Q,T\otimes_{\Z_p} \Lambda)\cong \varprojlim_{n\in \N} H^1(\Q_\Sigma/\Q_n,T)$,
    \item $H^1(\Q_\ell,T\otimes_{\Z_p} \Lambda)=H^1_\f(\Q_\ell, T\otimes_{\Z_p} \Lambda)\ \forall \ell\notin \Sigma$,
    \item $H^1(\Q,T\otimes_{\Z_p} \Lambda)=H^1(\Q_\Sigma/\Q,T\otimes_{\Z_p} \Lambda)$.
\end{itemize}
\end{lemma}

\begin{definition}
The canonical Selmer structure $\FF_\Lambda$ on $T\otimes_{\Z_p} \Lambda$ is defined by choosing $\Sigma$ as the set formed by $p$, $\infty$ and all the primes where $T$ is ramified, and local conditions 
\[H^1_\FLambda(\Q_\ell,T\otimes \Lambda)=H^1(\Q_\ell,T\otimes \Lambda) \forall \ell\in \Sigma.\]
\end{definition}

\begin{remark}
The canonical Selmer group $H^1_\FLambda(\Q,T\otimes_{\Z_p} \Lambda)$ is the full cohomology group $H^1(\Q,T\otimes \Lambda)$.
\end{remark}

The study of Kolyvagin systems for the representation $T\otimes_{\Z_p} \Lambda$ requires a slight modification of the definition of Kolyvagin systems given in Definition \ref{def:kol}. Following \cite[Definition 3.1.6]{MazurRubin}, we consider 
\[\overline{\KS}(T\otimes_{\Z_p} \Lambda,\FF_\Lambda,\PP):=\varprojlim_{k\in \N}\left(\varinjlim_{j\in N}\KS(T\otimes_{\Z_p} \Lambda/\m^k,\FF_\Lambda,\PP\cap\PP_j)\right).\]

\begin{remark}
For every $\kappa\in \overline{\KS}(T\otimes_{\Z_p} \Lambda,\FF_\Lambda,\PP)$, the leading term $\kappa_1$ is a well-defined class in $H^1_{\FF_\Lambda}(\Q,T\otimes \Lambda)$.
\end{remark}

In order to study the Kolyvagin systems for $(T\otimes \Lambda, \FLambda,\PP)$, we need to assume some extra hypothesis on $T$.
\begin{itemize}
\item\namedlabel{ITam}{(I1)} $(T\otimes \Q_p/\Z_p)^{\II_\ell}$ is $p$-divisible for all $\ell\in \Sigma(\FF)\setminus\{\infty\}$,
\item \namedlabel{Izerop}{(I2)} $H^0(\Q_p,T^*)=0$.
\end{itemize}

\begin{proposition}(\cite[Theorem A]{Buyukboduk10})
Assume $T$ is a $p$-adic Galois representation such that $\chi(T,\Fcan)=1$ and that $(T,\Fcan)$ satisfies \ref{Hirred}-\ref{Hchev}  and \ref{ITam}-\ref{Izerop}. If $\PP$ satisfies \ref{Hprimes}, then the module of Kolyvagin systems $\overline{\KS}(T\otimes \Lambda,\FLambda,\PP)$ is free of rank one over $\Lambda$. In this case, we say that a Kolyvagin system $\kappa\in\overline{\KS}(T\otimes \Lambda,\FLambda,\PP)$ is $\Lambda$-primitive if it generates $\overline{\KS}(T\otimes \Lambda,\FLambda,\PP)$ as a $\Lambda$-module.
\label{prop:lambda_kol_free}
\end{proposition}

\begin{definition}
Let $\kappa$ be a Kolyvagin system in $\overline{\KS}(T\otimes \Lambda, \FLambda, \PP)$. Then 
\[\textrm{Ind}(\kappa):=\textrm{char}\Biggl(\left(\frac{H^1(\Q,T\otimes \Lambda)}{\Lambda\kappa_1}\right)_{\tors}\Biggr),\]
where $\char$ denotes the characteristic ideal of an Iwasawa module.
\end{definition}

\begin{theorem}(\cite[Theorem 5.3.10]{MazurRubin})
Assume that $\chi(T,\FLambda)=1$ and that $T$ satisfies \ref{Hirred}-\ref{Hchev}, \ref{Hcartesian} and \ref{Hprimes}. Let 
\[X_\infty:=\Hom(H^1_{(\FLambda)^*}(\Q,(T\otimes \Lambda)^*),\Q_p/\Z_p)\] 
and let $\kappa\in \overline{\KS}(T\otimes \Lambda,\FLambda,\PP_1)$ whose leading term $\kappa_1$ is nonzero. Then $X_\infty$ is a torsion $\Lambda$-module satisfying that:
\begin{itemize}
\item $\textrm{char}(X_\infty)$ divides $\textrm{Ind}(\kappa)$,
\item $\textrm{char}(X_\infty)=\textrm{Ind}(\kappa)$ if and only if $\kappa$ is $\Lambda$-primitive.
\end{itemize}
\label{th:MC_ind}
\end{theorem}

The following result can be obtained by combining \cite[Theorem A]{Buyukboduk10} and \cite[Proposition 5.2.9]{MazurRubin}.

\begin{theorem}(\cite[Theorem A]{Buyukboduk10})
Assume that $\chi(T,\FLambda)=1$ and that $T$ satisfies \ref{Hirred}-\ref{Hchev} and \ref{ITam}-\ref{Izerop}. If $\PP$ satisfies \ref{Hprimes}, there is a surjective canonical map
\begin{equation}
\Pi: \overline{\KS}(T,\FLambda,\PP)\to \KS(T,\Fcan,\PP).
\label{eq:kol_red}
\end{equation}
\label{th:kol_red}
\end{theorem}

\begin{corollary}
In the setting of Theorem \ref{th:kol_red}, the image of a primitive Kolyvagin system in $\overline{\KS}(T\otimes \Lambda,\FLambda, \PP)$ under the map in \eqref{eq:kol_red} is a primitive Kolyvagin system in $\KS(T,\Fcan,\PP)$.
\label{cor:kol_red_prim}
\end{corollary}

\begin{corollary}
In the setting of Theorem \ref{th:kol_red}, the kernel of the map in \eqref{eq:kol_red} is $(\gamma-1)\overline{\KS}(T\otimes \Lambda,\FLambda,\PP)$, where $\gamma$ is a (topological) generator of $\Gal(\Q_\infty/\Q)$, which acts on $\overline{\KS}(T\otimes \Lambda,\FLambda,\PP)$ by conjugation on the cohomology groups. By Theorems \ref{th:MC_ind} and \ref{th:kol_red}, the Kolyvagin systems in the kernel are those satisfying that
\[\ord_{X\Lambda}(\textrm{Ind}(\kappa))>\ord_{X\Lambda}(X_\infty).\]
\label{cor:kol_red_loc}
\end{corollary}

\subsection{Euler systems}
\label{sec:euler}

Kolyvagin systems are a powerful tool to compute the Selmer groups. However, the main problem of this machinery is the difficulty to construct Kolyvagin systems. The only known way to do it is via Euler systems, when the Selmer structure is the canonical one in $T$. The process will be briefly described in this section.

\begin{definition}
Let $\mathcal K$ be an abelian extension of $\Q$, let $T$ be a $p$-adic Galois representation and let $\PP$ be a set of primes different from $p$. For every $n\in \NN(\PP)$, denote by $\Q(n)$ the maximal $p$-subextension $\Q(\mu_n)/\Q$. Assume that
\begin{itemize}
\item $\mathcal K$ contains $\Q(q)$ for every $q\in \PP$,
\item $\mathcal K$ contains the cyclotomic $\Z_p$-extension.
\end{itemize}
Let $\Omega$ be the set of number fields inside $\mathcal K$. A collection of cohomology classes
$$c=\{c_F\in H^1(F,T):\ F\in \Omega\}$$
is an \emph{Euler system} if whenever $F'/F$ is an extension of fields in $\Omega$, then 
$$\textrm{cor}_{F'/F}(c_{F'})=\left(\prod_{\ell\in \PP\cap\Sigma(F'/F)} P_\ell(\Frob_{\ell}^{-1})\right) c_F$$
where $\Sigma(F'/F)$ is the set of primes in $\PP$ which ramify in $F'/\Q$ but do not ramify in $F/\Q$. Also, recall that $P_\ell(x)=\det(1-\Frob_\ell\, x|T)$.
\label{def:euler_system}
\end{definition}

From an Euler system, there is a Kolyvagin derivative process that produces Kolyvagin systems. The reader is referred to \cite[Appendix A]{MazurRubin} for more details about this construction.
\begin{theorem} (\cite[Theorems 3.2.4 and 5.3.3]{MazurRubin})
Assume that
\begin{itemize}
\item $T/(\Frob_\ell-1) T$ is a cyclic $\OO$-module,
\item $\Frob_\ell^{p^k}-1$ is injective for every $k\in \Z_{\geq 0}$.
\end{itemize}
Then there exists maps 
\[\begin{array}{cc}
    \psi_\infty:\ES(T)\to \KS(T,\Fcan,\PP),\ &\psi_\infty:\ES(T)\to \overline{\KS}(T\otimes \Lambda,\FLambda,\PP),
\end{array}\]
such that the following diagram is commutative:

\begin{center}
    \begin{tikzpicture}[descr/.style={fill=white,inner sep=1.5pt}]
    
            \matrix (m) [
                matrix of math nodes,
                row sep=4em,
                column sep=4em,
                text height=1.5ex, text depth=0.25ex
            ]
            {\ES(T) & \\
            \overline{\KS}(T\otimes \Lambda,\FLambda,\PP) &\KS(T,\Fcan,\PP).\\
            };
    
            \path[overlay,->, font=\scriptsize,>=latex]
            (m-1-1) edge node[auto]{$\psi_\infty$} (m-2-1)
            (m-2-1) edge node[auto]{$\Pi$} (m-2-2)
            (m-1-1) edge node[auto]{$\psi$} (m-2-2);
    \end{tikzpicture}
    \end{center}
Here, $\Pi$ represents the projection map from Theorem \ref{th:kol_red}.
\label{th:eul_to_kol_lambda}
\end{theorem}

\subsection{Localisation of Kolyvagin systems}

According to Theorem \ref{th:KS}, there are no non-trivial Kolyvagin systems when $\chi(T,\FF)=0$. However, we can determine the structure of $H^1_{\FF}(\Q,T)$ by using the modified Selmer structure $\FF^q$, where $q$ is a prime satisfying \ref{Hloc}. Note that $\chi(T,\FF^q)=1$ by Lemma \ref{lem:Hloc_tors} and Remark \ref{rem:core_abc}. For that Selmer structure, we can consider a primitive Kolyvagin system $\kappa\in \textrm{KS}(T,\FF^q, \PP\setminus\{q\})$ and its localisation $\loc^\ell_\s(\kappa_n)$.

\begin{definition}
For every $\ell\in \PP$, fix a generator $\tau_\ell$ of $\GG_\ell$. Since $p^{k_\ell}\mid\#\GG_\ell$, the generator $\tau_\ell$  induces a map $\GG_\ell\to R/\m^{k_\ell}$. 

For every $n\in \NN(\PP)$ and every prime divisor $\ell$ of $n$, then $k_n\leq k_{\ell}$ by the definition of these quantities. Therefore, there is a surjection $R/\m^{k_\ell}\to R/\m^{k_{n}}$ and we can consider $\kappa_n$ as an element in $H^1_{\FF(n)}(\Q,T/\m^{k_n})$. Define
$$\delta_n=\delta_n(\kappa):=\loc^s_q(\kappa_n)\in R/\m^{k_n},$$
where we are using the isomorphism from \ref{Hloc}. Note that each $\delta_n$ is defined only up to multiplication by a unit in $R$ since it depends on the generators chosen.
\label{def:delta}
\end{definition}

For every $n\in \mathcal N=\mathcal N(\mathcal P)$, recall $\nu(n)$ is the number of primes dividing $n$ and let

$$\begin{aligned}
&\ord(\delta_n):=\max\{j\in\N\cup\{0,\infty\}:\delta_n\in \mathfrak{m}^j/\mathfrak{m}^{k_n} \},\\& \partial^{(i)}(\delta):=\min\{\ord(\delta_n): n\in \N,\nu(n)=i\}.
\end{aligned}$$

\begin{definition}
In analogy with the second part of Theorem \ref{th:mr459}, consider the ideals of $R$ defined by
$$\Theta_i(\kappa):=\m^{\partial^{(i)}(\delta)}.$$
\label{def:theta}
\end{definition}

In the rest of the section, we state some lemmas from \cite[\textsection 5.2]{Kim23}. However, the convention of orders in that article is slightly different and the proofs are trickier in our case. For the sake of completeness, we also include them here.

%

\begin{lemma} (\cite[Proposition 5.2]{Kim23})
Let $(T,\FF,\PP)$ be a Selmer triple satisfying Assumptions \ref{ass:R} and \ref{Hirred}-\ref{Hprimes} and let $\kappa\in \KS(T,\FF^q)$. For every $n\in\mathcal N$, let 
$$C_n:=\coker\left(\loc_q^s: H^1_{\mathcal F^q(n)}(\mathbb Q,T/\m^{k_n} T)\to H^1_\s(\mathbb Q_q,T/\m^{k_n} T)\cong R/\m^{k_n}\right).$$
If $\ord(\kappa_n)+\length(C_n)<k_n$, then 
$$\ord(\kappa_n)+\length(C_n)= \ord(\delta_n).$$
Otherwise, $\ord(\delta_n)$ is infinite, i.e., $\delta_n=0$.
\label{lem:ordkappadelta}
\end{lemma}

\begin{proof}

Since $\chi(T/\m^{k_n},\mathcal F^q)=1$, we have a non-canonical isomorphism
\begin{equation}
H^1_{\mathcal F^q(n)}(\Q,T/\m^{k_n})\cong R/\m^{k_n}\oplus H^1_{(\mathcal F^*)_q(n)}(\Q,T^*[\m^{k_n}]).
\label{eq:core1}
\end{equation}
Hence we can find an element $z\in H^1_{\FF^q(n)}(\Q,T/\m^{k_n} T)$ and a submodule $A$, isomorphic to $H^1_{(\mathcal F^*)_q(n)}(\Q,T^*[\m^{k_n}])$, such that 
\begin{equation}
H^1_{\mathcal F^q(n)}(\Q,T/\m^{k_n} T)=(R/\m^{k_n})\, z\oplus A.
\label{eq:split}
\end{equation}

By Theorem \ref{th:mr459},
\begin{equation}
\ord(\kappa_n)\geq \length\Bigl(H^1_{(\mathcal F^*)_q(n)}(\Q,T^*[\m^{k_n}])\Bigr),
\label{eq:ordk2}
\end{equation}
so the component of $\kappa_n$ in $A$ in the splitting in \eqref{eq:split} vanishes. Then $\kappa_n\in (R/\m^{k_n})\, z$. In particular, there is some unit $u\in R^\times$ and some generator $\pi$ of $\m$ such that
\begin{equation}
\kappa_n=u\pi^{\ord(\kappa_n)}z.
\label{eq:kn_ord}
\end{equation}

Assume first that $\delta_n:=\loc_q^\s(\kappa_n)\neq 0$. Then $\pi^{\ord(\kappa_n)}\loc_q^\s(z)\neq 0$, so using \eqref{eq:ordk2} we obtain that
\[\length\Biggl(\loc_q^s\Bigl(H^1_{\FF^q(n)}(\Q,T/\m^{k_n} T)\Bigr)\Biggr)\geq \ord(\kappa_n)\geq  \length\Bigl(H^1_{(\mathcal F^*)_q(n)}(\Q,T^*[\m^{k_n}])\Bigr).\]

Since $H^1_s(\Q_q,T/\m^{k_n} T)$ is isomorphic to $R/\m^{k_n}$ by Lemma \ref{lem:Hloc_tors}, \eqref{eq:ordk2} and the fact that $\pi^{\ord(\kappa_n)}\loc_q^\s(z)\neq 0$ imply that $\loc_q^s(z)$ generates the image of $\loc_q^\s$, i.e,
$$\loc_q^s\left(H^1_{\mathcal F^q(n)}(\Q,T/\m^{k_n} T)\right)=(R/\m^{k_n})\loc_q^s\left(z\right).$$
By equation \eqref{eq:kn_ord},
\begin{equation}
p^{\ord(\kappa_n)}\loc_q^s\left(H^1_{\mathcal F^q(n)}(\Q,T/\m^{k_n} T)\right)=(R/\m^{k_n})\delta_n.
\label{eq:modules}
\end{equation}
Since $H^1_s(\Q_q,T/\m^{k_n} T)$ is a cyclic $R/\m^{k_n}$-module, computing the length of each module in \eqref{eq:modules} leads to the following equality:
$$\length\left(\loc_q^s(H^1_{\mathcal F^q(n)}(\Q,T/\m^{k_n} T))\right)-\ord(\kappa_n)=k_n-\ord(\delta_n).$$
By definition, $\length(C_n)=k_n-\length\left(\loc_q^s(H^1_{\mathcal F^q(n)}(\Q,T/\m^{k_n} T))\right)$. Therefore,
$$\ord(\kappa_n)+\length(C_n)=\ord(\delta_n).$$
Note that, in this situation, since $\delta_n\neq 0$, its order is at most $k_n-1$, so
\[\ord(\kappa_n)+\length(C_n)<k_n.\]

Now assume that $\delta_n:=\loc_q^s(\kappa_n)=0$. Using \eqref{eq:kn_ord}, $\pi^{\ord(\kappa_n)}\loc_q^s(z)=0$, so we can conclude that
$$\length\biggl((R/\m^{k_n})\loc_q^s(z)\biggr)\leq \ord(\kappa_n).$$
Note that, since $H^1_s(\Q_q,T/\m^{k_n} T)$ is a free $R/\m^{k_n}$-module of rank one, we have that
\[\begin{aligned}&\length\left(\loc_q^s\left(H^1_{\mathcal F^q(n)}(\Q,T/\m^{k_n} T)\right)\right)=\\ &\max\Biggl\{\length\biggl(R/\m^{k_n}\loc_q^s(z)\biggr),\length\biggl(\loc_q^s(A)\biggr)\Biggr\}.\end{aligned}\]
However, by \eqref{eq:ordk2}, 
\[\length\biggl(\loc_q^\s(A)\biggr)\leq \length(A)\leq \ord(\kappa_n).\]
Hence,
$$\length\left(\loc_q^s\left(H^1_{\mathcal F^q(n)}(\Q,T/\m^{k_n} T)\right)\right)\leq\ord(\kappa_n).$$
Therefore, 
\[\ord(\kappa_n)+\length(C_n)\geq k_n.\qedhere\]
\end{proof}

The following lemma is proven in \cite[lemma 5.3]{Kim23} for the particular case where $T$ is the Tate moudule of an elliptic curve and $\FF$ is the Bloch-Kato Selmer structure.

\begin{lemma}
Let $(T,\FF,\PP)$ be a Selmer triple satisfying Assumptions \ref{ass:R} and \ref{Hirred}-\ref{Hprimes} and let $\kappa\in \KS(T,\FF^q,\PP)$ be a primitive Kolyvagin system. Then there is some $n\in \NN$ such that $\delta_n\in R^\times$.
\label{lem:dinf}
\end{lemma}

\begin{proof}
By Proposition \ref{prop:local_duality}, for every $n\in \NN$ there is an exact sequence
\begin{equation}
\xymatrix{0\ar[r]&H^1_{(\mathcal F^*)_q(n)}(\Q,T^*[\m^{k_n}])\ar[r] & H^1_{\mathcal F^*(n)}(\Q,T^*[\m^{k_n}])\ar[r] & C_n\du\ar[r] & 0,}
\label{eq:dualseq}
\end{equation}
where $C_n$ was defined in Lemma \ref{lem:ordkappadelta}

By Proposition \ref{prop:core_vertex}, there is some $n_0\in \NN$ such that 
$$H^1_{\mathcal F^*(n_0)}(\Q,T^*)=0.$$

By Lemma \ref{lem:mr353},
\[H^1_{\FF^*(n_0)}(\Q,T^*[\m^{k_{n_0}}])=H^1_{\mathcal F^*(n_0)}(\Q,T^*)[\m^{k_{n_0}}]=0.\]

Therefore, \eqref{eq:dualseq} implies that $C_{n_0}=0$. Then Theorem \ref{th:mr459} and Lemma \ref{lem:ordkappadelta} imply that $\ord(\delta_{n_0})=\ord(\kappa_{n_0})=0$, which means that $\delta_n\in R^\times$.
\end{proof}

\begin{lemma}
Let $(T,\FF,\PP)$ be a Selmer triple satisfying Assumptions \ref{ass:R} and \ref{Hirred}-\ref{Hprimes} and let $\kappa\in \KS(T,\FF^q,\PP)$ be a primitive Kolyvagin system. For every $n\in\mathcal N$, we have that
$$\ord(\delta_n)=\length(H^1_{\mathcal F(n)}(\mathbb Q,T/\m^{k_n} T))$$
if the latter is less than $k_n$. Otherwise, $\ord(\delta_n)=\infty$.
\label{lem:ord_delta}
\end{lemma}

\begin{proof}
Since $\chi(T,\FF)=0$ by \ref{Hcore0}, Theorem \ref{th:core_rank}, \eqref{eq:dualseq} and Theorem \ref{th:mr459} imply that
\begin{equation}
\begin{aligned}\length\left(H^1_{\mathcal F(n)}(\mathbb Q,T/\m^{k_n} T)\right)=\length\left(H^1_{\mathcal F^*(n)}(\mathbb Q,T^*[\m^{k_n}])\right)=\\
\length\left(H^1_{(\mathcal F^*)_q(n)}(\mathbb Q,T^*[\m^{k_n}])\right)+\length(C_n)=\ord(\kappa_n)+\length(C_n).\end{aligned}
\label{eq:dual_lengths}
\end{equation}

If $\length\left(H^1_{\mathcal F(n)}(\mathbb Q,T/\m^{k_n} T)\right)<k_n$, then Lemma \ref{lem:ordkappadelta} implies that
$$\length\left(H^1_{\mathcal F(n)}(\mathbb Q,T/\m^k T)\right)=\ord(\delta_n).$$

Assume now that $\length(H^1_{\mathcal F(n)}(\mathbb Q,T/\m^{k_n} T))\geq k_n$. By Theorem \ref{th:mr459} and \eqref{eq:dual_lengths},
\[\ord(\kappa_n)+\length(C_n)\geq k_n.\]

By the second part of Lemma \ref{lem:ordkappadelta}, $\ord(\delta_n)=\infty$.
\end{proof}

\subsection{Fitting ideals of Selmer groups of core rank zero}
\label{sec:fitd}

Throughout this section, the concept of rank of a finitely generated module will play a role. With the aim of avoiding confusion, especially when $R$ is artinian, we clarify this concept in the following definition.

\begin{definition}
Let $M$ be an $R$-module. We define the rank of $M$ as the maximum number of linearly independent elements in $M$. In case that $R$ is local, principal and artinian of length $k$, this quantity coincides with the following:
\[\rank_R(M)=\dim_{R/\m}\m^{k-1} M.\]
\label{def:rank}
\end{definition}

The main purpose of this section is to prove the following theorem, whose proof will appear at the end of the section.

\begin{theorem}
Assume $R$ is a discrete valuation ring or an artinian, local, principal ring. Let $(T,\FF,\PP)$ be a Selmer triple satisfying \ref{Hirred}-\ref{Hprimes}. For every $i\in \Z_{\geq0}$, we have that
\begin{equation}
\Theta_i(\kappa)\subset\Fitt_i^R\left(H^1_{\mathcal F^*}(\mathbb Q, T^*)\du\right).
\label{eq:theta}
\end{equation}
where $\Theta_i(\kappa)$ was defined in Definition \ref{def:theta}. Moreover, if one of the following conditions is satisfied:
\begin{enumerate}[(i)]
\item $i=\rank_R(H^1_{\FF^*}(\Q,T^*)\du)=:r,$
\item $\Theta_{i-1}(\kappa)\subsetneq \Fitt_{i-1}^R\left(H^1_{\mathcal F^*}(\mathbb Q, T^*)\du\right),$
\item There is some $k\in \N$ and some $n\in\mathcal N$ such that $\nu(n)=i-1$, $\Theta_{i-1}(\kappa)=\delta_n R$ and 
$$H^1_{\mathcal F(n)}(\Q,T/\m^k T)\cong R/\m^{e_1}\times \cdots \times R/\m^{e_s}$$
for some $e_1>e_2\geq \cdots\geq e_s$.
\end{enumerate}
then we have the equality $\Theta_i(\kappa)=\Fitt_i^R\left(H^1_{\mathcal F^*}(\mathbb Q, T^*)\du\right)$.
\label{th:kur_par}
\end{theorem}

\begin{remark}
Assume that either
\begin{itemize}
\item $R$ is a discrete valuation ring, or
\item $\length(R)\geq \length\left(H^1_\FF(\Q,T)_{\tors}\right)$.
\end{itemize}
Then $\rank_R(H^1_{\FF^*}(\Q,T^*)\du)$ is the minimal $i$ such that $\Theta_i\neq 0$.
\end{remark}

Provided that we know the ideals $\Theta_i(\kappa)$, Theorem \ref{th:kur_par} determines the Fitting ideals of the dual Selmer group. That is enough to determine the Selmer groups up to isomorphism.
\begin{corollary}
Assume $R$ is a discrete valuation ring and write $\Theta_i(\kappa)=\mathfrak m^{n_i}$. Then
$$\Fitt_i^R\left(H^1_{\mathcal F^*}(\mathbb Q, T^*)\du\right)=\mathfrak m^{\min\left\{n_i,\frac{n_{i+1}+n_{i-1}}{2}\right\}}.$$
\label{cor:kur_par}
\end{corollary}

\begin{proof}
Write $\Fitt_i^R\left(H^1_{\mathcal F^*}(\mathbb Q, T^*)\du\right)=\m^{m_i}$, so the inequality in Theorem \ref{th:kur_par} implies that $n_i\geq m_i$. By the structure theorem of finitely generated modules over principal ideal domains, the following inequality holds for $i\in \N$:

\[m_{i+1}-m_i\geq m_i-m_{i-1}.\]
Hence the index $m_i$ can be upper bounded using $m_{i-1}$ and $m_{i+1}$:
\begin{equation}
m_i\leq \frac{m_{i+1}+m_{i-1}}{2}.
\label{eq:Fitting_ineq}
\end{equation}

Assume that $n_i=m_i$. Then 
\[m_i\leq \frac{m_{i+1}+m_{i-1}}{2}\leq \frac{n_{i+1}+n_{i-1}}{2}.\]
Therefore,
\[m_i=n_i=\min\left\{n_i,\frac{n_{i+1}+n_{i-1}}{2}\right\}.\]

Assume that $n_i> m_i$. Theorem \ref{th:kur_par} can be applied to $i+1$. Since the second condition for the equality stated in this theorem holds by our assumption, we obtain that $m_{i+1}=n_{i+1}$. On the other hand, assume by contradiction that $n_{i-1}>m_{i-1}$. In this case, Theorem \ref{th:kur_par}, would imply that $n_i=m_i$, contradicting our assumption. Therefore, $n_{i-1}=m_{i-1}$. Moreover, condition (iii) in Theorem \ref{th:kur_par} cannot be satisfied since, otherwise, our assumption would not be true. Hence, the equality holds in equation \eqref{eq:Fitting_ineq} and we obtain
\[m_i=\frac{m_{i+1}+m_{i-1}}{2}=\frac{n_{i+1}+n_{i-1}}{2}=\min\left\{n_i,\frac{n_{i+1}+n_{i-1}}{2}\right\}.\qedhere\]
\end{proof}

\begin{proof}[Proof of Theorem \ref{th:kur_par}]


If $R$ is artinian, choose $k=\length(R)$. If $R$ is a discrete valuation ring, we can choose some $k\in \N$ such that $\PP_k\subset \PP$ and
\begin{equation}
k\geq \length(H^1_{\mathcal F^*}(\mathbb Q, T^*)\du_{\tors}).
\label{eq:k_cond}
\end{equation}

 It is enough to study the Fitting ideals of $H^1_{\FF^*}(\Q,T^*[\m^k])\du$. Indeed, when $R$ is a discrete valuation ring, let $\alpha_i\in \Z_{\geq 0}$ be such that $\Fitt_i^{R}\left(H^1_{\mathcal F^*}(\mathbb Q, T^*)\du\right)=\mathfrak m^{\alpha_i}$. By Lemma \ref{lem:mr353}, 
\[\Fitt_i^{R/\m^k}\left(H^1_{\mathcal F^*}(\mathbb Q, T^*[\m^k])\right)=\Fitt_i^{R/\m^k}\left(H^1_{\mathcal F^*}(\mathbb Q, T^*)\du\otimes_R R/\m^k\right)=\mathfrak m^{\min\{k,\alpha_i\}}.\]
Since $k$ has been chosen satisfying \eqref{eq:k_cond}, then $\min\{k,\alpha_i\}=k$ if and only if $\m^{\alpha_i}=0$.

For every $n\in\NN(\PP_K)$, consider the exact sequence
$$\xymatrix{0\ar[r] & H^1_{\mathcal F_n}(\mathbb Q,T/\m^{k} T)\ar[r] & H^1_{\mathcal F}(\Q,T/\m^{k} T)\ar[r] & \bigoplus_{\ell\mid n} H^1_\f(\Q_\ell,T/\m^{k} T)}.$$
Note the last term is a free $R/\m^{k}$-module of rank $\nu(n)$ by Corollary \ref{cor:local_coh_cyclic}. Since $\chi(T/\m^k,\mathcal F)=0$, we have that
$$\min\{k,\alpha_i\}\leq \length\left(H^1_{\mathcal F_n}(\mathbb Q,T/\m^{k} T)\right)\leq \length\left( H^1_{\mathcal F(n)}(\mathbb Q,T/\m^{k} T)\right).$$
since $H^1_{\mathcal F_n}(\mathbb Q,T/\m^{k} T)\subset H^1_{\mathcal F(n)}(\mathbb Q,T/\m^{k} T)$. By Lemma \ref{lem:ord_delta}, for every $n\in \NN$ such that $\nu(n)=i$, we then have that
$$\ord(\delta_n)\geq \alpha_i$$
Equivalently, $\delta_n\in \mathfrak m^{\alpha_i}$. Therefore,
$$\Theta_i(\kappa)\subset \Fitt_i^{R/\m^k}\left(H^1_{\mathcal F^*}(\mathbb Q, T^*[\m^k])\du\right).$$

To prove the equality in \eqref{eq:theta} for some $i\in \mathbb Z_{\geq 0}$, we need to find some $n\in\mathcal N_k$ such that $\nu(n)=i$ and $\length(H^1_{\mathcal F(n)}(\mathbb Q,T))=\alpha_i$.

If $\Fitt_i^{R/\m^k}\left(H^1_{\mathcal F^*}(\mathbb Q, T^*[\m^k])\du\right)=0$, then $\alpha_i=k$ and the result is clear. So let $i$ be the rank of $H^1_{\mathcal F}(\Q,T/\m^k T)\cong H^1_{\mathcal F^*}(\Q,T^*[\m^k])$ as an $R/\m^k$-module. 

 Using an inductive application of Corollary \ref{cor:loc_sur}, we can choose primes $\ell_1,\ldots, \ell_i$ such that the maps
$$H^1_{\mathcal F}(\mathbb Q,T/\mathfrak m^k T)\to \bigoplus_{k=1}^i H^1_\f(\mathbb Q_{\ell_i},T/\mathfrak m^k T),\ \ 
H^1_{\mathcal F^*}(\mathbb Q,T^*[\mathfrak m^k])\to \bigoplus_{k=1}^i H^1_\f(\mathbb Q_{\ell_i},T^*[\mathfrak m^k])$$
are surjective. By Lemma \ref{lem:mr417}, for $n_r:=\ell_1\cdots\ell_i$ we have that 
\[H^1_{\mathcal F^*(n_r)}(\mathbb Q,T^*[\mathfrak m^k])=H^1_{(\mathcal F^*)_{n_r}}(\mathbb Q,T^*[\mathfrak m^k]).\] 
Thus
$$\ord(\delta_{n_r})=\length\left(H^1_{\mathcal F^*({n_r})}(\mathbb Q,T^*[\mathfrak m^k])\right)=\length\left( H^1_{(\mathcal F^*)_{n_r}}(\mathbb Q,T^*[\mathfrak m^k])\right)=\alpha_i,$$
where the last equality comes from the structure theorem of finitely generated $R/\m^k$-modules. Therefore, the equality holds for $i=r$.

For every $i>r$, construct $n_i\in \NN$ and $h_i\in \N$ inductively as follows. Note that $n_r$ was already constructed and let $h_r$ be the exponent of $H^1_{\FF^*(n_r)}(\Q,T^*[\m^k]).$

Assume that $n_i\in \mathcal N$ was already constructed satisfying that $\nu(n_i)=i$.

Write the structure of the Selmer group as
$$H^1_{\mathcal F(n_i)}(\Q, T)\cong H^1_{\mathcal F^*(n_i)}(\Q,T^*[\m^k])\cong R/\mathfrak m^{h_i}\times R/\m^{e_{i+2}}\times \cdots\times R/\m^{e_{s}}$$
for some integers $e_{i+2},\ldots, e_s$. 

By Lemma \ref{lem:cr0}, there are infinitely many primes $\ell_{i+1}\in \PP_k$ satisfying that
$$H^1_{\mathcal F^*(n\ell_{i+1})}(\Q,T^*[\m^k])\cong R/\m^{h_{i+1}}\times R/\m^{e_{i+3}}\times \cdots \times R/\m^{e_s}$$
for $e_{i+2}\leq h_{i+1}\leq k$. We can choose the prime $\ell_{i+1}$ minimising $h_{i+1}$ and define $n_{i+1}:=n_i\ell_{i+1}$.

By Lemma \ref{lem:ord_delta},
$$\delta_{n_i}R=\m^{h_i}\prod_{j=i+2}^s \m^{e_j}\subset\prod_{j=i+1}^s \m^{e_j}=\Fitt_{i+1}^{R/\m^k}\left(H^1_{\mathcal F^*}(\Q,T^*[\m^k])\right).$$
When $h_{i}= e_{i+1}$,
$$\Theta_{i+1}(\kappa)=\Fitt_{i+1}^{R/\m^k}\left(H^1_{\mathcal F^*}(\Q,T^*[\m^k])\right).$$

When $h_i>e_{i+1}\geq e_{i+2}$ or $h_i\geq e_{i+1}>e_{i+2}$, Lemma \ref{lem:cr0} implies the existence $\ell_{i+1}$ such that $h_{i+1}=e_{i+2}$.

If the assumption $\Theta_i(\kappa)\subsetneq \Fitt_i^R(H^1_{\FF^*}(\Q,T^*[\m^k]))$ holds, then $h_i>e_{i+1}$, so $h_{i+1}=e_{i+2}$. 

Same result can be obtained assuming hypothesis (iii). In this case, $e_{i+1}>e_{i+2}$ and we obtain that 
\[\Theta_{i+1}(\kappa)=\Fitt_{i+1}^R(H^1_{\FF^*}(\Q,T^*[\m^k])).\qedhere\]
\end{proof}

\subsection{Non-self-dual case}
\label{sec:fitn}

In the case when $T$ is not residually self-dual, we can improve the previous result to get the equality in \eqref{eq:theta} for every $i\in \Z_{\geq 0}$. More precisely, assume the following extra assumptions:
\begin{itemize}
\item\namedlabel{Nsd}{(N1)} $\Hom_{\F_p[G_\Q]}(T/\m T,T^*[\m])=0$,
\item\namedlabel{Nsur}{(N2)} The image of the homomorphism $R\to \textrm{End}(T)$ is contained in the image of $\Z_p[[G_\Q]]\to \textrm{End}(T)$.
\end{itemize}

Under those assumptions, the following improvement of Theorem \ref{th:kur_par} is true. Sakamoto proves the equality \eqref{eq:theta_nd} below under stronger assumptions when the coefficient ring $R$ is a Gorenstein ring of dimension zero. In particular, R. Sakamoto's result only worked when $H^1_{\FF}(\Q,T)=0$. However, when $R$ is a principal ring, we can weaken the assumptions to obtain the following result.

\begin{theorem}
Let $R$ be either a discrete valuation ring or a principal, artinian, local ring and let $(T,\FF,\PP)$ be a Selmer triple satisfying \ref{Hirred}-\ref{Hprimes} and \ref{Nsd}-\ref{Nsur}. Then for every $i\in \Z_{\geq 0}$, the following equality is satisfied:
\begin{equation}
\Theta_i(\kappa)=\Fitt_i^R\left(H^1_{\mathcal F^*}(\mathbb Q, T^*)\du\right).
\label{eq:theta_nd}
\end{equation}
\label{th:kur}
\end{theorem}

The proof of Theorem \ref{th:kur} will be presented at the end of this section. The reason for assuming \ref{Nsd}-\ref{Nsur} is that we can apply the following result.

\begin{proposition}(\cite[Proposition 3.6.2]{MazurRubin})
Assume that $R$ and $T$ satisfy Assumption \ref{ass:R}, \ref{Hirred}-\ref{Hchev} and \ref{Nsd}-\ref{Nsur}. Let $C\subset H^1(\Q,T/\m^k)$ and $D\subset H^1(\Q,T^*[\m^k])$ be finite submodules and choose some homomorphisms
$$\begin{array}{cc}
\phi:\ C\to R/\m^k ,\ \ &\psi:D\to R/\m^k.
\end{array}$$
There exists a set $S\subset \PP_k$ of positive density such that, for all $\ell\in S$,
$$\begin{aligned}
C\cap \ker\left[\loc_\ell:\ H^1(\Q,T/\m^k)\to H^1(\Q_\ell,T/\m^k)\right]=\ker(\phi),\\
 D\cap \ker\left[\loc_\ell:\ H^1(\Q,T^*[\m^k])\to H^1(\Q_\ell,T^*[\m^k])\right]=\ker(\psi).
 \end{aligned}$$
\label{prop:mr362}
\end{proposition}

The proof of Theorem \ref{th:kur} is based on the following lemma.

\begin{lemma}
Let $R$ is either a discrete valuation ring or a principal, artinian, local ring and let $(T,\FF,\PP)$ be a Selmer triple satisfying \ref{Hirred}-\ref{Hprimes} and \ref{Nsd}-\ref{Nsur}. Assume that
$$H^1_{\FF}(\Q,T)\cong H^1_{\FF^*}(\Q,T^*)\cong R/\m^{e_1}\times \cdots\times R/\m^{e_s}$$
for some $e_1\geq \ldots \geq e_s\in\N$. Then there are infinitely many primes $\ell\in \PP_k$ such that 
$$H^1_{\FF(\ell)}(\Q,T)\cong H^1_{\FF^*(\ell)}(\Q,T^*)\cong R/\m^{e_2}\times \cdots\times R/\m^{e_s}.$$
\label{lem:nd}
\end{lemma}

\begin{proof}
Without loss of generality we can assume $\length(R)>e_1$, since the result otherwise follows from Lemma \ref{lem:cr0}. Then  we can choose some $k\in \N$ such that $\length(R)\geq k>e_1$ and $\PP_k\subset \PP$.

By Corollary \ref{cor:loc_sur}, we can find an auxiliary prime $b\in \PP_k$ such that the localisation maps
\[\begin{aligned}
&\loc_b:\ H^1_{\FF}(K,T/\m^kT)\to H^1_\f(K_b,T/\m^kT)[\m^{e_1}],\\ &\loc_b:\ H^1_{\FF^*}(K,T^*[\m^k])\to H^1_\f(K_b,T^*[\m^k])[\m^{e_1}]
\end{aligned}\]
are surjective. Hence
\[H^1_{(\FF^*)_b}(K,T^*[\m^k])\cong R/\m^{e_2}\times \cdots\times R/\m^{e_s}.\]
Since $\chi(T/\m^k T\FF^b)=1$ by Proposition \ref{prop:core_abc}, Theorem \ref{th:core_rank} implies that
\begin{equation}
H^1_{\FF^b}(K,T/\m^k T)\cong R/\m^k\times R/\m^{e_2}\times \cdots\times R/\m^{e_s}.
\label{eq:str_Fb}
\end{equation}
By Proposition \ref{prop:mr362}, we can find a prime $\ell$ satisfying the following:
\begin{itemize}
\item The kernel of the localisations $\loc_\ell$ and $\loc_b$ defined on the dual Selmer group $H^1_{\FF^*}(\Q,T^*[\m^k])$ are the same. Following Lemma \ref{lem:mr121}, we can define non-canonical isomorphisms
\begin{equation}
H^1_\f(K_b,T^*[\m^k])\cong H^1_\f(K_\ell,T^*[\m^k])\cong T^*/(\m^k,\tau-1) T^*\cong R/\m^k.
\label{eq:fin_iso}
\end{equation}
Under this isomorphism, we can understand $\loc_b$ and $\loc_\ell$ as elements in the dual space $\Hom(H^1_{\FF^*}(\Q,T^*[\m^k]),R/\m^k)$. In this setting, the above condition implies that there exists a unit $u\in R^\times$ such that $\loc_\ell= u\loc_b$.

\item The localisation map 
$$\loc_\ell:\ H^1_{\FF^b}(K,T/\m^k T)\to H^1_\f(K_\ell,T/\m^k T)$$ 
is surjective.

The first condition implies that
\[H^1_{(\FF^*)_{b\ell}}(K,T^*[\m^k])=H^1_{(\FF^*)_b}(K,T^*[\m^k])\cong R/\m^{e_2}\times \cdots\times R/\m^{e_s}.\]
\end{itemize}
Since $\chi(T/\m^k T,\FF^{b\ell})=2$ by Proposition \ref{prop:core_abc}, then Theorem \ref{th:core_rank} gives an isomorphism
$$H^1_{\FF^{b\ell}}(K,T/\m^k T)\cong (R/\m^k)^2\times R/\m^{e_2}\times \cdots\times R/\m^{e_s}.$$

By Proposition \ref{prop:global_duality}, we can consider the following exact sequence
$$\xymatrix{H^1_{\FF^{b\ell}}(K,T/\m^k T)\ar[r] & H^1_{\s}(K_b, T/\m^k T)\oplus H^1_{\s}(K_\ell, T/\m^k T) \ar[r] & H^1_{\FF^*}(K,T^*[\m^k])^\vee.}$$
Dualising the isomorphisms in \eqref{eq:fin_iso}, we obtain an isomorphism
\begin{equation}
H^1_\s(K_b,T/\m^k T)\oplus H^1_\s(K_\ell,T/\m^k T)\cong (R/\m^k)^2
\label{eq:sing_iso}
\end{equation}
such that the element $(1,-u^{-1})$ belongs to the kernel of the second map. Therefore, there is an element $z\in H^1_{\FF^{b\ell }}(\Q,T/\m^k T)$ such that $\loc_b(z)=1$ and $\loc_\ell(z)=-u^{-1}$, under the identifications in \eqref{eq:sing_iso}.

It implies that the relaxed Selmer group splits as follows:
\begin{equation}
H^1_{\FF^{b\ell}}(K,T/\m^k T)=(R/\m^k) z\oplus H^1_{\FF^b}(K,T/\m^k T)=(R/\m^k) z\oplus H^1_{\FF^\ell}(K,T/\m^k T).
\label{eq:rel_ql_split}
\end{equation}

We want to show now that 
$$\Pi_\ell\circ\loc_\ell:\ H^1_{\FF^\ell}(K,T/\m^k T)\to H^1_\f(K_\ell,T/\m^k T),$$ 
where $\Pi_\ell:\ H^1(K_\ell,T/\m^k T)\to H^1_{\f}(K_\ell,T/\m^k T)$ is the projection arising from the splitting in Lemma \ref{lem:local_split}, is surjective. 

Indeed, let $x\in H^1_{\FF^\ell}(K,T/\m^k T)$ be such that $\pi^{k-1}x\neq 0$, where $\pi $ is a generator of $\m$. By \eqref{eq:rel_ql_split}, there is a unique decomposition $x=\alpha z+\beta$, where $\alpha\in R/\m^k$ and $\beta\in H^1_{\FF^q}(K,T/\m^k T)$. Since
$$H^1_{\FF^\ell}(K,T/\m^k T)\cong R/\m^k \times R/\m^{e_2}\times \cdots\times R/\m^{e_s}$$
and 
$$H^1_{\FF}(K,T/\m^k T)\cong R/\m^{e_1} \times R/\m^{e_2}\times \cdots\times R/\m^{e_s},$$
 then 
 \[\pi^{k-e_1}x=\pi^{k-e_1}\alpha z+\pi^{k-e_1}\beta\in \pi^{k-e_1} H^1_{\FF^\ell}(K,T/\m^k T)\subset H^1_\FF(K,T/\m^k T)\subset H^1_{\FF^b}(K,T/\m^k).\] 
 By \eqref{eq:rel_ql_split}, that implies that $\pi^{k-e_1}\alpha=0$, which means that $\alpha\in \m^{e_1}/\m^k$. Then, since $\pi^{k-1}\alpha=0$ and $\pi^{k-1}x\neq 0$, we get that $\pi^{k-1} \beta\neq 0$. By the assumptions on the prime $\ell$,
 \[\loc_\ell\Bigl(H^1_{\FF^b}(K,T/\m^k T)\Bigr)=H^1_\f(K_\ell, T/\m^k T),\]
  so the isomorphism in \eqref{eq:str_Fb} implies that $\loc_\ell(\beta )$ generates $H^1_\f(K_\ell, T/\m^k T)$. Indeed, every element of $H^1_{\FF^b}(K,T/\m^k T)$ is a linear combination of $\beta$ and elements of $\m^{e_2}$-torsion, so $\loc_\ell$ would only be surjective when $\loc_\ell(\beta)$ generates the whole $H^1_\f(K_\ell, T/\m^k T)$.

We have that
\[(\Pi_\ell\circ\loc_\ell)(x)-(\Pi_\ell\circ\loc_\ell)(\beta)= \alpha(\Pi_\ell\circ\loc_\ell)(z)\in \m^{e_1} H^1_\f(K_\ell, T/\m^k T).\]
Then $(\Pi_\ell\circ\loc_\ell)(x)$ generates $H^1_\f(K_\ell, T)$ and thus 
\[\Pi_\ell\circ\loc_{\ell}:\ H^1_{\FF^\ell}(K,T/\m^k T)\to H^1_\f(K_\ell, T/\m^k T)\]
is surjective. Note that 
\[H^1_{\FF(\ell)}(K,T/\m^k T)=\ker\left(\Pi_\ell\circ\loc_\ell:H^1_{\FF^\ell}(K,T/\m^k)\to H^1_\f(K_\ell,T/\m^k)\right).\]
Then the structure theorem implies that
\[H^1_{\FF(\ell)}(K,T/\m^k T)\cong R/\m^{e_2}\times\cdots\times R/\m^{e_s}.\qedhere\]
\end{proof}

\begin{proof}[Proof of Theorem \ref{th:kur}]

Assume that
$$H^1_{\FF^*}(\Q,T^*)\du\cong R^r\times R/\m^{e_1}\times \cdots\times R/\m^{e_s}.$$

By Theorem \ref{th:kur_par},
$$\Theta_i(\kappa)\subset\Fitt_i^R\left(H^1_{\mathcal F^*}(\mathbb Q, T^*)\du\right).$$

By Lemma \ref{lem:ord_delta}, for every $i\in \Z_{\geq0}$ we just need to find some $n\in \NN$ such that $\nu(n)=i$ and $\Fitt_i^R\left(H^1_{\mathcal F^*}(\mathbb Q, T^*)\du\right)=\m^{\lambda(n)}$, where $\lambda(n)=\length(H^1_{\FF^*(n)}(\Q,T^*[\m^{k_n}]))$. For every $i\in\{0,\ldots,s\}$, we will construct some $n_i\in \NN(\PP_k)$ such that $\nu(n_i)=i$ and
$$H^1_{\FF^*(n_i)}(\Q,T^*)\du\cong  R/\m^{e_{i+1}}\times \cdots\times R/\m^{e_s}.$$
The process for constructing the $n_i$ was described in the proof of Theorem \ref{th:kur_par} and, under assumptions \ref{Nsd} and \ref{Nsur}, Lemma \ref{lem:nd} guarantees that $h_i=e_{i+1}$ for all $i$, so the Selmer group for $\FF^*(n_i)$ has the desired structure.
\end{proof}

\section{Selmer group of an elliptic curve}
\label{sec:EC}

\subsection{Main results}
\label{sec:EC_intro}

The aim of this section is to apply the results from the previous one to compute the Galois structure of the classical Selmer group of an elliptic curve (defined over $\Q$) over certain abelian extensions.

Throughout this section, let $E/\Q$ be an elliptic curve defined over $\Q$ and let $p\geq 5$ be a prime number. Denote by $N$ the conductor of $E$ and by $T$ the Tate module of $E$.

Assume the following hypotheses:
\begin{itemize}
\item \namedlabel{ESur}{(E1)} The homomorphism $\rho:\ G_\Q\to \textrm{Aut}(T)$ induced by the Galois action is surjective.
\item \namedlabel{EManin}{(E2)} Either the Manin constant $c_0(E)$ or $c_1(E)$ is prime to $p$.
\end{itemize}

Assumption \ref{ESur} is necessary to ensure that $T$ satisfies the assumptions \ref{Hirred}-\ref{Hprimes}. Assumption \ref{EManin} is necessary to guarantee the integrality of the modular symbols, which will be defined below.

By the modularity theorem, there exist morphisms
\[\begin{array}{cc}
\varphi_0:\ X_0(N)\to E,\ &\varphi_1:\ X_1(N)\to E,
\end{array}\]
where $X_0(N)$ and $X_1(N)$ are the modular curves associated with the groups $\Gamma_0(N)$ and $\Gamma_1(N)$, respectively (see \cite[Definition 1.2.1]{DiamondShurman}). We may choose $\varphi_0$ and $\varphi_1$ of minimal degree.

The modularity theorem also proves the existence of a newform $f$ of weight $2$ and level $N$ associated with the isogeny class of $E$. Fixing a Néron differential $\omega_E$, there exist some constants $c_0(E)$ and $c_1(E)$ such that\footnote{We may choose $\varphi_0$ and $\varphi_1$ in a way that both $c_0(E)$ and $c_1(E)$ are positive.}
\[\begin{array}{cc}
\varphi_0^*(\omega)=c_0(E) 2\pi i f(\tau)\,d\tau,\ &\varphi_1^*(\omega)=c_1(E) 2\pi i f(\tau)\,d\tau.
\end{array}\]

The Manin constant $c_0(E)$ is an integer (see \cite[Proposition 2]{Edixhoven91}) and is conjecturally  equal to $1$ if $E$ is an $X_0$-optimal elliptic curve, i.e., the degree of $\varphi_0$ is minimal among all modular parametrizations of curves in the isogeny class of $E$ (see \cite[\textsection 5]{Manin72}). The conjecture was proven in \cite{Mazur78} if $E$ is semistable. In particular, it is proven that, for an $X_0$-optimal elliptic curve, $c_0(E)$ is only divisible by $2$ and primes of additive reduction. By \ref{ESur}, $E$ does not admit any $p$-isogeny. Hence $c_0(E)$ is conjecturally prime to $p$ under assumption \ref{ESur}.

The constant $c_1(E)$ was conjectured to be $1$ for all elliptic curves by Stevens in \cite[conjecture 1]{Stevens89}.

By the modularity theorem, we can define for an elliptic curve the modular symbols of its associated modular form. In fact, if $f$ is the modular form associated with the isogeny class of $E$, the modular symbol of $E$ for some $\frac{a}{m}\in \Q$ is defined as
\[\lambda\left(\frac{a}{m}\right):=\int_{i\infty}^{\frac{a}{m}} 2\pi i f(z)\,dz,\]
where the integral follows the vertical line in the upper half plane from the cusp at infinity to $\frac{a}{m}\in\Q$.

The modular symbols satisfy the following relation (see \cite[lemma 5]{WiersemaWuthrich}):
\[\lambda\left(\frac{-a}{m}\right)=\overline{\lambda\left(\frac{a}{m}\right)}.\]
We can consider the real and imaginary parts of the modular symbols and normalise them by one of the Néron periods $\Omega_E^{\pm}$ to obtain rational numbers carrying important arithmetic information:
\begin{equation} \left[\frac{a}{m}\right]^\pm=\frac{\lambda\left(\frac{a}{m}\right)\pm\lambda\left(\frac{-a}{m}\right)}{2\Omega_E^{\pm}}\in \Q.
\label{eq:modular_symbol}
\end{equation}

Assumption \ref{EManin} is required to ensure that the denominator of $\left[\frac{a}{m}\right]^{\pm}$ is not divisible by $p$. Thus
\[\left[\frac{a}{m}\right]^{\pm}\in \Z_{(p)}\subset \Z_p.\]

Although we need the elliptic curve $E$ to be defined over $\Q$ in order to apply the modularity theorem to make use of the modular symbols, we will study the properties of the Mordell-Weil group over a finite abelian extension $K/\Q$ satisfying the following hypotheses:
\begin{itemize}
\item\namedlabel{Kur}{(K1)} $K/\Q$ is unramified at $p$ and at every prime at which $E$ has bad reduction.
\item\namedlabel{Kdeg}{(K2)} The degree $[K:\Q]$ is prime to $p$. 
\item\namedlabel{Kloc}{(K3)} $E(K_\p)[p]=\{O\}$ for every prime $\p$ of $K$ above $p$.
\item\namedlabel{Ktam}{(K4)} All the Tamagawa numbers of $E$ over $K$ are prime to $p$.
\item \namedlabel{KIMCloc}{(K5)} For every character $\chi$ of $\Gal(K/\Q)$, the Iwasawa main conjecture localised at $X\Lambda$ holds for the modular form $f_\chi$, which is defined in \eqref{eq:twsited_mf} below.
\end{itemize}

Throughout this section, denote by $G$ the Galois group $\Gal(K/\Q)$. Also, let $d$ and $c$ be the degree and the conductor of $K/\Q$, respectively. We will also denote by $\OO_d$ the ring of integers of $\Q_p(\mu_d)$ and $\Lambda_d:=\Lambda\otimes \OO_d$, where $\Lambda$ is the Iwasawa algebra defined in \textsection \ref{sec:lambda}.

Assumption \ref{KIMCloc} deserves some comment. If $f=\sum_{n=1}^\infty a_nq^n$ is the newform associated with $E$, the $\chi$-twist of $f$ is defined as
\begin{equation}
f_\chi=\sum_{n=1}^\infty \chi(n) a_n q^n\in S_2(\Gamma_1(N\cond(\chi)^2),\chi^2).
\label{eq:twsited_mf}
\end{equation}
where $\Gamma_1(A)$ is one of the modular groups of level $A$. 
\begin{remark}
With this definition of $f_\chi$, we have that 
\[L(f_\chi,s)=L(E,\chi,s).\]

Indeed, the $L$-function of the twisted modular form is defined as 
\[L(f_\chi,s)=\sum_{n=1}^\infty \frac{\chi(n)a_n}{n^s}=\prod_{\ell} \Bigl(1-\ell^{-s}\chi(\ell)a_\ell+\textbf{1}_N(\ell)\ell^{1-2s}\chi(\ell)^2\Bigr)^{-1}.\]
where $N$ is the level of $f$ and $\textbf{1}_N(\ell)$ is the trivial Dirichlet character modulo $N$, given by $\textbf{1}_N(m)=1$ if $\gcd(N,m)=1$ and $\textbf{1}_N(m)=0$ otherwise.

This definition coincides with the motivic $L$-function of $T_pE\otimes \OO_d(\chi)$, where $\chi$ is normalised such that $\chi(\ell)=\chi(\Frob_\ell$), where $\Frob_\ell$ denotes the arithmetic Frobenius. In this setting, when $ \ell\neq p$, the Euler factor is
\[P_\ell(T)_\chi:=\det_{\OO_d\otimes \Q_\ell} \Bigl(1-\Frob_\ell^{-1} T|((T_p\otimes \OO_d(\chi))^*)^{I_\ell}\Bigr)=1-\ell\chi(\ell)a_\ell T+\textbf{1}_N(\ell)\ell \chi(\ell)^2 T^2.\]
In case, $\ell=p$, one would obtain the same formula using Fontaine's period ring $B_{\crys}$:
\[P_p(T)_\chi=\Bigl(1-\Frob_\ell^{-1} T|(T_p\otimes \OO_d(\chi))^*\otimes B_\crys\Bigr)=1-p\chi(p) a_p T+\textbf{1}_N(p)p \chi(p)^2 T^2.\]
Then the motivic $L$-function is defined as
\[L(T_pE\otimes \OO_d(\chi),s):=\prod_\ell P_\ell(\ell^{-s})=\prod_\ell (1-\ell^{-s} \chi(\ell)a_\ell+\textbf{1}_N(\ell)\ell^{1-2s} \chi(\ell)^2)=L(f_\chi,s).\]
\end{remark}

We assume the Iwasawa main conjecture in the sense of Kato.

\begin{conjecture}(Iwasawa main conjecture for $f_\chi$ in the sense of Kato) Define the Iwasawa cohomology as
    \[
    {H^1_{\textrm{IW}}(\Q_\infty,T\otimes\OO_d(\chi))}:=\varprojlim_n H^1(\Q_n,T\otimes \OO_d(\chi)).\]
Also denote 
    \[X_\infty:=\Hom(H^1_{(\FLambda)^*}(\Q,(T\otimes \Lambda_d(\chi))^*),\Q_p/\Z_p).\]
    Let $z_{\Q_\infty,\chi}\in H^1_{\textrm{IW}}(\Q_\infty,T\otimes\chi)$ be the $\chi$-twist of Kato's zeta element (see \eqref{eq:kataoka_interp} and \eqref{eq:euler_twist}). The Iwasawa main conjecture is the equality of $\Lambda$-ideals
    \[\char\left(\frac{H^1_{\textrm{IW}}(\Q_\infty,T\otimes \OO_d(\chi))}{\Lambda z_{\Q_\infty,\chi}}\right)=\char(X_\infty).\]
    \label{conj:IMC}
    \end{conjecture}
    
    \begin{remark}
    By Theorem \ref{th:MC_ind}, conjecture \ref{conj:IMC} is equivalent to the primitivity of the $\chi$-twisted Kato's Kolyvagin system over $\Q_\infty$. By the argument in Remark \ref{rem:katos_equal} below shows that this is equivalent to the primality of certain Kato's Kolyvagin system for $f_\chi$.
    \end{remark}

    \begin{remark} 
        Iwasawa main conjecture for $f_\chi$ was proven in \cite[Theorem 1]{SkinnerUrban} when $p$ does not divide the level of $f_\chi$, which, under our assumption \ref{Kur}, is equivalent to $E$ having good reduction at $p$, and the existence of an auxilliary prime $\ell\mid\mid N$ such that $\chi$ is unramified at $\ell$ and the reduction modulo $\ell$ of the Galois representation $\overline{\rho}_{f_\chi}$ satisfies that $\dim_{\F_\ell}\overline{\rho}_{f_\chi}^{\II_\ell}=1$ and $\dim_{\F_\ell}\overline{\rho}_{f_\chi}^{G_{\Q_\ell}}=0$. 

          The proof of the Iwasawa main conjecture was extended in \cite[Theorem 1.1]{FouquetWan} to some cases in which $E$ has bad reduction at $p$, assuming the existence of the above auxiliary prime $\ell$ and that the semisimplification of $\overline{\rho}|_{G_{\Q_p}}$ is different to $\psi\oplus\psi$ and $\psi\oplus \chi_{\textrm{cyc}}\psi$ for any character $\psi$ and the cyclotomic character $\chi_{\textrm{cyc}}$.
    \end{remark}

 However, assumption \ref{KIMCloc} only assumes a weaker Iwasawa main conjecture: a special case of the following conjecture when $\beta=X\Lambda$.
    
    \begin{conjecture}(Localised Iwasawa main conjecture for $f_\chi$)
    Let $\beta$ be a height one prime ideal of $\Lambda$. The Iwasawa main conjecture localised at $\beta$ is the equality
    \[\ord_\beta\left(\char\left(\frac{H^1_{\textrm{IW}}(\Q_\infty,T\otimes \OO_d(\chi))}{\Lambda z_{\Q_\infty,\chi}^\infty}\right)\right)=\ord_\beta\left(\char(X_\infty)\right).\]
    \label{conj:IMC_loc}
    \end{conjecture}


We want to describe the Galois structure of the classical Selmer group $\Sel(K,E[p^\infty])$. Indeed, the Selmer group $\Sel(K,E[p^\infty])$ has a natural action of the Galois group $\Gal(K/\Q)$, in which the Galois automorphism acts by conjugation on the cohomology group. Hence, it has a $\Z_p[G]$-module structure, which is the one we are interested in. 

By Shapiro's lemma (see \cite[Proposition 1.6.4]{NSW}), there is an isomorphism
\[H^1(K,T)\cong H^1\left(\Q,\Ind_{G_K}^{G_\Q}(T)\right)\cong H^1(\Q,T\otimes \Z_p[G]).\]
At this point, it is important to make a remark on the Galois action on $T\otimes \Z_p[G]$. In order to construct the Galois cohomology group $H^1(\Q,T\otimes \Z_p[G])$, we need $T\otimes \Z_p[G]$ to be endowed with a continuous action of $G_\Q$. It is given by the following formula:
\[\sigma(t\otimes x)=\sigma t\otimes \sigma x\ \forall \sigma\in G,\ \forall t\in T,\ \forall x\in \Z_p[G].\]

However, cohomological conjugation endows $H^1(K,T)$ with a natural $G$-action. The way to see this action in $H^1(\Q,T\otimes \Z_p[G])$ is considering $T\otimes \Z_p[G]$ as a $\Z_p[G]$-module with the multiplication by the elements of $G$ given by
\[\sigma(t\otimes x)=t\otimes x\sigma^{-1} \ \forall \sigma\in G,\ \forall t\in T,\ \forall x\in \Z_p[G].\]

When $T\otimes \Z_p[G]$ is a $\Z_p[G]$-module, then $H^1(\Q,T\otimes \Z_p[G])$ is also a $\Z_p[G]$-module. In this situation, Shapiro's lemma isomorphism respects the $\Z_p[G]$-structures.

We now study the local conditions used to define the Selmer group. Let $\ell$ be a rational prime and let $v$ be a prime of $K$ above $\ell$. Denote by $G_{v/\ell}$ the Galois group $\Gal(K_v/\Q_\ell)$, which is canonically isomorphic to a subgroup of $G$. Shapiro's lemma can also be applied to the local cohomology groups.
\[H^1(K_v,T)\cong H^1\Bigl(\Q,T\otimes \Z_p[G_{v/\ell}]\Bigr).\]

The classical local condition at a prime $v$ of $K$ is defined as the image of the Kummer map.
\[H^1_{\Fcl}(K_v,T)=\Im\biggl(E(K_v)\widehat\otimes \Z_p\to H^1(K_v,T)\biggr).\]

When $\ell\neq p$, this local condition coincides with the finite cohomology subgroup:
\[H^1_{\Fcl}(K_v,T)=\ker\biggl(H^1(K_v,T)\to H^1(\II_{v},T\otimes \Q_p)\biggr).\]
where $\II_{v}$ is the inertia subgroup of $G_{v}$.
When $\ell=p$, the classical condition can be described purely in terms of $T$, by using $p$-adic Hodge theory. In this case, the classical local condition coincides with the Bloch-Kato local condition defined in \textsection \ref{sec:global}:
 \[H^1_{\FBK}(K_v,T):=\ker\biggl(H^1(K_v,T)\to H^1(K_v,T\otimes B_\crys)\biggr).\]

The Bloch-Kato Selmer structure can be defined similarly for $H^1(\Q_p,T\otimes \Z_p[G_{v/p}])$:
\[H^1_{\FBK}(\Q_p,T\otimes \Z_p[G_{v/p}])=\ker\biggl(H^1(\Q_p,T\otimes \Z_p[G_{v/p}])\to H^1(K_v,T\otimes\Z_p[G_{v/p}]\otimes B_\crys)\biggr).\]
The isomorphism given by Shapiro's lemma identifies both Bloch-Kato conditions, which can be used to define a Selmer structure on $T\otimes \Z_p[G]$. Note that, as $\Z_p[G_{v/\ell}]$-modules, there is an isomorphism 
\[T\otimes \Z_p[G]\cong\bigoplus_{v\mid \ell}  T\otimes \Z_p[G_{v/\ell}],\]
so the localisation maps at primes above $\ell$ can be expressed as the composition
\[\xymatrix{H^1(\Q,T\otimes \Z_p[G])\ar[r] & H^1(\Q_\ell,T\otimes \Z_p[G])\ar[r]^{\sim\ \ \ \ \ \ \ } &\bigoplus_{v\mid \ell}H^1(\Q_\ell,T\otimes \Z_p[G_{v/\ell}]).}\]
These maps can be used to define the Bloch-Kato Selmer group
\[H^1_{\Fcl}(K,T)\cong H^1_{\FBK}(\Q,T\otimes \Z_p[G]).\]

Since $\Z_p[G]$ is not a local ring, we cannot apply the general theory of Kolyvagin systems directly. Since the order of $G$ is prime to $p$ by \ref{Kdeg}, we can split the Selmer group into character parts. In order to do that, we tensor it with $\OO_d$. We thus get an isomorphism 
\[H^1_{\FBK}(\Q,T\otimes \Z_p[G])\otimes \OO_d\cong H^1_{\FBK}(\Q,T\otimes \OO_d[G])=\bigoplus_{\chi\in \widehat{G}} H^1_{\FBK}(\Q,T\otimes e_\chi\OO_d[G]),\]
where $\widehat{G}$ is the character group and $e_\chi$ is the idempotent element associated with the character $\chi$, defined as
\begin{equation}
    e_{\psi}:=\frac{1}{[K:\Q]}\sum_{\sigma\in \Gal(K/\Q)} \overline\psi(\sigma) \sigma\in \Z_p[\psi][\Gal(F/\Q)],
    \label{eq:idempotent}
\end{equation}
where $\Z_p[\psi]$ is the ring obtained by adjoining the values of $\psi$ to $\Z_p$.

Since $T\otimes e_\chi\OO_d[G]$ is isomorphic to $T\otimes \OO_d(\chi)$, where $\OO_d(\chi)$ is $\OO_d$ endowed with an action of $G$ given by $\sigma x=\chi(\sigma)x$, we have an isomorphism
\[H^1_{\FBK}(\Q,T\otimes \Z_p[G])\otimes \OO_d\cong\bigoplus_{\chi\in \widehat G} H^1_{\FBK}(\Q,T\otimes \OO_d(\chi)).\]

Therefore, we can study the groups $H^1(\Q,T\otimes \OO_d(\chi))$ instead. The Galois structure of $H^1_{\FBK}(\Q,T\otimes \Z_p[G])$ is completely determined by the $\OO_d$ structure of $H^1_{\FBK}(\Q,T\otimes \OO_d(\chi))$ of every $\chi$ (see \textsection \ref{sec:examples} for examples of this process). Since $\OO_d$ is a principal local ring, we can apply the general theory of Kolyvagin systems to study this cohomology group.

Knowing the Fitting ideals of the twisted Selmer groups, we can determine $H^1_{\FBK}(K,T)$ up to isomorphism.

\begin{proposition}
The Fitting ideals of $H^1_\FBK(K,T)$ can be computed as:
\begin{equation}
\Fitt^{i}_{\Z_p[G]}\Bigl(H^1_{\FBK}(K,T)\Bigr)=\Z_p[G]\cap\left(\sum_{\chi\in \widehat G} e_\chibar\ \Fitt^i_{\OO_d}\Bigl(H^1_{\FBK}(Q,T\otimes \OO_d(\chi))\Bigr)\right).
\label{eq:fitting_rational}
\end{equation}
Moreover, there is an isomorphism
\[H^1_\FBK(K,T)\cong \bigoplus_{i} \Z_p[G]/I_i,\]
where $I_i$ are ideals of $\Z_p[G]$ satisfying that $I_{i-1}\subset I_i$ for all $i$ and that 
\[\Fitt_{\Z_p[G]}^{i-1}H^1_\FBK(K,T)=I_i\, \Fitt_{\Z_p[G]}^iH^1_\FBK(K,T).\]
\label{prop:fitting_integral}
\end{proposition}

\begin{proof}
The first assertion follows from the faithfully flatness of the morphism $\Z_p\to \OO_d$. The second part of this proposition follows by applying the structure theorem of finitely generated modules over principal ideal domains to the discrete valuation rings in the decomposition of $\Z_p[G]$.
\end{proof}

We have already established the $\Z_p$-module and the Selmer structure we are going to study. In order to complete the Selmer triple, we need to establish which will be our set of Kolyvagin primes $\PP$.

\begin{definition}
The Selmer triple we are going to study is defined as follows.
\begin{itemize}
\item $T=T_pE\otimes \OO_d(\chi)$, where $T_pE$ is the Tate module of the elliptic curve and $\OO_d(\chi)$ is $\Z_p[\chi]$ twisted by a character $\chi$ of $G$ with values in $\overline \Z_p$.
\item The Selmer structure is the Bloch-Kato Selmer structure (see Definition \ref{def:BK_str}).
\item The set of Kolyvagin primes $\PP$ is formed by the good reduction primes that are unramified and split completely at $K/\Q$.
\end{itemize}
\label{def:Kato_triple}
\end{definition}

\begin{remark}
For every $k\in \mathbb N$, the set of primes $\PP\cap \PP_k$ is $\PP_{k,K}$, i.e., the primes $\ell$ satisfying the following conditions.
\begin{itemize}
\item \namedlabel{PEgood}{(PE0)} $E$ has good reduction at $\ell$ and $\ell$ is unramified in $K/\Q$.
\item \namedlabel{PE1modp}{(PE1)} $\ell\equiv 1 \mod p^k$.
\item \namedlabel{PE1dim}{(PE2)} $\widetilde E_\ell(\F_\ell)[p^k]$ is free of rank one over $\Z/p^k$.
\item \namedlabel{PEsplit}{(PE3)} $\ell$ splits completely in $K/\Q$.
\end{itemize}
\label{rem:Eprimes}
\end{remark}

In order to study the structure of $H^1_\FBK(\mathbb Q, T\otimes \OO_d(\chi))$, we define some twists of the Kurihara numbers. Kurihara numbers were defined by M.~Kurihara in \cite{Kur2014} and \cite{Kur2012}. In those articles, they were related to the structure of the Selmer group $\Sel(\Q,E[p^\infty])$ and the veracity of the Iwasawa main conjecture. The definition given by M.~Kurihara was the particular case of Definition \ref{def:kurihara_numbers} below when $\chi$ is the trivial character. In this definition, we twist the Kurihara numbers by the different characters $\chi$ of $G$, so we can get information about the structure of $H^1_\FBK(\mathbb Q, T\otimes \OO_d(\chi))$.

\begin{definition}
Let $n\in \NN(\PP_{k,K})$ and let $\chi$ be a Dirichlet character of conductor $c$. We define the \emph{twisted Kurihara number} as
\begin{equation}
\widetilde \delta_{n,\chi}=\sum_{a\in (\Z/cn\Z)^\times} \chibar(a)\left[\frac{a}{cn}\right]^{\chi(-1)}\left(\prod_{\ell\mid n}\log_{\eta_\ell}^p(a)\right)\in \OO_d/p^{k_n},
\label{eq:kurihara_number}
\end{equation}
where $\eta_\ell$ is a generator of $(\Z/\ell)^\times$ and $\log_{\eta_\ell}^p(a)$ is the unique $x\in \Z/p^k$ such that $\eta_\ell^{-x} a$ has order prime to $p$ in $(\Z/\ell)^\times$ (see Definition \ref{def:logarithm} below). Here $k_n$ is the minimal $k\in \mathbb N$ such that $n\in \NN(\PP_k)$.
\label{def:kurihara_numbers}
\end{definition}

\begin{remark}
The definition of the twisted Kurihara numbers depends on the choice of the primitive roots $\eta_\ell$ for the prime divisors $\ell$ of $n$. However, the $p$-adic valuation of $\delta_{n,\chi}$ is independent of this choice.
\label{rem:delta_well-def}
\end{remark}

\begin{remark}
Note that for $n=1$, the Birch formula in \cite[(I.8.6)]{MTT} relates the twisted Kurihara number with the twisted special $L$-value:
$$\delta_{1,\chi}=\sum_{a\in(\Z/c\Z)^\times} \chibar(a)\left[\frac{a}{c}\right]^{\chi(-1)}=\frac{1}{\tau(\chibar)}\frac{L(E,\chi,1)}{\Omega_E^{\chi(-1)}},$$
where $\tau(\chi)$ is the Gauss sum of $\chi$.
\end{remark}


\begin{definition}
Define the quantities
\[\begin{aligned}
&\ord(\delta_{n,\chi}):=\max\{j\in\N\cup\{0,\infty\}:\delta_n\in p^j(\OO_d/p^{k_n}) \},\\&\partial^{(i)}(\delta_\chi):=\min\{\ord(\delta_{n,\chi}): n\in \NN(\PP),\nu(n)=i\},\\
&\partial^{(\infty)}(\delta_\chi):=\min\{\partial^{(i)}(\delta_\chi):\ i\in \N_0\}.
\end{aligned}\]
In analogy with Definition \ref{def:theta}, define the ideals $\Theta_{i,\chi}$ as 
\[\Theta_{i,\chi}:=(p)^{\partial^{(i)}(\delta_\chi)}\OO_d\subset \OO_d.\]
\end{definition}

The following theorem describes the group structure of the Bloch-Kato Selmer group of $T(\chi)$ in terms of the $\chi$-twisted Kurihara numbers.

\begin{theorem}
Let $E/\Q$ and $p\geq 5$ be an elliptic curve and a prime number satisfying \ref{ESur}-\ref{EManin} and let $K/\Q$ be an abelian extension satisfying \ref{Kur}-\ref{KIMCloc}. Then $\partial^{(\infty)}(\widetilde \delta_\chi)$ is finite. Call $r$ to the minimum $i$ such that $\Theta_{i,\chi}\neq 0$ and $s$ to the minimum $j$ such that $\partial^{(j)}(\widetilde \delta_\chi)=\partial^{(\infty)}(\widetilde \delta_\chi)$. Also, denote by $n_i$ the exponent $\Theta_{i,\chi}=(p)^{n_i}$ and $a_i=\frac{n_i-n_{i+2}}{2}$.
\begin{enumerate}
\item If $\chi=\overline\chi$, then
$$H^1_{\FBK}(\Q,T\otimes \OO_d(\chi))\cong \OO_d^r\oplus \left(\OO_d/(p)^{a_r}\right)^2\oplus\left(\OO_d/(p)^{a_{r+2}}\right)^2\oplus\cdots\oplus \left(\OO_d/(p)^{a_{s-2}}\right)^2.$$
\item If $\chi\neq \overline \chi$, then
$$H^1_{\FBK}(\Q,T\otimes \OO_d(\chi))\cong \OO_d^r\oplus \OO_d/(p)^{n_r-n_{r+1}}\oplus\OO_d/(p)^{n_{r+1}-n_{r+2}}\oplus\cdots\oplus \OO_d/(p)^{n_{s-1}-n_s}.$$
\end{enumerate}
\label{th:EK_str}
\end{theorem}

The rest of this section is dedicated to the proof of Theorem \ref{th:EK_str}, which will be concluded in \textsection\ref{sec:proof_str}.

\begin{remark}
Theorem \ref{th:EK_str} when $K=\Q$ is a result of C.H.~Kim in \cite{Kim23}.
\end{remark}

\begin{remark}
It is enough to prove Theorem \ref{th:EK_str} for the primitive characters of $G$, i.e., those of conductor $c$. Indeed, if $\chi$ is a character of conductor $m$, consider $K_m=K\cap \Q(\mu_m)$. Note that $\chi$ is a primitive character of $\Gal(K_m/\Q)$. Since $[K:K_m]$ is prime to $p$,
\[H^1_{\FBK}(K_m,T)\cong H^1_{\FBK}(K,T)^{\Gal(K/K_m)}.\]
Hence the $\chi$ parts of both Selmer group coincide. Hence if Theorem \ref{th:EK_str} holds for $K_m$ and $\chi$, it also holds for $K$ and $\chi$.
\label{rem:EK_str}
\end{remark}

\begin{remark}
When $\chi$ is the trivial character, M. Kurihara conjectured in \cite{Kur2012} the existence of of some $n\in \NN$ such that $\ord(\delta_n)$ is equal to zero. When $p$ is a prime of ordinary reduction, M. Kurihara proved it under the assumptions of the non-degeneracy of the $p$-adic height pairing and the Iwasawa main conjecture (see \cite[Theorem B]{Kur2014}). An analogue statement can be conjectured for other characters.
\end{remark}

\begin{conjecture}
There exists some $n\in \NN$ such that $\ord(\delta_{n,\chi})=0$. In other words, $\partial^{(\infty)}(\delta_{n,\chi})=0$.
\label{conj:kur}
\end{conjecture}

When $\chi$ is the trivial character, R. Sakamoto proved in \cite[Theorem 1.2]{Sakamoto21} (see also \cite[Theorem 1.11]{Kim23}) that M.~Kurihara's conjecture is equivalent to the Iwasawa main conjecture.

\begin{theorem}
Let $E/\Q$ and $p\geq 5$ be an elliptic curve and a prime number satisfying \ref{ESur}-\ref{EManin}, let $K/\Q$ be an abelian extension satisfying \ref{Kur}-\ref{KIMCloc} and let $\chi$ be a character of $\Gal(K/\Q)$. Then $\partial^{(\infty)}(\delta_\chi)=0$ if and only if the Iwasawa main conjecture \ref{conj:IMC} holds true.
\label{th:EK_IMC}
\end{theorem}

The proof of Theorem \ref{th:EK_IMC} will be shown in \textsection\ref{sec:proof_IMC}.

The rest of \textsection \ref{sec:EC} is dedicated to the proofs of theorems \ref{th:EK_str} and \ref{th:EK_IMC}. Theorem \ref{th:EK_str} is just an application of theorems \ref{th:kur_par} and \ref{th:kur}. In order to prove it, we need to check that the Selmer triple satisfies the assumptions of these theorems.

\begin{proposition}
The Selmer triple $(T\otimes \OO_d(\chi),\Fcl,\PP)$ satisfies the assumptions \ref{Hirred}-\ref{Hprimes} made in \textsection \ref{sec:global}.
\label{prop:EK:hyp}
\end{proposition}

\begin{proof}
Assumption \ref{Hchev} is clear. \ref{Hprimes} holds from \ref{rem:Eprimes}

Note that $\Q(T)$ is only ramified at $p$ and at bad primes of $E$, so $\Q(T)\cap K$ is unramified at every prime by \ref{Kur}. Since $\Q$ has class number $1$, then $\Q(T)\cap K=\Q$. Therefore, every $\sigma\in \Gal(\Q(T)/\Q)$ can be lifted to some $\widetilde \sigma\in G_\Q$ such that $\chi(\widetilde \sigma)=1$.

Every basis of $T$ as a $\Z_p$-module is a basis of $T\otimes \OO_d(\chi)$ as an $\OO_d$-module. After fixing such a basis, the representation induces a map $G_\Q\to \GL_2(\OO_d)$. By \ref{ESur}, the image of this map contains $\GL_2(\Z_p)$, so \ref{Hirred} and \ref{Hsur} hold true.

Let $\Delta\subset \GL_2(\OO_d)$ be the subgroup formed by the matrices $\zeta I$, where $I$ is the identity matrix and $\zeta$ is a $(p-1)^{\textrm{th}}$-root of unity. Since the image of the map $G_\Q\to \GL_2(\OO_d)$ considered above contains $\GL_2(\Z_p)$, there is an inclusion $\Delta\hookrightarrow \Gal(\Q(T\otimes \OO_d(\chi))/\Q)$. If we denote the latter Galois group by $H$, consider the inflation-restriction sequence
\begin{center}
    \begin{tikzpicture}[descr/.style={fill=white,inner sep=1.5pt}]
    
            \matrix (m) [
                matrix of math nodes,
                row sep=4em,
                column sep=1.5em,
                text height=1.5ex, text depth=0.25ex
            ]
            {  0 & H^1\Bigl(H/\Delta,(T/p T\otimes \OO_d(\chi))^\Delta\Bigr) &
                H^1\Bigl(H,T/p T\otimes \OO_d(\chi)\Bigr) & H^1\Bigl(\Delta,T/p T\otimes \OO_d(\chi)\Bigr). \\
            };
    
            \path[overlay,->, font=\scriptsize,>=latex]
            (m-1-1) edge node[auto]{} (m-1-2)
            (m-1-2) edge node[auto]{} (m-1-3)
            (m-1-3) edge node[auto]{} (m-1-4);
       
    \end{tikzpicture}
    \end{center}
Since $T(\chi)^\Delta=0$ and the order of $\Delta$ is prime to $p$, then $H^1(\Q(T\otimes\OO_d(\chi))/\Q,T/p)=0$. Therefore, \ref{Hcoh0} is proven since the Weil pairing implies that $\Q(T,\mu_{p^\infty})=\Q(T)$ and $T/p T\cong T^*[p]$.

For every $\ell \in \Sigma(\FF_{cl})$ and every prime $v$ of $K$ above $\ell$, let $G_{v/\ell}=\Gal(K_v/\Q_\ell)$ and let $H^1_{/\FBK}(\Q_\ell, T\otimes \OO_d(\chi))$ be the quotient $H^1(\Q_\ell, T\otimes \OO_d(\chi))/H^1_{\FBK}(\Q_\ell, T\otimes \OO_d(\chi))$. Since $\# G_v$ is prime to $p$,
$$H^1_{/\FBK}\Bigl(\Q_\ell,T\otimes \OO_d(\chi)\Bigr)\cong H^1_{\FBK}\Bigl(\Q_\ell,T^*\otimes \OO_d(\overline\chi)\Bigr)\du\cong \Biggl(H^1_{\FBK}\biggl(K_v,E[p^\infty]\otimes \OO_d(\chibar)\biggr)\du\Biggr)_{G_{v/\ell}}.$$
The last cohomology group vanishes when $\ell\neq p$. Otherwise, we can compute
\begin{equation} 
    \biggl(H^1_{\FBK}(K_v,E[p^\infty])\du\otimes \OO_d(\chi)\biggr)_{G_{v/p}}\cong \left(\left(\varinjlim_{n\in\mathbb{N}}E(K_v)/p^n E(K_v)\right)^\vee\otimes \OO_d(\chi)\right)_{G_{v/p}}.
    \label{eq:H1s_tf}
\end{equation}
The last term is equal to 
\[\biggl((E(K_v)\otimes \Q_p/\Z_p)^\vee\otimes \OO_d(\chi)\biggr)_{G_{v/p}},\]
which is torsion-free. Therefore, \cite[lemma 3.7.1]{MazurRubin} implies that  \ref{Hcartesian} holds true.

By \cite[Theorem 1.1]{DokschitserDokschitser}, we have that
\[H^1_{\FBK}\Bigl(\Q,T/p^kT\otimes \OO_d(\chi)\Bigr)\cong H^1_\FBK\Bigl(\Q,T^*[p^k]\otimes \OO_d(\chibar)\Bigr)\]
and hence $\chi(T\otimes\OO_d(\chi),\FBK)=0$, so \ref{Hcore0} holds. For \ref{Hloc}, it holds for the prime $p$ by \ref{Kloc}. Indeed, from the computation in \eqref{eq:H1s_tf}, one can deduce that $H^1_{/\FBK}(\Q_p,T\otimes \OO_d(\chi))$ is free of rank one over $\OO_d$ and, from the local duality in Proposition \ref{prop:local_duality}, we have that 
\[H^2(\Q_p,T\otimes \OO_d(\chi))\cong H^0(\Q_v,E[p^\infty]\otimes \OO_d(\chibar))\du=0.\qedhere\]
\end{proof}

The proof of theorems \ref{th:EK_str} and \ref{th:EK_IMC} is structured as follows. \textsection\ref{sec:exp}-\textsection\ref{sec:Kato_kol} are dedicated to relate Kato's Euler system to the Kurihara numbers. In \textsection\ref{sec:exp}, we introduce Kato's Euler system for $T_pE$, originally constructed in \cite{Kato}, and its link to the special $L$-values via the dual exponential map, following \cite{Kataoka21}. In \textsection\ref{sec:twist}, we apply the twisting process explained in \cite[\textsection II.4]{Rubin} to obtain an Euler system for $T\otimes \OO_d(\chi)$. In \textsection\ref{sec:MT}, we use the interpolation property of Mazur-Tate elements to relate them to Kato's Euler system. After applying the Kolyvagin derivative, we obtain a relation between Kato's Euler system and Kurihara numbers.

This relation is obtained through the dual exponential map. Therefore, it is necessary to compute the image of the integral cohomology group under the dual exponential map, as done in \textsection\ref{sec:exp_im}. With these computations, we can prove the equality between the orders of the twisted Kato's Kolyvagin system and the twisted Kurihara numbers in \textsection\ref{sec:Kato_kol}.

The proof of Theorem \ref{th:EK_IMC}, based on the equivalence between the Iwasawa main conjecture and the primitivity of Kato's Euler system, is concluded in \textsection\ref{sec:proof_IMC}. \textsection\ref{sec:proof_str} is dedicated to the proof of Theorem \ref{th:EK_str} as an application of theorems \ref{th:kur_par} and \ref{th:kur} and the functional equation of the Kurihara numbers (\cite[lemma 5.2.1]{Kur2012}).

\subsection{Dual exponential map}
\label{sec:exp}

K. Kato constructed in \cite{Kato} an Euler system for the Tate module of an elliptic curve \(T\) and defined an exponential map to relate this Euler system to the special values of the twisted $L$-functions of the elliptic curve.

Using the Néron differential $\omega_E$ we had already fixed, we can define the dual exponential map (see \cite[definition 3.10]{BlochKato}) as a map
$$\exp^*_{\omega_E}:\ H^1_{/\FBK}(F_\p,T)\otimes \Q_p\to F_\p,$$
where $F_\p$ is the completion at a prime $\p$ above $p$ of an abelian extension $F/\Q$ and $H^1_{/\FBK}(F_\p,T)$ is the quotient $H^1(F_\p,T)/H^1_{\FBK}(F_\p,T)$. 

We will sketch the construction of the exponential map from \cite{BlochKato}. The following is the fundamental short exact sequence from $p$-adic Hodge theory:
\begin{equation}
\xymatrix{0\ar[r] &\Q_p\ar[r] & B_\crys^{\varphi=1}\oplus B_{\dR}^+\ar[r] & B_{\dR}\ar[r] & 0,}
\label{eq:pHT}
\end{equation}
where $B_\crys$ and $B_\dR$ are the crystalline and de Rahm period rings. For $V:=T\otimes \Q_p$, we denote
$$\begin{array}{ccc}
D_\crys(V)=(B_\crys\otimes V)^{G_{F_\p}},\ \ &D_\dR(V)=(B_\dR\otimes V)^{G_{F_\p}},\ \ &D_\dR(V)^+=(B_\dR^+\otimes V)^{G_{F_\p}}.
\end{array}$$

If we consider the cohomological exact sequence in \eqref{eq:pHT} tensored with $V$,
$$\xymatrix{0\ar[r] &\Q_p\ar[r] & D_\crys(V)^{\varphi=1}\oplus D_{\dR}(V)^+\ar[r] & D_{\dR}(V)\ar[r] & H^1_\FBK(F_\p,V)\ar[r] & 0.}$$
Since the tangent space of $E$ is isomorphic to $D_\dR(V)/D_\dR(V)^+$, this exact sequence induces a surjective map
$$\exp:\ \textrm{tan}(E/F_\p)\to H^1_\FBK(F_\p,V).$$

The dual $\omega_E^*$ of the Néron differential generates the tangent space of $E/F_\p$ as an $F_\p$-vector space, so we can consider the exponential as a map from $F_\p$,
$$\exp_{\omega_E}:\ F_\p\to H^1_{\FBK}(F_\p,V).$$
Its dual map is the one we will be interested in:
$$\exp_{\omega_E}^*:\ H^1_{/\FBK}(F_\p,V)\to \Hom(F_\p,\Q_p).$$
We can identify the latter with $F_\p$ via the following isomorphism:
\begin{equation}
F_\p\to \Hom(F_\p,\Q_p),\ x\mapsto (y\mapsto \Tr_{F_\p/\Q_p}(xy)).
\label{eq:dual_trace}
\end{equation}

The above map is defined for every Galois representation using $p$-adic Hodge theory. However, when $T$ is the Tate module of an elliptic curve $E$, the exponential map has a geometrical meaning (see \cite[example $3.10.1$]{BlochKato}). Note that in this case,
$$H^1_\FBK(F_\p,T)=\varprojlim_n E(F_\p)/p^n E(F_\p).$$
Then the exponential map coincides with the Lie group exponential map defined on the elliptic curve. Moreover, it can be also understood as the tensor with $\Q_p$ of the formal group exponential map defined on a certain power of the maximal ideal of the ring of integers of $F_\p$.

We will be interested in the image under this map of the elements coming from the global cohomology group $H^1(F,T)$. In order to do that, we consider the localisation at a rational prime $q$ as the direct sum of the localisation maps above every prime $v$ above $q$:
$$\loc_q^s:\ H^1(F,V)\to \bigoplus_{v\mid q} H^1_{/\FBK}(F_v,V).$$

Then there is an Euler system $(z_F)_{F\in \Omega}$ satisfying the following interpolation property (see \cite{Kataoka21}, theorem 6.1). Fix an inclusion $\overline \Q\subset \C$; for every character $\psi$ of $\Gal(F/\Q)$, we have that
\begin{equation}\sum_{\sigma\in \Gal(F/\Q)} \psi(\sigma) \sigma(\exp^*_{\omega_E}(\loc_s^p(z_F)))=\frac{L_{S_F\cup\{p\}}(E,\psi,1)}{\Omega_E^{\psi(-1)}}\in F\otimes \Q_p,
\label{eq:kataoka_interp}
\end{equation}
where $S_F$ is the primes ramifying at $F/\Q$ and $\Omega_E^{\pm}$ denote the Néron periods of $E$. Here $L_{S_F\cup\{p\}}(E,\psi,1)$ is the $S_F\cup\{p\}$-truncated and $\psi$-twisted L-function, defined as
$$L_{S_F\cup\{p\}}(E,\psi,s)=\prod_{\ell\notin S_F, \ell\neq p} (1-a_\ell\psi(\ell) \ell^{-s} +\textbf{1}_N(\ell)\psi(\ell)^2 \ell^{1-2s})^{-1}.$$

Call $w_F:=\exp^*_{\omega_E}(\loc^s_p(z_F))$. For every $n,m\in \N$ such that $m\mid n$, we use the following notation. We call $\tilde n$ the product of all the primes dividing $n$ and $r(m,n):=\textrm{lcm}(\tilde n, m)$. Furthermore, let $s(m,n):=\frac{r(m,n)}{m}$ be the product of primes dividing $n$ but not $m$. Note that $m$ and $s(m,n)$ are relatively prime. When there is no risk of confusion, we will denote these quantities by $r$ and $s$.

For the remaining of this section, fix $n\in \N$ divisible by neither $p$ nor any bad prime of $E$. For every character $\psi$ of conductor $m$, the algebraic $L$-value is defined as
$$\LL_n(E,\psi):=\frac{L_{S_n\cup\{p\}}(E,\psi,1)}{\varphi(n)e_{\overline \psi}(\zeta_{r})\Omega_E^{\psi(-1)}}=(-1)^{\nu(s)}\overline\psi(s)\frac{\varphi(r)L_{S_n\cup\{p\}}(E,\psi,1)}{\varphi(n)\varphi(m)e_{\overline \psi}(\zeta_{m})\Omega_E^{\psi(-1)}}.$$
where $S_n$ is the set of prime divisors of $n$ and $\varphi$ represents the Euler totient function and $\zeta_j=e^{\frac{2\pi i}{j}}\in \C$. This is a modification of the definition in \cite{WiersemaWuthrich} to consider the cases when $\psi$ is a non-primitive character. Note that, when $\psi$ is primitive, the product $\varphi(n) e_\psibar(\zeta_r)$ coincides with the Gauss sum of $\psi$. It is worth mentioning that the last equality comes from the fact that 
\begin{equation}
e_\psibar(\zeta_r)=\frac{\varphi(m)}{\varphi(r)}e_{\psibar}\biggl(\Tr_{\Q(\mu_r)/\Q(\mu_m)}(\zeta_{r})\biggr)=\frac{\varphi(m)}{\varphi(r)}e_\psibar\biggl((-1)^{\nu(s)} \zeta_m^{(s^{-1})} \biggr)=\frac{(-1)^{\nu(s)}\varphi(m)}{\psibar(s)\varphi(r)}e_{\psibar}(\zeta_m),
    \label{eq:trace}
\end{equation}
where $s^{-1}$ is the inverse of $s$ mod $m$.

If we let $\DD_n$ be the set of Dirichlet characters modulo $n$, we can define the Stickelberger element as
$$\Theta_n:= \sum_{\psi \in \DD_n} \LL_n\Bigl(E,\overline\psi\Bigr) e_\psi\in \Q_p[\zeta_{\varphi(n)}][\Gal(\Q(\mu_n)/\Q)].$$

We can lower bound the $p$-adic valuation of the $\psi$-parts of $\Theta_n$ for certain characters. In order to do that, define 
\begin{equation}
k_\psi':=v_p\Biggl(\frac{1-a_p\psi(p)+\textbf{1}_N(p)\psi(p)^2}{p}\Biggr).
\label{eq:kchi}
\end{equation}

\begin{proposition}
Let $m\in \N$ be an integer prime to $Np$ and let $\psi$ be a character of conductor $m$, satisfying that $r(m,n)=n$ or, equivalently, that $\gcd(m,n/m)=1$. Then $\psi(\Theta_n)=\LL_n\Bigl(E,\psibar\Bigr)\in p^{k_\psi'}\Z_p[\psi]$.
\label{prop:integrality_L}
\end{proposition}

\begin{proof}
By assumption \ref{EManin} and \cite[Theorem 2]{WiersemaWuthrich}, the primitive algebraic $L$-value, defined as $\LL\Bigl(E,\psibar\Bigr):=\frac{L(E,\psibar,1)}{\varphi(m)e_{\psi}(\zeta_m)\Omega_E^{\psi(-1)}}$ belongs to $\Z_p[\psi]$. It satisfies an explicit relation with $\LL_n\Bigl(E,\psibar\Bigr):$
\[\LL_n(E,\psibar)=(-1)^{\nu(s)}\psi(s) \frac{\varphi(r)}{\varphi(n)} \prod_{\ell\in (S_n\cup\{p\})\setminus S_m} \left(\frac{\ell-a_\ell\psibar(\ell) +\textbf{1}_N(\ell)\psibar(\ell)^2}{\ell}\right)\ \LL\Bigl(E,\psibar\Bigr). \]
When $r=n$, all of the terms in the right hand side belong to $\Z_p[\psi]$ with the possible exception of the Euler factor at $p$, which has $p$-adic valuation $k_\psibar'$, which is the same as the one of $k_\psi'$. Therefore,
\[\LL_n\Bigl(E,\psibar\Bigr)\in p^{k_\psi'}\Z_p[\psi].\]\qedhere
\end{proof}

Stickelberger elements can be also defined for every abelian extension of $\Q$. If $F/\Q$ is an abelian extension of conductor $n$, consider the projection
\begin{equation}
c_{\Q(\mu_n),F}:\ \Q_p[\Gal(\Q(\mu_n)/\Q)]\to \Q_p[\Gal(F/\Q)],\ \sigma\mapsto \sigma|_F.
\label{eq:projection}
\end{equation}

\begin{definition}
Let $F/\Q$ be an abelian extension of conductor $n$. Then we define the Stickelberger element of $F$ as
$$\Theta_F:=c_{\Q(\mu_n),F} \Theta_n.$$
\end{definition}

The way to relate the Stickelberger elements to Kato's Euler system is by considering their action on certain elements in $\Q(\mu_n)$.

\begin{definition}
For every $n\in \N$, consider the element
$$\xi_n=\sum_{\tilde n\mid d\mid n}\zeta_d.$$
\end{definition}

\begin{lemma}
If $\psi$ is a Dirichlet character modulo $n$ of conductor $m$, then
$$e_\psi(\xi_n)=e_\psi(\zeta_{r(m,n)}).$$
\label{lem:chi_xi}
\end{lemma}

\begin{proof}
For every $d$ dividing $n$, let $d'=\textrm{gcd}(d,m)$. For every $\sigma\in \Gal(\Q(\mu_d)/\Q(\mu_{d'}))$, we can find a lift $\widetilde \sigma \in \Gal(\Q(\mu_n)/\Q(\mu_{d'}))$ such that 
$$\begin{array}{cc}
\widetilde \sigma\bigm|_{\Q(\mu_d)}=\sigma\textrm{ and }& \widetilde\sigma\bigm|_{\Q(\mu_m)}=\textrm{Id}|_{\Q(\mu_m)}.
\end{array}$$

Since $\psi$ has conductor $m$,
$$e_\psi(\sigma(\zeta_d))=e_\psi(\widetilde \sigma(\zeta_d))=e_\psi(\zeta_d)\ \forall \sigma\in \Gal(\Q(\mu_d)/\Q(\mu_{d'})).$$
Therefore
$$e_\psi(\zeta_d)=\frac{1}{[\Q(\mu_d):\Q(\mu_{d'})]} e_\psi(\Tr_{\Q(\mu_d)/\Q(\mu_{d'})} \zeta_d).$$

Assume first that there is a prime $\ell$ such that $v_\ell(d)>v_\ell(m)\geq1$. Then $\gcd(d,m)\mid \frac{d}{\ell}$, so
$$\Tr_{\Q(\mu_d)/\Q(\mu_{d'})} \zeta_d=\Tr_{\Q(\mu_{d/\ell})/\Q(\mu_{d'})}\left(\Tr_{\Q(\mu_d)/\Q(\mu_{d/\ell})} \zeta_d\right).$$
However, $\Tr_{\Q(\mu_d)/\Q(\mu_{d/\ell})} \zeta_d=0$ because $\ell\mid \frac{d}{\ell}$. Indeed, $\zeta_d$ is a root of the polynomial $P(T)=T^\ell-\zeta_{d/\ell}$. Since $[\Q(\mu_d):\Q(\mu_{d/\ell})]=\ell$ because $\ell^2\mid d$, then $P(T)$ is irreducible. The sum of the roots of $P$ is zero and, hence, so is the trace of $\zeta_d$. Therefore, $e_\psi(\zeta_d)=0$ in this case.

Given some $d$ such that the above prime $\ell$ does not exist, then $d\mid r$. Moreover, assume $d< r$. Since $\tilde n\mid d$, then $d'<m$,
$$e_\psi(\zeta_d)=\frac{1}{[\Q(\mu_d):\Q(\mu_{d'})]}e_{\psi}(\Tr_{\Q(\mu_d)/\Q(\mu_d')}\zeta_d).$$
Since $\psi$ has conductor $m$, then $e_\psi(x)=0$ for every $x\in \Q(\mu_{d'})$ and, therefore, $e_\psi(\zeta_d)=0$.

The only remaining term is $\zeta_{r}$, so 
\[e_\psi(\xi_n)=\sum_{\tilde n\mid d\mid n} e_\psi(\zeta_d)=e_\psi(\zeta_{r}).\qedhere\]
\end{proof}

The concept of $\xi_n$ extends naturally to abelian extensions of $\Q$. Assume $F/\Q$ is an abelian extension of conductor $n$. Define
$$\xi_F:= \Tr_{\Q(\mu_n)/F} \xi_n.$$

\begin{corollary}
If $F/\Q$ is an abelian extension of conductor $n$ and $\psi$ is a character of conductor $m$ whose fixed field contains $F$, then 
$$e_\psi(\xi_F)=[\Q(\mu_n):F] e_\psi(\zeta_{r}).$$
\label{cor:chi_xi}
\end{corollary}

\begin{proof}
Since $F$ contains the fixed field of $\psi$, by Lemma \ref{lem:chi_xi},
\[e_\psi(\xi_F)=[\Q(\mu_n):F] e_\psi(\xi_n)=[\Q(\mu_n):F] e_\psi(\zeta_{r}).\]\qedhere
\end{proof}

We can compute
$$e_\psi(\Theta_F(\xi_F))=\Theta_n e_\psi(\xi_F)=[\Q(\mu_n):F]\Theta_n e_\psi(\zeta_{r})=[\Q(\mu_n):F]\LL_n\Bigl(E,\psibar\Bigr)e_\psi(\zeta_{r}).$$

Thus,
$$e_\psi(\Theta_F(\xi_F))=\frac{L_{S_n\cup\{p\}}(E,\psibar,1)}{[F:\Q]\Omega_E^{\psi(-1)}}\in F\otimes \Q_p.$$

Since that is true for all characters, we can conclude from equation \eqref{eq:kataoka_interp} that
\begin{equation}
w_F=\Theta_F(\xi_F)\in F\otimes \Q_p.
\label{eq:Kato_MT}
\end{equation}

\subsection{Twisting Kato's Euler system}
\label{sec:twist}

According to Definition \ref{def:Kato_triple}, we are interested in studying the representation
\[T(\chi):=T_pE\otimes \OO_d(\chi),\]
where $\chi$ is a primitive character of $G$. In order to do that, we have to apply the twisting process of Euler systems described in \cite[\textsection II.4]{Rubin}. Define
\begin{equation}
z_{F,\chi}:=\textrm{cor}_{KF/F}(z_{KF}\otimes 1_\chi),
\label{eq:euler_twist}
\end{equation}
where $1_\chi$ represents the unit element in $\OO_d(\chi)$.

The goal of this section is to describe the dual exponential of the twisted zeta element. In particular, with the notation of \eqref{eq:Kato_MT}, we want to show that this value coincides with the $\chibar$ part of $\Theta_{KF}(\xi_{KF})$.

\begin{proposition} (\cite[Proposition II.4.2]{Rubin})
The collection $\{z_{F,\chi}\}_{F\in \Omega}$ is an Euler system for $T(\chi)$.
\end{proposition}

We will now describe the dual exponential map of \(V\otimes \OO_d(\chi)\). Note that \(V\) and \(V\otimes \OO_d(\chibar)\) are equal as $G_{(KF)_w}$-modules for every prime $w$ of $KF$ above $p$. Hence we can consider the exponential map

\[\exp:\ \bigoplus_{w\mid p}\frac{(B_{\dR}\otimes V\otimes \OO_d(\chibar))^{G_{(KF)_w}}}{(B^+_{\dR}\otimes V\otimes \OO_d(\chibar))^{G_{(KF)_w}}}\to \bigoplus_{w\mid p}H^1_\FBK((KF)_w,V\otimes \OO_d(\chibar)).\] 

For every $w$, we have an isomorphism (depending on fixing a Weierstrass model for $E$)
\[\frac{(B_{\dR}\otimes V\otimes \OO_d(\chibar))^{G_{(KF)_w}}}{(B^+_{\dR}\otimes V\otimes \OO_d(\chibar))^{G_{(KF)_w}}}\cong (KF)_w\otimes \OO_d(\chibar).\]

Since \(G_{w/v}:=\Gal((KF)_w/F_v)\) is finite, then \(H^1\left((KF)_w/F_v,B_\dR^+\otimes V\otimes \chibar\right)=0\), so
\[\left(\frac{\Bigl(B_{\dR}\otimes V\otimes \OO_d(\chibar)\Bigr)^{G_{(KF)_w}}}{\Bigl(B^+_{\dR}\otimes V\otimes \OO_d(\chibar)\Bigr)^{G_{(KF)_w}}}\right)^{G_w}=\frac{\Bigl(B_{\dR}\otimes V\otimes \OO_d(\chibar)\Bigr)^{G_{F_v}}}{\Bigl(B^+_{\dR}\otimes V\otimes \OO_d(\chibar)\Bigr)^{G_{F_v}}}.\]

By considering the direct sum over all $w\mid p$, the exponential map over $F_v$ can be written as
\[\exp_{\omega_E,\chibar}:\ \Bigl(KF\otimes \Q_p\otimes \OO_d(\chibar)\Bigr)^{\Gal(KF/F)}\to \bigoplus_{v\mid p} H^1_\FBK\Bigl((F_v,V\otimes \OO_d(\chibar))\Bigr).\]

Note that the first term is the $\chi$-part of $KF\otimes \Q_p$. Hence the dual exponential map can be written as 
\[\exp^*_{\omega_E,\chibar}:\ \bigoplus_{v\mid p} H^1_{/\FBK}(F_v,V\otimes \OO_d(\chi))\to \Hom\Bigl(e_\chi(KF\otimes L),\Q_p\Bigr)\cong e_{\chibar}\Bigl((KF)\otimes L\Bigr),\]
where we denote $L=\OO_d\otimes \Q_p$ and the last isomorphism comes from the fact that the identification in \eqref{eq:dual_trace} is Galois equivariant. Note that we are also using \eqref{eq:dual_trace} to identify \(\Hom(L,\Q_p)\cong L\).

In an abuse of notation, we will also denote by $\exp_{\omega_E,\chibar}^*$ to the following map
\[\exp^*_{\omega_E,\chibar}:\ H^1(F,V\otimes \OO_d(\chi))\to \bigoplus_{v\mid p} H^1_{/\FBK}(F_v,V\otimes \OO_d(\chi))\to e_\chibar((KF)\otimes L).\]

Now we will describe how the twisting process in \eqref{eq:euler_twist} is reflected in the images of the dual exponential map. 

First note that over $KF$, the map $\exp_{\omega_E,\chibar}^*$ coincides with $\exp_{\omega_E}^*\otimes \textrm{Id}_{\OO_d(\chi)}$. Hence we just need to see how $\exp_{\omega_E,\chibar}^*$ behaves under the corestriction. 

\begin{proposition}
Let $c\in H^1(F,V\otimes \OO_d(\chi))$ and $d=\cor_{KF/F} c\in H^1(KF,V\otimes \OO_d(\chi))$, where $\cor$ denotes the corestriction map. Then
\[\exp_{\omega_E,\chibar}^*(d)=N_{KF/F}\exp_{\omega_E,\chibar}^*(c)\in (KF\otimes \Q_p\otimes \OO_d(\chi))^{G_F}.\]
\label{prop:exp_cor}
\end{proposition}

\begin{proof}
Localising at primes above $p$, we have that
\[\loc_{p}^s(d)=\left(\bigoplus_{w\mid p} \cor_{(KF)_w/F_v}\right) (\loc_p^s(c)).\]

By \cite[Proposition 1.5.3 (iv)]{NSW}, if we understand $\loc_p^s(c)\du$ and $\loc_p^s(d)\du$ as maps into the duals of the finite cohomology groups, we have that 
\[\loc_{p}^s(c)\du=\loc_{p}^s(d)\du\circ\left(\bigoplus_{w\mid p}\res_{(KF)_w/F_v}\right).\]
where $\res$ denotes the restriction map. By \cite[Proposition 1.5.2]{NSW}, the following diagram is commutative

\begin{center}
    \begin{tikzpicture}[descr/.style={fill=white,inner sep=1.5pt}]
    
            \matrix (m) [
                matrix of math nodes,
                row sep=4em,
                column sep=5em,
                text height=1.5ex, text depth=0.25ex
            ]
            {  (KF\otimes \Q_p\otimes \chibar)^{G_F}  &    \bigoplus_{v\mid p} H^1_\f(F_v,V\otimes \chibar)   & \Q_p\\
            KF\otimes \Q_p\otimes \chibar  &    \bigoplus_{w\mid p} H^1_\f((KF)_w,V\otimes \chibar)   & \Q_p.\\
            };
    
            \path[overlay,->, font=\scriptsize,>=latex]
            (m-1-1) edge node[auto]{$\exp_{\omega_E}$} (m-1-2)
            (m-1-2) edge node[auto]{$\loc_p^s(d)\du$} (m-1-3) 

            (m-2-1) edge node[auto]{$\exp_{\omega_E}$} (m-2-2)
            (m-2-2) edge node[auto]{$\loc_p^s(c)\du$}(m-2-3)
            (m-1-1) edge node[auto]{$\subset$}(m-2-1)
            (m-1-2) edge node[auto]{$\bigoplus_{v\mid p} \res$} (m-2-2)
            (m-1-3) edge node[auto]{$=$} (m-2-3);

    \end{tikzpicture}
    \end{center}

Therefore, $\exp^*_{\omega_E,\chi}(d)$ is the restriction to $(KF\otimes\Q_p\otimes \chibar)^{G_F}$ of $\exp^*_{\omega_E,\chi}(c)$. Under the identification \eqref{eq:dual_trace}, it means that
\[\exp_{\omega_E,\chi}^*(d)=N_{KF/F}\exp_{\omega_E,\chi}^*(c).\qedhere\]
\end{proof}

If $K\cap F=\Q$, the dual exponential map of the twisted Kato's Euler system is 
\[w_{F,\chi}:=\exp_{\omega_E,\chibar}^*(z_{F,\chi})=N_{KF/F}(\exp_{\omega_E,\chibar}^*(z_{KF}\otimes 1_\chi))=[K:\Q]e_\chibar(\omega_{KF}).\]

By equation \eqref{eq:Kato_MT},
\begin{equation}
w_{F,\chi}=[K:\Q]e_{\chibar}(\Theta_{KF}(\xi_{KF}))= d\, e_\chibar(\Theta_{KF})(\xi_{KF}).
\label{eq:twisted_omega}
\end{equation}

\begin{remark}
\label{rem:katos_equal}
Kato's theory is not exclusive for elliptic curves, but can be done for modular forms. In particular, we can apply it to the modular form $f_\chi$, for some character $\chi$ of $G$, in order to obtain an Euler system $z_{F}^\chi\in H^1\bigl(F,V\otimes \OO_d(\chi)\bigr)$ satisfying the interpolation property. It is essentially the same Euler system as the one defined in \eqref{eq:euler_twist}.

Let $g$ be a modular form and let $K_g$ be its field of coefficients. Let $\lambda$ be a prime of $K_g$ above $p$. Kato defined in \cite[\textsection 6.3]{Kato} the $\Q_p$-vector space $V_{K_{g,\lambda}}(g)$ to be the maximal Hecke eigenquotient of $H^1_{\textrm{ét}}(Y_1(N),K_{g,\lambda})$ associated with $g$. Every $\gamma\in V_{K_{g,\lambda}(g)}$ can be used to construct an Euler system $z_\gamma$ which can be characterised by an interpolation property.

If we denote $S(g)$ to the Hecke eigenspace of $S_2(\Gamma_1(N))$ containing $g$ and $V_\C(g)=V_{K_g,\lambda}(g)\otimes_{K_{g,\lambda}}\C$, the period map defined in \cite[\textsection 4.10]{Kato} induces a map
\[\per:\ S(g)\to V_{\C}(g).\]
Let $F$ be a number field and let $\psi$ be a character of $\Gal(F/\Q)$. Then the interpolation property for $z_\gamma$ can be written as
\[\sum_{\sigma\in \Gal(F/\Q)} \psi(\sigma) \per\bigl(\sigma(\exp^*(\loc_s^p(z_{\gamma,F})))\bigr)^{\psi(-1)}=L_{S_F\cup\{p\}}(g,\chi,1)\gamma^{\psi(-1)},\]
where, for some $\eta\in V_{K_g,\lambda}(g)$, $\eta^{\pm}$ denotes the projection of $\eta$ to the eigenspace in which the complex conjugation acts by multiplication with $\pm$. Note that \eqref{eq:kataoka_interp} can be obtained after choosing a suitable $\gamma_0$ (see \cite[Theorem 6.1]{Kataoka21}). 

This interpolation property can be used to compare the $\chi$-twisted Euler systems constructed for $f$ in \eqref{eq:euler_twist}, denoted by $z^f_{F,\chi}$ and the Euler system for $f_\chi$, constructed for a suitable $\gamma\in V_{L_{\lambda}}(f_\chi)$, where $\lambda$ is a prime of $L$ above $p$. This Euler system will be denoted by $z_{\gamma_F}^{f_\chi}$.

We follow the argument in \cite[\textsection 14.6]{Kato}. There is an isomorphism of Galois representations 
\begin{equation}\Psi:\ V_{\Q_p}(f)\otimes L_\lambda(\chi)\cong V_{L_\lambda}(f_\chi).
\label{eq:Kato146}
\end{equation}

Choose $\gamma=\Psi(\gamma_0\otimes 1)$. Note that $\gamma^{\pm\chi(-1)}=\Psi(\gamma^{\pm}\otimes 1)$.
The isomorphism in \eqref{eq:Kato146} induces an isomorphism of cohomology groups:
\begin{equation}
H^1(F,V_{\Q_p}(f)\otimes L_\lambda(\chi))\cong H^1(F,V_{L_{\lambda}}(f_\chi)).
\label{eq:iso_coh}
\end{equation}

 We claim that $z_{\gamma,F}^{f_\chi}$ is the image of $z_{\chi,F}^f$ under the isomorphism in \eqref{eq:iso_coh}. Consider the commutative diagram

\begin{center}
    \begin{tikzpicture}[descr/.style={fill=white,inner sep=1.5pt}]
    
            \matrix (m) [
                matrix of math nodes,
                row sep=4em,
                column sep=4em,
                text height=1.5ex, text depth=0.25ex
            ]{
            H^1(F,V_{\Q_p}(f)\otimes L_\lambda(\chi)) &
            S(f)\otimes L_\lambda(\chi)  &
            V_\C(f)  \\
            H^1(F,V_{L_{\lambda}}(f_\chi)) &
            S(f_\chi)\otimes L_\lambda  &
            V_\C(f_\chi).\\
            };
    
            \path[overlay,->, font=\scriptsize,>=latex]
            (m-1-1) edge node[auto]{$\exp^*$} (m-1-2)
            (m-1-2) edge node[auto]{$\per$} (m-1-3)
            (m-2-1) edge node[auto]{$\exp^*$} (m-2-2)
            (m-2-2) edge node[auto]{$\per$} (m-2-3)
            (m-1-1) edge node[auto]{$\sim$} (m-2-1)
            (m-1-3) edge node[auto]{$\sim$} (m-2-3);

    \end{tikzpicture}
    \end{center}

Following \cite[(6.3)]{Kataoka21}, we can compute the image of the zeta elements under the composition $\per\circ \exp^*$. If $\psi$ is a character of $\Gal(F/\Q)$, we have that
\[\begin{aligned}
& (\per\circ \exp^*)\biggl(e_\psi\Bigl(z_{F,\chi}^f\Bigr)\biggr)=L_{S_F\cup\{p\}}(E,\chi\psi,1) (\gamma_0^{\chi\psi(-1)}\otimes 1_\C),\\
&(\per\circ \exp^*)\biggl(e_\psi\Bigl(z_{\gamma,F}^{f_\chi}\Bigr)\biggr)=L_{S_F\cup\{p\}}(f_\chi,\psi,1) (\gamma^{\psi(-1)}\otimes 1_\C).
\end{aligned}\]

Since $L_{S_F\cup\{p\}}(E,\chi\psi)=L_{S_F\cup\{p\}}(f_\chi,\psi)$ and $\gamma_0^{\chi\psi(-1)}\otimes 1_\C$ is identified with $\gamma^{\psi(-1)}\otimes 1_\C$ under the isomorphism on the right, the images of both zeta elements are the same under the identifications made. Since the compositions $\per\circ\exp^*$ are injective maps, we can conclude that both zeta elements are identified under the isomorphism
 \[H^1(F,V_{\Q_p}(f)\otimes L_\lambda(\chi))\cong H^1(F,V_{L_{\lambda}}(f_\chi)).\]
\end{remark}

\subsection{Mazur-Tate elements and Kolyvagin derivative}
\label{sec:MT}

Using the interpolation in \eqref{eq:kataoka_interp}, we can relate $\omega_F$ to  Mazur-Tate elements defined in \cite{MazurTate}. The main advantage of this relation is that Mazur-Tate elements have an explicit formula in terms of the modular symbols defined in \eqref{eq:modular_symbol}, which is a key fact in the relation between Kato's Euler system and the Kurihara numbers.
\begin{definition}
Let $n\in \Z$. The \emph{Mazur-Tate modular element} for $n$ is defined as
$$\theta_{n}:=\sum_{a\in (\Z/n\Z)^*} \left(\left[\frac{a}{n}\right]^++\left[\frac{a}{n}\right]^-\right) \sigma_a \in \Z_p[\Gal(\Q(\mu_n)/\Q)],$$
where $\sigma_a$ is the element of $\Gal(\Q(\mu_n)/\Q)$ that sends $\zeta_n$ to $\zeta_n^a$. The integrality of the Mazur-Tate modular elements holds true under assumption \ref{EManin}.
\end{definition}

Mazur-Tate elements can be also defined for abelian extensions of $\Q$.

\begin{definition}
Let $F/\Q$ be an abelian extension of conductor $n$. Then the Mazur-Tate element of $F$ is defined as 
$$\theta_F:=c_{\Q(\mu_n)/F} (\theta_n).$$
where $c_{\Q(\mu_n)/F}$ is the map defined in \eqref{eq:projection}.
\end{definition}

Mazur-Tate elements can be related to special $L$-values by using the Birch's formula (see \cite[formula (8.6)]{MTT}).

\begin{proposition}(\cite[\textsection 1.4]{MazurTate},\cite[lemma 6, proposition 7]{WiersemaWuthrich})
If $\psi$ is a Dirichlet character of conductor $n$, then
$$\psi(\theta_n)=\frac{n}{\varphi(n) e_\psi(\zeta_n)} \frac{L_{S_n}(E,\psibar,1)}{\Omega_E^{\psi(-1)}}\in\Z_p[\psi],$$
where $S_n$ is the set of primes dividing $n$ and $\varphi(n)$ is the Euler totient function. Note that the product $\varphi(n) e_{\psi}(\zeta_n)$ coincides with the Gauss sum $\tau(\psibar)$.
\label{prop:MT_interpolation}
\end{proposition}

For primitive characters, the $\psi$ parts of the Mazur-Tate element are easily comparable using the $\psi$ parts of $\Theta_n$.

\begin{corollary}
If $\psi$ is a Dirichlet character of conductor $n$, then
$$\psi(\Theta_n)=\frac{1}{n}\psi(\theta_n)(1-p^{-1} a_p \psibar(p)-p^{-1}\textbf{1}_{N}(p) \psibar(p)^2).$$
\label{cor:Theta_MT}
\end{corollary}

For every $n\in \NN(\PP_k)$, recall that $\Q(n)$ is the maximal $p$-subextension inside $\Q(\mu_n)$ and let $K(n):=K\Q(n)$. Note that there is a canonical identification
\begin{equation}
\Gal(K(n)/\Q)=\Gal(K/\Q)\times \Gal(\Q(n)/\Q)=G\times \GG_n.
\label{eq:K(n)_prod}
\end{equation}

We can use Corollary \ref{cor:Theta_MT} to compare $\Theta_{K(n)}$ and $\theta_{K(n)}$ as elements in the group ring $\Z_p[\Gal(K(n)/\Q)]$. We follow the process in \cite{Ota18} and \cite{KimKimSun}. However, in our case, we only have the equality for the primitive character parts. That implies that both elements are not necessarily equal, but they are related enough so we can compare their Kolyvagin derivatives.

Consider the element
$$\Upsilon_{K(n)}:=\Theta_K(n)-\frac{1}{n}(1-p^{-1} a_p \Frob_p^{-1}-p^{-1} \textbf{1}_{N}(p) \Frob_p^{-2})\theta_{K(n)}\in \Z_p[G].$$

The following result can be deduced from Corollary \ref{cor:Theta_MT}
\begin{corollary}
For every primitive character $\psi$ of $\Gal(K(n)/\Q)$, we have that
$$e_\psi \Upsilon_{K(n)}=0.$$
\label{cor:ups_interp}
\end{corollary}

Recall that $\chi$ was a primitive character of $\Gal(K/\Q)$ which was                                              used to construct $T(\chi)$. Using the identification in equation \eqref{eq:K(n)_prod}, we can consider $\chi(\Theta_n)$ and $\chi(\theta_n)$ as elements in $\Z_p[\chi][\GG_n]$. For every Dirichlet character $\psi$ of conductor $n$, then $\chi\times \psi$ is a primitive character of $\Gal(K(n)/\Q)$ and we have that 
$$e_{\chi\times \psi} \Upsilon_{K(n)}=0.$$

Hence
\begin{equation}
\chi\left(\Theta_{K(n)}-\frac{1}{n} (1-p^{-1} a_p \Frob_p^{-1}+p^{-1}\textbf{1}_{N}(p)  \Frob_p^{-2}) \theta_{K(n)}\right)=\sum_{\ell|n} \nu_{n,n/\ell}\ \alpha_{n,\ell},
\label{eq:psi}
\end{equation}
where $\nu_{n,n\ell}:\ \Q_p[\zeta_{cn}][\GG_{n/\ell}]\to \Q_p[\zeta_{cn}][\GG_n]$ is the norm map and $\alpha_{n,\ell}$ is certain element in $ \Q_p[\zeta_{cn}][\GG_{n\ell}]$, which we do not need to determine explicitly.

For every character $\psi'$ of $\Gal(\Q(n)/\Q)$ (not necessarily primitive), $\chi\times \psi'$ has conductor $m$ satisfying that $r(m,cn)=cn$. Therefore, Proposition \ref{prop:integrality_L} implies that $\chi(\Theta_{K(n)})\in p^{k_\chi'} \Z_p[\GG_n]$. Hence $\alpha_{n,\ell}\in p^{k_\chi'}\Z_p[\zeta_{cn}][\GG_{n/\ell}]$ for every prime divisor $\ell$ of $n$.

That is enough to relate the Kolyvagin derivatives of the (primitive) character parts of $\Theta_{K(n)}$ and $\theta_{K(n)}$.

\begin{proposition}
Let $\chi$ be a primitive character of $\Gal(K/\Q)$ and let $n\in \NN(\PP_k)$. Then the Kolyvagin derivative $D_n$ (see \cite[\textsection IV.4]{Rubin}) can be computed as
$$D_n\chi(\Theta_{K(n)})\equiv\frac{1}{n} \left(1-p^{-1} a_p \chibar(p) +p^{-1} \textbf{1}_{N}(p)\chibar(p)^2\right)D_n\chi(\theta_{K(n)}) \mod p^{k+k_\chi'} \Z_p[\chi] [\GG_n].$$

\label{prop:kol_der_comp}
\end{proposition}

\begin{proof}
Note that, since $\Gal(K(n)/\Q)$ is abelian,

\[\begin{aligned}
&D_n\left(\chi\left((1-p^{-1} a_p \Frob_p^{-1}+p^{-1} \textbf{1}_{N}(p)\Frob_p^{-2}) \theta_{K(n)}\right)\right)=\\
&\chi\left(1-p^{-1} a_p \Frob_p^{-1}-p^{-1} \Frob_p^{-2}\right)D_n(\chi(\theta_{K(n)}))=\\
&(1-p^{-1} a_p \chibar(p)+p^{-1}\textbf{1}_{N}(p) \chibar(p)^2)D_n(\chi(\theta_{K(n)})).
\end{aligned}\]
By equation \eqref{eq:psi}, it is enough to prove that, for every prime divisor $\ell$ of $n$,
$$D_n\, \nu_{n,n/\ell} \alpha\in p^{k+k_\chi'} \Z_p[\psi] [\GG_n]$$
for every $\alpha \in  p^{k_\chi'}\Z_p[\psi] [\GG_{n/\ell}]$. In fact, since $ \nu_{n,n/\ell} \alpha$ is $\GG_\ell$ invariant, then 
$$D_\ell \nu_{n,n/\ell} \alpha=\frac{p^{n_\ell}(p^{n_\ell}-1)}{2}  \nu_{n,n/\ell} \alpha \in p^{k+k'_\chi} \Z_p[\psi] [\GG_n],$$
since $\ell\in \PP_k$. Thus
\[D_n \nu_{n,n/\ell} \alpha=D_{n/\ell} D_\ell \nu_{n,n/\ell} \alpha\in p^{k+k'_\chi} \Z_p[\psi] [\GG_n].\qedhere\]
\end{proof}

By Proposition \cite[Proposition 4.4.2]{Rubin}, for every $n\in \NN(\PP_k)$, $D_n z_{K\Q(n)}$ is invariant under the action of $\GG_n$ modulo $p^k$. Consequently, $D_n \Theta_{K\Q(n)}$ and, therefore, $D_n\theta_{K\Q(n)}$ are $\GG_n$-invariants modulo $p^k$. This is equivalent to
$$D_n\Theta_{K(n)}(\zeta_{K(n)}),\ D_n\theta_{K(n)}(\zeta_{K(n)})\in K\otimes \Q_p.$$

In order to compute this value, Proposition \ref{prop:logs_formula} below is very useful. Before stating it, we need to define a $p$-primary logarithm in $(\Z/\ell)^\times$.

\begin{definition}
Assume $H$ is a finite cyclic group whose order is exactly divisible by $p^k$ for some $k\in \N$. If $x$ is a generator of the $p$-primary part $H$, then for every $a\in H$ we will define $\log_x(a)$ to be the unique element in $y\in \Z/p^k$ such that $a^{-1}x^y$ has order prime to $p$.
\label{def:logarithm}
\end{definition}

\begin{remark}
Given two generators $x_1$ and $x_2$ of the $p$-primary part of $H$, the logarithms $\log_{x_1}(a)$ and $\log_{x_2}(a)$ have the same $p$-adic valuation.
\label{rem:logs_val}
\end{remark}

The following computation is done in \cite[Lemma 3.11]{Sakamoto21}

\begin{proposition}(\cite[Lemma 3.11]{Sakamoto21}) 
Let $R$ be a ring, let $n=\ell_1\cdots\ell_s\in\NN(\PP_k)$ and let 
$$\theta=\sum_{\sigma\in \GG_n} a_\sigma\sigma \in R[\GG_n]$$
be an element such that $D_n\theta$ is Galois invariant modulo $p^k$. Then 
$$D_n\theta\equiv\sum_{\sigma \in \GG_n} a_\sigma \prod_{\ell\mid n}\log_{\tau_\ell}(\sigma) N_n\mod p^k R[\GG_n],$$
where $N_n=\sum_{\sigma\in \GG_n} \sigma$ is the norm element and $\tau_\ell$ is the generator of $\GG_\ell$ used to define the Kolyvagin derivative.
\label{prop:logs_formula}
\end{proposition}

Proposition \ref{prop:logs_formula} motivates Definition \ref{def:kurihara_numbers} of the (twisted) Kurihara numbers.

\begin{remark}
The definition of $\delta_{n,\chi}$ depends on the choices of the generators $\eta_\ell\in (\Z/\ell)^\times$ for every $\ell\mid n$. However, by Remark \ref{rem:logs_val}, the $p$-adic valuation of $\delta_{n,\chi}$ is well defined independently of the choices made in the construction of $\delta_{n,\chi}$.

\end{remark}

From propositions \ref{prop:kol_der_comp} and \ref{prop:logs_formula}, we obtain the following.
\begin{corollary}
For every primitive character $\chi$ of $\Gal(K/\Q)$ and every $n\in \NN(\PP_k)$,
\[D_n e_\chi(\Theta_{cn})\equiv \frac{\varphi(n)}{n}(1-p^{-1} a_p \chibar(p) +p^{-1}\textbf{1}_N(p) \chibar(p)^2) \delta_{n,\chibar}e_{\chi\times\textbf{1}_n}\]
modulo $\frac{p^{k+k_\chibar'}}{\varphi(c)}\Z_p[\chi][G_{\Q(\mu_c)/\Q}\times\GG_n]$.
\label{cor:Theta_delta_cn}
\end{corollary}

\begin{proof}
Since $D_n\chi(\theta_{cn})\equiv  \delta_{n,\chibar} N_n=\delta_{n,\chibar} \varphi(n) e_{\textbf{1}_n}\mod p^k \Z_p[\GG_n]$ by Proposition \ref{prop:logs_formula}, the congruence holds modulo $p^{k+k'_\chibar}\varphi(c)^{-1}$ after multiplying both sides by the Euler product $\left(1-p^{-1} a_p \chibar(p) +\textbf{1}_{N}(p)p^{-1} \chibar(p)^2\right)$ and by the idempotent element. Thus, the corollary follows from Proposition \ref{prop:kol_der_comp}.
\end{proof}

We can adapt Corollary \ref{cor:Theta_delta_cn} to describe the $\chi$-part of $\Theta_{K(n)}$ in terms of the Kurihara numbers.

\begin{corollary}
For every primitive character $\chi$ of $\Gal(K/\Q)$, we have that 
\[D_n e_\chi(\Theta_{K(n)})=\frac{\varphi(n)}{n} (1-p^{-1} a_p\chibar(p)+p^{-1}\textbf{1}_{N}(p)\chibar(p)^2)\delta_{n,\chibar} e_{\chi\times \textbf{1}_n}\]
modulo $p^{k+k_\chibar'}\Z_p[\chi][G\times \GG_n]$.
\label{cor:Theta_delta_Kn}
\end{corollary}

\begin{proof}
It follows from Corollary \ref{cor:Theta_delta_cn}, projecting the congruence to $\Z_p[\chi][G\times G_n]$. 

Since both sides of that equation are invariant under the action of the Galois group $\Gal(\Q(\mu_c)/K)$, whose order is $\varphi(c)/d$, we can remove the denominator $\varphi(c)$ from the ideal of the congruence.
\end{proof}

By equation \eqref{eq:twisted_omega}, we obtain the following corollary.

\begin{corollary}
For every primitive character $\chi$ of $\Gal(K/\Q)$ and every $n\in \NN(\PP_k)$, modulo ${p^{k+k'_\chi}}\Z_p[G\times\GG_n](\xi_{K(n)})$, we have that
\[D_n (w_{\Q(n),\chi})\equiv  \frac{(-1)^{\nu(n)}\chi(n) }{n }(1-p^{-1} a_p \chi(p) +p^{-1}\textbf{1}_{N}(p) \chi(p)^2) \delta_{n,\chi}\varphi(c)e_\chibar(\zeta_c).\]
\label{cor:omega_delta_}
\end{corollary}

\begin{proof}
By the transitivity of the trace,
\[\Tr_{K(n)/K}(\xi_{K(n)})=\Tr_{\Q(\mu_{nc})/K}\left(\sum_{n\widetilde{c}\mid d\mid nc} \zeta_d\right)=(-1)^{\nu(n)}\ \Tr_{\Q(\mu_c)/K}\left(\sum_{\widetilde c|d|c} \zeta_d^{(n^{-1})}\right).\]
Denote
\[\widetilde \xi_K=\Tr_{\Q(\mu_c)/K}\left(\sum_{\widetilde c|d|c} \zeta_d^{(n^{-1})}\right).\]

Since $\chi$ has conductor $c$, by Corollary \ref{cor:chi_xi}
\[e_{\chi\times \textbf{1}_n}(\xi_{K(n)})=\frac{(-1)^{\nu(n)}}{\varphi(n)}e_{\chi}(\widetilde \xi_K)=\frac{(-1)^{\nu(n)}\chibar(n)}{\varphi(n)}[\Q(\mu_c):K]  e_\chi(\zeta_c).\]
Since $[\Q(\mu_c):K]=\frac{\varphi(c)}{d}$, by equation \eqref{eq:twisted_omega} and Corollary \ref{cor:Theta_delta_Kn}, we obtain that
\[D_n (w_{\Q(n),\chi})\equiv  \frac{(-1)^{\nu(n)}\chi(n) }{n}(1-p^{-1} a_p \chi(p) +p^{-1} \chi(p)^2) \delta_{n,\chi}\varphi(c)e_\chibar(\zeta_c).\qedhere\]
\end{proof}

\subsection{Image of Bloch-Kato dual exponential map}
\label{sec:exp_im}


%

In \textsection \ref{sec:exp}, we have introduced the dual exponential map
$$\exp_{\omega_E}^*:\ H^1_{/\FBK}(K_\p,V)\xrightarrow{\sim} K_\p,$$
where $K_\p$ is the completion of $K$ at a prime $\p$ above $p$. However, the group we are interested in is $H^1_{/\FBK}(K_\p,T)$. Hence we want to know the image of the composition map
$$H^1_{/\FBK}(K_\p,T)\to H^1_{/\FBK}(K_\p,V)\to K_\p.$$

In order to compute this image, we generalise the argument in \cite[\textsection 5.2]{Rubin}. By the identifications we have made so far, some $z\in H^1_{/\FBK}(K_\p,T)$ is identified under local Tate duality to the map
$$E(K_\p)\otimes \Q_p\to \Q_p,\ x\mapsto \textrm{Tr}_{K_\p/\Q_p}(\exp_{\omega_E}^*(z)\log_{\omega_E}(x)).$$
Hence an element $y\in K_\p$ belongs to $\exp_{\omega_E}^*(H^1_{\FBK}(K_\p,T))$ if and only if
\begin{equation}
\textrm{Tr}_{K_\p/\Q_p}(y\log_{\omega_E}(x))\in \mathbb Z_p\ \forall x\in E(K_\p),
\label{eq:im_tr}
\end{equation}
where the logarithm is the extension of the one defined in the formal group of the elliptic curve.

Denote by $\OO_\p$ and $\m_\p$ to the ring of integers of $K_\p$ and its maximal ideal, respectively. Denote by $E_1(K_\p)$ the kernel of the reduction map and $E_0(K_\p)$ the points whose reduction is a non-singular point. Since $K_\p/\Q_p$ is unramified by \ref{Kur}, the logarithm induces an isomorphism
$$\log_{\omega_E}: E_1(K_\p)\xrightarrow{\sim} \m_\p.$$
To describe the image $\log_{\omega_E}(E(K_\p))$, we look at the $\chi$-part of this map for the different characters $\chi$ of $\Gal(K_\p/\Q_p)$. By \ref{Kloc}, $E(K_\p)_\chi$ is free of rank one over $\OO[\chi]$.  Since $p$ does not divide the Tamagawa number $c_\p$ at the prime $\p$, by \ref{Ktam}, then the quotient $E(K_\p)/E_0(K_\p)$ has order prime to $p$. Therefore
$$\log_{\omega_E}(E(K_\p))=\log_{\omega_E}(E_0(K_\p)).$$

Consider the exact sequence
$$\xymatrix{0\ar[r] & E_1(K_\p)\ar[r] & E_0(K_\p)\ar[r] &\widetilde E_0(\kappa_\p)\ar[r] & 0,}$$
where $\widetilde E_0(\kappa_\p)$ denote the groups of non-singular points of the reduced curve modulo $\p$. Since $[K_\p:\Q_p]$ is prime to $p$ by \ref{Kdeg}, the sequence remains exact after taking $\chi$-parts. Since $E(K_\p)[p]=0$ by \ref{Kloc}, we have that
$$\log_{\omega_E}(E_0(K_p)_\chi)=\frac{1}{p^{\length(\widetilde E_0(\kappa_\p)[p^\infty]_{\chi})}} \log_{\omega_E}(E_1(K_p)_\chi).$$

Since $K_\p/\Q_p$ is unramified by \ref{Kur}, then $\log_{\omega_E}(E_1(K_\p)_\chi)=(\m_\p)_{\chi}=p(\OO_\p)_\chi$ and
\begin{equation}
\log_{\omega_E}(E(K_\p)_\chi)=\frac{p}{p^{\length(\widetilde E_0(\kappa_\p)[p^\infty]_{\chi})}} \OO_\chi.
\label{eq:log_im}
\end{equation}

The trace map satisfies, for every $x\in K_\p$, the identity $\textrm{Tr}(x)=\textrm{Tr}(e_1 x)$, where $e_1\in \mathbb Z_p[\Gal(K_\p/\Q_p)]$ is the idempotent element associated with the trivial character, i.e., $e_1=\frac{1}{[K_\p:\Q_p]}\sum_{\sigma\in \Gal(K_\p/\Q_p)} \sigma$. By \ref{Kdeg}, $[K_\p:\Q_p]$ is prime to $p$. Thus $\textrm{Tr}(x)\in \Z_p$ if and only if $e_1x \in \Z_p$.

Moreover, note that given $x_1,x_2\in K_\p$ such that $\sigma(x_1)=\chi_1(\sigma)x_1$ and $\sigma(x_2)=\chi_2(\sigma) x_2$ for every $\sigma\in \Gal(K_\p/\Q_p)$ and some characters $\chi_1$ and $\chi_2$, then $\sigma(x_1 x_2)=(\chi_1\chi_2)(\sigma)(x_1x_2)$.

Combining equations \eqref{eq:im_tr} and \eqref{eq:log_im}, we obtain
\begin{equation}
\exp^*(H^1_{/\FBK}(\Q_p,T\otimes \OO_d(\chi)))= p^{\length(\widetilde E_0(\kappa_\p)[p^\infty]_{\chibar})-1} \OO_\chibar.
\label{eq:exp_tw_im}
\end{equation}

We will now relate the length of $\widetilde E_0(\kappa_\p)_{\chi}$ to the $p$-adic valuation of the Euler factor at $p$ evaluated at $s=1$. 

\begin{proposition}
 The length of $e_\chi\left(\widetilde E_0(\OO_d/\m_p)[p^\infty]\right)$ as an $\OO_d$-module is one unit larger than the valuation of the twisted Euler factor $1-\frac{a_p}{p}\chi(p)+\textbf{1}_N(p)\frac{1}{p}\chi(p)^2$ at $p$ evaluated at $s=1$.
\end{proposition}

\begin{proof}
    Recall the definition of $k_\chi'=v_p\left(1-\frac{a_p}{p}\chi(p)+\textbf{1}_N(\ell)\frac{1}{p}\chi(p)^2\right)$ in Definition \ref{def:kurihara_numbers}. We will consider different cases depending of the type of reduction of $E$ at $p$.

Assume first that $E$ has good ordinary reduction at $p$ and let $\alpha\in \Z_p^\times$ be the unit root of the Euler polynomial $X^2-a_p X+p$. Then the arithmetic Frobenius acts on the reduced $p$-primary torsion $\widetilde E[p^\infty]$ by multiplication by $\alpha$ and thus acts on $\widetilde E[p^\infty]\otimes O(\overline\chi)$ multiplying by $\alpha \overline\chi(p)$. $\widetilde E(\kappa_\p)_\chi$ is the kernel of the action of $\Frob_p-1$ on $\widetilde E[p^\infty]\otimes O(\overline\chi)$, so its length is the $p$-adic valuation of $(\alpha \overline\chi(p)-1)$.

Thus we just need to compute the $p$-adic valuation of $\alpha \overline\chi(p)-1$. The twisted Euler polynomial factors as
$$\chi(p)^2X^2-a_p \chi(p)X+p=\chi(p)^2(\overline\chi(p)\alpha-X)(\overline\chi(p)\beta-X),$$
where $\beta\in p\Z_p$ is the other root. Since $\overline\chi(p)\beta-1$ is a unit, evaluating at $X=1$ we get 
$$(\chi(p)^2 -\chi(p) a_p+p)\sim (\overline\chi(p)\alpha-1),$$
where $\sim$ denotes equality up to multiplication by a $p$-adic unit.

If $E$ has good supersingular reduction at $p$, then clearly $\widetilde E(\kappa_\p)[p^\infty]=\{O\}$. In the supersingular case, then $p\nmid N$ and $p\mid a_p$, so the $k_\chi'$ has $p$-adic valuation $-1$ when evaluated at $X=1$. Since $E(\kappa_\p)_\chi=\{0\}$, then the proposition holds true in this case.

If $E$ has multiplicative reduction, then $\widetilde E_0(\kappa_\p)\cong \kappa_\p^\times$ has order prime to $p$, so the $p$-adic valuation of the right hand side is $-1$. In this case, $a_p=\pm 1$ depending on the reduction being split or not and $\textbf{1}_N(\ell)=0$, so $k_\chi'$ has $p$-adic valuation $-1$ as well.

If $E$ has additive reduction, then 
\[ \widetilde E_0(\kappa_\p)_\chi\cong(\kappa_\p)_\chi=\# (\OO_\p/p\OO_p)_\chi\]
has length one, so the equality is also satisfied in this case.
\end{proof}

By \ref{Kloc}, $H^2(\Q_p,T\otimes \OO_d(\chi))=H^0(\Q_p,E[p^\infty]\otimes \OO_d(\chibar))=0$, so there is an isomorphism
$$ H^1_{/\FBK}(\Q_p,T\otimes \OO_d(\chi))/p^k \cong H^1_{/\FBK}(\Q_p,T/p^kT\otimes \OO_d(\chi)).$$
The dual exponential map induces thus an isomorphism
\begin{equation}
H^1_{\FBK}(\Q_p,T/p^kT\otimes \OO_d(\chi))\cong\frac{\exp_{\omega_E}^*(H^1_{/\FBK}(\Q_p,T\otimes \OO_d(\chi)))}{p^k \exp_{\omega_E}^*(H^1_{/\FBK}(\Q_p,T\otimes \OO_d(\chi)))}=\frac{p^{k'_\chibar}(\OO_p)_\chibar}{p^{k+k'_\chibar}(\OO_p)_\chibar},
\label{eq:exp_tw}
\end{equation}
where $\OO_p:=\bigoplus_{\p\mid p} \OO_\p\subset K\otimes \Q_p$.

\subsection{Kato's Kolyvagin system}
\label{sec:Kato_kol}

In \textsection \ref{sec:MT}, we have computed the value of the dual exponential map
\[\exp_{\omega_E,\chibar}^*(D_n z_{\Q(n),\chi})\equiv \frac{(-1)^{\nu(n)}\chi(n)}{n}(1-p^{-1} a_p \chi(p) +p^{-1} \chi(p)^2) \delta_{n,\chi} \varphi(c)e_\chibar(\zeta_c)\]
modulo \(p^{k+k'_\chi}\Z_p[\GG_n](\xi_{K(n)})\). Since \(p^{k+k'_\chi}\Z_p[\GG_n](\xi_{K(n)})\cap K\) is contained in \(p^{k+k'_\psi} \OO_p\), the congruence also holds modulo the latter.

The goal of this section is to relate this value with Kato's Kolyvagin system. After that, we will check that the $p$-divisibility of Kato's Kolyvagin system and the Kurihara numbers coincide.

The Kolyvagin derivative \(\kappa_{n,\chi}\) is the preimage of \(D_n z_{\Q(n),\chi}\) under the restriction map
\[\res_{\Q(n)/\Q}:\ H^1(\Q,T/p^k T\otimes \OO_d(\chi))\to H^1(\Q(n),T/p^k T\otimes \OO_d(\chi))^{\GG_n}.\]

If we understand $\loc_p^\s(\kappa_{n,\chi})$ and $\loc_v^\s(D_n\omega_{\Q(n),\chi})$ as maps in the duals of $H^1_\FBK(\Q_p,T^*\otimes \OO_d(\overline\chi))$ and $H^1_\FBK(\Q_p,T^*\otimes \OO_d(\overline\chi))$, we obtain by \cite[Proposition 1.5.3(iv)]{NSW} that 
$$\loc_v^\s(D_n w_{n,\chi})=\loc_p^\s(\kappa_{n,\chi})\circ \cor_{\Q(n)/\Q}.$$

Similarly to the argument in Proposition \ref{prop:exp_cor}, we can study the behaviour of the dual exponential map under restriction maps.

\begin{proposition}
Let $c\in H^1_{/\FBK}(\Q,V\otimes \OO_d(\chi))$ and denote its restriction by $d=\res_{\Q(n)/\Q}(c)\in H^1_{/\FBK}(\Q(n),V\otimes \OO_d(\chi))$. Then
\[\exp_{\omega_E,\chibar}^*(c)=\exp_{\omega_E,\chibar}^*(d)\in (\Q(n)\otimes \Q_p\otimes \OO_d(\chi))^{\GG_n}.\]
\end{proposition}

\begin{proof}
Localising at primes above $p$, we get that 
\[\loc_p^\s(d)=\left(\bigoplus_{v\mid p} \res_{\Q(n)_v/\Q_p}\right)\left(\loc_p^\s(c)\right).\]

By \cite[Proposition 1.5.3 (iv)]{NSW}, if we consider the Pontryagin duals of the localisation maps $\loc_p^\s(c)\du$ and $\loc_p^\s(d)\du$, then
\[\loc_p^\s(d)\du=\loc_p^\s(c)\du   \circ\left(\bigoplus_{v\mid p} \cor_{\Q(n)_v/\Q_p}\right).\]

By \cite[Proposition 1.5.2]{NSW},
\[\exp_{\omega_E,\chibar}^*(d)=\exp_{\omega_E,\chibar}^*(c)\circ N_{\Q(n)/\Q}.\]

By the identification in \eqref{eq:dual_trace},
\[\exp_{\omega_E,\chibar}^*(c)=\exp_{\omega_E,\chibar}^*(d).\qedhere\]
\end{proof}

Under the isomorphism in \eqref{eq:exp_tw}, the weak Kato's Kolyvagin system is

\[\widetilde \kappa_{n,\chi}=\frac{(-1)^{\nu(n)}\chi(n)}{n}(1-p^{-1} a_p \chi(p) +p^{-1} \chi(p)^2) \delta_{n,\chi} \varphi(c)e_\chibar(\zeta_c).\]

In order to obtain Kato's Kolyvagin system $\kappa_\chi$ from $\widetilde \kappa_\chi$, we need to apply the modification from \cite[Appendix A]{MazurRubin}. However, in this case, $\kappa_\chi=\widetilde \kappa_\chi$. In fact, by the definition of Kolyvagin primes, $a_\ell=2\mod p^k$ for every $\ell\mid n$. Therefore, for every $\ell$, we have that 
\[P_\ell=X^2-a_\ell X+\ell\equiv (X-1)^2\mod p^k.\] 

Therefore, for every $\ell\mid n$, $\rho_\ell(P_\ell(\Frob_{\pi(\ell)}^{-1}))=0$. Hence, in the formula in \cite[(33)]{MazurRubin}, the only term that does not vanish is the one associated with the trivial permutation, so $\widetilde \kappa_{n,\chi}=\kappa_{n,\chi}$.


Since $p$ is unramified in $K/\Q$ by \ref{Kur}, then $\varphi(c)e_\chibar(\zeta_c)$ generates the module $(\OO_{p})_\chibar$, and taking into account the description of the image of the dual exponential map in \eqref{eq:exp_tw}, we get that 
\begin{equation}
\ord(\loc_s^p(\kappa_{n,\chi}))=\ord(\delta_{n,\chi}).
\label{eq:orders}
\end{equation}

\subsection{Primitivity of Kato's Euler system and proof of Theorem \ref{th:EK_IMC}}
\label{sec:proof_IMC}

In this section we will prove Theorem \ref{th:EK_IMC}. From the $\chi$-twisted Kato's Euler system constructed in \textsection \ref{sec:twist}, one can apply the Kolyvagin derivative process as in Theorem \ref{th:eul_to_kol_lambda} to obtain a Kolyvagin system $\kappa^\infty_\chi\in \overline{\KS}(T\otimes \Lambda\otimes \OO_d(\chi),\FLambda, \PP)$.

By \ref{Kloc} and \ref{Ktam}, the assumptions in Proposition \ref{prop:lambda_kol_free} are satisfied, so $\overline{\KS}(T\otimes \Lambda\otimes \OO_d(\chi),\FLambda, \PP)$ is free of rank one over $\Lambda$ and, by Theorem \ref{th:MC_ind}, $\kappa^\infty_\chi$ is primitive if and only if the Iwasawa main conjecture \ref{conj:IMC} holds true.

Kolyvagin derivative process can be applied to obtain the Kolyvagin system $\kappa_\chi \in \textrm{KS}(T\otimes \OO_d(\chi),\Fcan,\PP)$ which was studied in \textsection \ref{sec:Kato_kol}. By Proposition \ref{prop:lambda_kol_free}, it will be the image of $\kappa^\infty_\chi$ under the canonical map
\begin{equation}
\overline{\textrm{KS}}(T\otimes \OO_d(\chi)\otimes \Lambda,\FLambda,\PP)\to \textrm{KS}(T\otimes \OO_d(\chi),\Fcan,\PP).
\label{eq:EK:red}
\end{equation}
By Corollary \ref{cor:kol_red_prim}, $\kappa_\chi$ is primitive if and only if $\kappa^\infty_\chi$ is primitive. By Lemma \ref{lem:dinf} and equation \eqref{eq:orders}, this is equivalent to $\partial^{(\infty)}(\delta_\chi)=0$, which is M. Kurihara's conjecture \ref{conj:kur}. 

Hence the Iwasawa main conjecture and the Kurihara conjecture are equivalent for the twist $T\otimes \OO_d(\chi)$, so Theorem \ref{th:EK_IMC} holds.

Nevertheless, a non-primitive Kolyvagin system is useful for determining the structure of the Selmer group as long as it is non-zero. If we assume hypothesis \ref{KIMCloc}, then $\kappa_\chi^\infty\notin (X\Lambda)\overline{\KS}(T\otimes \Lambda\otimes \OO_d(\chi),\FLambda,\PP)$ by Corollary \ref{cor:kol_red_loc}. But this is the kernel of the map in \eqref{eq:EK:red}, so this condition means that $\kappa_\chi$ is nonzero, fact that will be necessary to apply theorems \ref{th:kur_par} and \ref{th:kur} in the next section.

\subsection{Functional equation and proof of Theorem \ref{th:EK_str}}
\label{sec:proof_str}

The second part of Theorem \ref{th:EK_str} is a direct consequence of Theorem \ref{th:kur}. By equation \eqref{eq:orders}, the ideals $\Theta_{i}(\widetilde\delta_\chi)$ in Theorem \ref{th:EK_str} coincide with the ideals $\Theta_{i}(\kappa_\chi)$ in Theorem \ref{th:kur}. By Remark \ref{rem:EK_str}, we can assume $\chi$ is primitive. 

Theorem \ref{th:kur} describes the structure of the Selmer group in a different way depending on whether the character $\chi$ is a quadratic character or not.

Assume first that $\chi\neq\chibar$. In this case, Theorem \ref{th:kur} implies that the ideals $\Theta_{i}(\widetilde \delta_\chi)$ 
\[\Theta_{i}(\widetilde\delta_\chi)=\Fitt_{i}^{\OO_d}(H^1_\FBK(\Q,T\otimes \OO_d(\chi))).\]

Since $\OO_d$ is a principal ideal domain, the structure theorem of finitely generated modules implies that 
\[H^1_{\FBK}(\Q,T\otimes \OO_d
(\chi))\cong \OO_d^r\oplus \bigoplus_{i=1}^{{s-r}} \frac{\OO_d}{(p)^{\partial^{(r+i)}(\widetilde \delta)-\partial^{(r+i-1)}(\widetilde \delta)}},\]
where we are using the notation of Theorem \ref{th:EK_str}. Therefore, the proof of the second part is complete.

Now consider the case when $\chi=\chibar$, the module $T\otimes \OO_d(\chi)$ is its own Cartier dual, so Theorem \ref{th:kur} does not apply. We have to use \ref{th:kur_par} instead, which leads to a weaker control on the Fitting ideals of the Selmer group. However, Kurihara numbers satisfy a functional equation (inherited from the Mazur-Tate elements) and this fact leads to the determination of the structure of the Selmer group.

Recall that $N$ is the conductor of the elliptic curve. The Mazur-Tate element satisfy the following functional equation
\begin{proposition} (\cite[\textsection 1.6]{MazurTate})
Let $\iota:\ \mathbb Z_p[\Gal(\Q(\mu_n)/\Q)]\to \mathbb Z_p[\Gal(\Q(\mu_n)/\Q)]$ be the map induced from the inversion in $\Gal(\Q(\mu_n)/\Q)$. Then 
\[\iota(\theta_n)=\varepsilon \sigma_{-N} \theta_n,\]
where $\varepsilon\in\{\pm1\}$ is the root number of the elliptic curve and $\sigma_{-N}\in \Gal(\Q(\mu_n/\Q))$ is the Galois automorphism that sends $\zeta_n$ to $\zeta_n^{-N}$.
\label{eq:MT_fe}
\end{proposition}

\begin{corollary}
If $\psi$ is a character of $G$, then 
\[\iota\left(\psibar(\theta_{K(n)})\right)=\varepsilon \chi(-N) \psi(\theta_{K(n)})\in \Z_p[\Gal(\Q(\mu_n)/\Q)].\]
\end{corollary}

Assume $\ell\in \PP_k$. Then the construction of the Kolyvagin derivative implies for every $\theta\in \Z_p[\Gal(\Q(\mu_n)/\Q)]$ that 
$$D_\ell \iota(\theta)\equiv-D_\ell(\theta) \mod p^k.$$

Therefore, for every $n\in \NN(\PP_k)$, 
$$D_n\psibar(\theta_{K(n)})=(-1)^{\nu(n)} \varepsilon\psi(N) D_n\psi(\theta_{K(n)}).$$

Since $\chi=\chibar$, we have that 
\[\delta_{n,\chi} (1-(-1)^{\nu(n)}\varepsilon\chi(-N))=0.\]

Since $\chi(-N)=\pm1$ , we have two possible cases.
\begin{corollary}
Depending on the root number $\varepsilon$ of the elliptic curve and the value $\chi(-N)$, we have that 
\begin{itemize}
\item If $\varepsilon\chi(-N)=1$, then $\delta_{n,\chi}=0$ when $n$ has an odd number of prime divisors,
\item If $\varepsilon\chi(-N)=-1$, then $\delta_{n,\chi}=0$ when $n$ has an even number of prime divisors.
\end{itemize}
\label{cor:delta_fe}
\end{corollary}

Call $r=\rank_{\OO_d} H^1_\FBK(\Q,T\otimes \OO_d(\chi))$. The first condition of Theorem \ref{th:kur_par} is satisfied in this case, so
$$\Theta_{r}(\widetilde \delta_\chi)= \Fitt_{r}^{\OO_d} H^1_{\FBK}(\Q,T\otimes \OO_d(\chi)).$$

For any index $i\geq r$ having different parity than $r$, $\Theta_{i,\chi}=0$ by Corollary \ref{cor:delta_fe} and the corresponding Fitting ideal is non-zero. Hence the second condition of Theorem \ref{th:kur_par} holds for $i+1$, so 
$$\Theta_{i+1}(\widetilde \delta_\chi)= \Fitt_{i+1}^{\OO_d} \Bigl(H^1_{\FBK}(\Q,T\otimes \OO_d(\chi))\Bigr).$$

By the structure theorem of finitely generated modules over principal ideal domains, we obtain
$$H^1_{\FBK}(\Q,T(\chi))\cong \OO_d^r\oplus \bigoplus_{i=1}^{\frac{s-r}{2}} \frac{\OO_d}{(p)^{\partial^{(r+2i)}(\widetilde \delta)-\partial^{(r+2i-2)}(\widetilde \delta)}},$$
so Theorem \ref{th:EK_str} has been proven.

\subsection{Examples}
\label{sec:examples}

We end this article by showing some examples of computations of the Selmer group of elliptic curves. All the computations are done using Sagemath \cite{sagemath}. For all the elliptic curves $E$ and abelian extensions $K/\Q$ appearing in these examples, we will assume that $\Sha(E/K)$ is finite.

\begin{example}
    Consider the elliptic curve 196\,794cd1 in Cremona's database and the prime $p=5$. Over the abelian extension $K=\Q(\mu_7)$, we can determine the full group structure of the Selmer group. We proceed by studying the twisted Kurihara number for every Dirichlet character of conductor dividing $7$.

    When $\chi$ is the trivial character, we compute $\delta_{1,\chi}=0$, so $\rank(E(\Q))\geq 1$. To obtain further information about the structure of the Selmer group, we need to compute $\Theta_{1,\chi}$. The smallest Kolyvagin prime $\ell=93\,251$ satisfies that $\ord_5(\delta_{\ell,\chi})=2$. It guarantees that $\rank(E(\Q))= 1$ and $\#\Sha(E/\Q)[5^\infty]\mid 25$. If we compute $\ord_5(\delta_{\ell,\chi})$ for the smallest Kolyvagin primes, we will obtain the value $2$ for approximately the $80\%$ of the computations and a higher value in the remaining cases. That would lead us to guess that $\Theta_2=25\Z_5$, which is equivalent to $\#\Sha(E/\Q)[5^\infty]=25$. However, we cannot prove it with only a finite amount of computations.

    But we can use a different method to compute the order of the Tate-Shafarevich group using the Iwasawa main conjecture, which can be numerically verified it using Theorem \ref{th:EK_IMC}. The smallest Kolyvagin primes for $k=1$ are $11$, $31$ and $131$. Their product $n=44\,671$ satisfies that $\delta_{n,\chi}\in \Z_5^\times$, so the Iwasawa main conjecture holds true in this elliptic curve.
    
    Using Sagemath, we can check that the $5$-adic $L$-function can be written as
    \[\char(X_\infty)=\mathcal L_5(E,T)=(5^2+O(5^3))T+O(T^2).\]
    Under our assumptions, Mazur's control theorem works perfectly, so we can conclude that $\Sel(\Q,T_5E)\cong \Z_5\times \Z_5/(5) \times\Z_5/(5)$. Therefore,
    \[\Sha(E/\Q)[5^\infty]=\Sha(E/K)[5^\infty]^{\Gal(K/\Q)}\cong \Z_5/(5)\times \Z_5/(5).\]

    If $\chi$ is the quadratic character of conductor $7$, then $\delta_{1,\chi}=0$, so $\rank(E(K)_{\chi})\geq 1$. In order to determine the full structure of the twisted Selmer group, we need to compute $\Theta_{1,\chi}$. Since $11$ is a Kolyvagin prime and $\ord_5(\delta_{11,\chi})=0$, we deduce that $\Theta_1=\Z_5$. Hence $\rank(E(K)_{\chi})= 1$ and $\Sha(E/K)[5^\infty]_\chi=\{0\}$.

    For all other characters of conductor $7$, we compute $\ord_5(\delta_{1,\chi})=1$, which implies that $\Sha(E/K)_\chi\cong \OO_6/(5)$.

    All the information above is enough to compute the structure of the Selmer group $\Sel(K,T_5E)\otimes \OO_6$ as an $\OO_6[\Gal(K/\Q)]$-module. By Proposition \ref{prop:fitting_integral}, the Fitting ideal are then given by the expressions
    \[\begin{aligned}
    &\Fitt_{\Z_5}^0(\Sel(K,T_5E))=\frac{5}{3} (2\sigma_1-\sigma_2-\sigma_4),\\
    &\Fitt_{\Z_5}^1(\Sel(K,T_5E))={5} \sigma_1+{4}(\sigma_2+\sigma_3+\sigma_4+\sigma_5+\sigma_6),\\
    &\Fitt_{\Z_5}^2(\Sel(K,T_5E))=\frac{5}{3} \sigma_1+\frac{2}{3}(\sigma_2+\sigma_3+\sigma_4+\sigma_5+\sigma_6).\\
    \end{aligned}\]

    By Proposition \ref{prop:fitting_integral}, there is an isomorphism of $\Z_p[\Gal(K/\Q)]$-modules
    \[\Sel(K,T_5E)=\frac{\Z_p[G]}{\left({5}(2\sigma_1-\sigma_2-\sigma_4)\right)}\times \left(\frac{\Z_p[G]}{\left({5} \sigma_1+{2}(\sigma_2+\sigma_3+\sigma_4+\sigma_5+\sigma_6)\right)}\right)^2.\]
\end{example}

\begin{example}
Let $E$ be the elliptic curve 35a1 in Cremona's database and let $p=7$. For a Dirichlet character $\chi$ of conductor $51$ and order $8$ such that $\chi(35)=-1$ and $\chi(37)$ is a primitive $8^{\textrm{th}}$ root of unity $\zeta_{8}$ satisfying that $\zeta_{8}^2+3\zeta_8+1\in 7\Z_7$.

We can compute $\ord_7(\delta_{1,\chi})=2$, so we know that $\rank(E(\Q(\mu_{51})))_\chi=0$ and 
\[\#\Sha(E/\Q(\mu_{51}))[7^\infty]_\chi=\Bigl(\#(\OO_{16}/(7))\Bigr)^2.\]

Note that $\chi$ is not self-dual, so there are no non-degenerate pairings defined on the twisted Selmer group that determine its structure. In this case, the computation of $\Theta_{1,\chi}$ is what determines whether the Tate-Shafarevich group is isomorphic to $\OO_{16}/(7^2)$ or $\OO_{16}/(7)\times \OO_{16}/(7)$.

The smallest Kolyvagin prime is $\ell=2\,801$ and satisfies that $\delta_{\ell,\chi}\in \OO_8^\times$. Hence $\Theta_{1,\chi}=\OO_8$ and 
\[\#\Sha(E/\Q(\mu_{51}))[7^\infty]_\chi\cong\OO_{16}/(7^2).\]
\end{example}

\begin{example}
    Consider the elliptic curve 11a1 in Cremona's database, the abelian extension $K=\Q(\mu_{61})$ and the prime $p=101$. 

    Let $\chi$ be the character of conductor $61$ and order $20$ such that $\chi(2)$ is the unique primitive $20^{\textrm{th}}$ root of unity in $60+101\Z_{101}$. Then
    \[\ord_{101}(\delta_{1,\chi})=\ord_{101}(\delta_{1,\chibar})=1.\]
    Theorem \ref{th:EK_str} then implies that 
    \[\Sel(\Q,T_pE\otimes \Z_{101}(\chi))\cong \Sel(\Q,T_pE\otimes \Z_{101}(\chibar))\cong \Z_{101}/(101).\]

    For the quadratic character $\chi'$ of conductor $61$, we obtain that $\delta_{1,\chi'}=0$, so Theorem \ref{th:EK_str} implies that $\rank_{\Z_p}(\Sel(\Q,T_pE\otimes ))\geq 1$. To determine the full structure of this Selmer group, we need to compute $\Theta_{1,\chi'}$. The smallest Kolyvagin prime is $\ell'=64\,237$ and satisfies that
    $\ord_{101}(\delta_{\ell',\chi})=0$. Hence, $\Theta_{i,\chi}=\Z_p$, so
    \[\Sel(\Q,T_pE\otimes \Z_{101})\cong \Z_{101}.\]

    Finally, consider a character $\chi''$ of conductor $61$ and order $6$. It satisfies that $\chi(2)$ is a $6^{\textrm{th}}$-root of unity. Then $\delta_{1,\chi''}=\delta_{1,\overline{\chi''}}=0$. The smallest Kolyvagin prime for these characters is $\ell''=2\,528\,233$, satisfying that 
    \[\ord_{101}(\delta_{\ell'',\chi''})=\ord_{101}(\delta_{\ell'',\overline{\chi''}})=0.\] 
    Hence,
    \[\Sel\Bigl(\Q,T_pE\otimes \OO_6(\chi'')\Bigr)\cong \Sel\Bigl(\Q,T_pE\otimes \OO_6(\overline{\chi''})\Bigr)\cong \OO_6.\]

    Overall, assuming the finiteness of the Tate-Shafarevich group, we know that $\rank(E(K))=3$. Furthermore, we know how the rank grows along the subextensions of $K/\Q$. Indeed, if $K_2$ and $K_3$ are the subextensions of degrees $2$ and $3$, respectively, we know that $\rank(E(K_2))=1$ and $\rank(E(K_3))=2$.

    Similarly, if $L$ is a subextension of $K/\Q$, then $\Sha(E/L)$ would have order $101^2$ if $[L:\Q]\in 20\Z$ and would be trivial otherwise.

\end{example}

\begin{example}
We can use this method to find Tate-Shafarevich groups divisible by large primes. Consider the elliptic curve 27a1 in the Cremona tables and the prime $p=472\,558\,791\,937$. Let $\chi$ be the Dirichlet character of conductor $89$ satisfying that $\chi(3)$ is the unique primitive $88^{\textrm{th}}$ root of unity in $382\,613\,086\,515+p\Z_p$. Then
\[\ord_p(\delta_{1,\chi})=1.\] 
Therefore, we can conclude that
\[\Sel(\Q,T_p E\otimes \Z_p(\chi))\cong \Z_p/(p).\]
\end{example}

\begingroup

\sloppy
\hbadness=10000
\printbibliography[heading=bibintoc]
\endgroup

\end{document}